\documentclass[12pt]{amsart}
\usepackage{amsmath,amssymb,enumerate,verbatim, amsthm}
\usepackage[normalem]{ulem}
\usepackage{cancel}

\usepackage{srcltx}

\usepackage[usenames,dvipsnames]{xcolor}
\usepackage{graphicx}
\usepackage{multirow}
\usepackage{booktabs}
\usepackage{mathtools,bm}
\usepackage{blkarray}
\usepackage{young}

\usepackage[all]{xy}

\newcommand\cf{\textit{cf.~}}
\newcommand\ie{\textit{i.e.~}}

\newcommand\bigb{\ \big|\ }

\newcommand\disp{\displaystyle}
\newcommand\beq{\begin{equation}}
\newcommand\eeq{\end{equation}}

\newcommand{\ep}{{\epsilon}}
\newcommand\Ep{\mathcal{E}}

\newcommand\al{\alpha}
\newcommand\la{\lambda}

\newcommand{\bbC}{{\mathbb{C}}}

\newcommand{\bbN}{{\mathbb N}}

\newcommand{\bbR}{{\mathbb R}}

\newcommand{\bbZ}{{\mathbb Z}}

\newcommand{\fk}{\mathfrak}
\newcommand{\ovl}{\overline}
\newcommand{\bb}{\mathbb}

\newcommand{\dpfr}{\displaystyle\frac{1}{2}}

\newcommand{\frakg}{{\mathfrak{g}}}
\newcommand{\frakh}{{\mathfrak{h}}}

\newcommand{\frakm}{{\mathfrak{m}}}

\newcommand{\frakk}{{\mathfrak k}}
\newcommand{\frakl}{{\mathfrak{l}}}

\newcommand{\frakso}{{\mathfrak{so}}}

\newcommand{\frakgl}{{\mathfrak{gl}}}

\newcommand{\Hom}{{\mathrm{Hom}}}

\newtheorem{definition}[subsection]{Definition}
\newtheorem{lemma}[subsection]{Lemma}
\newtheorem{conj}[subsection]{Conjecture}
\newtheorem{example}[subsection]{Example}
\newtheorem{prop}[subsection]{Proposition}
\newtheorem{theorem}[subsection]{Theorem}
\newtheorem{cor}[subsection]{Corollary}
\newtheorem{remark}[subsection]{Remark}

\newcommand{\calB}{{\mathcal{B}}}

\newcommand{\calF}{{\mathcal{F}}}

\newcommand{\calI}{{\mathcal{I}}}

\newcommand{\calM}{{\mathcal{M}}}

\newcommand{\calO}{{\mathcal{O}}}

\newcommand{\calU}{{\mathcal{U}}}

\newcommand{\calX}{{\mathcal{X}}}

\newcommand{\calW}{{\mathcal{W}}}

\newcommand\clrr{\color{red} }

\newcommand{\vio}{\textcolor{violet}}

\newcommand{\tu}{\widetilde}
\newcommand{\wti}{\widetilde}

\newcommand{\wht}{\widehat}
\newcommand\ul[1]{\underline{#1}}

\newcommand\ad{\operatorname{ad}}
\newcommand\Ad{\operatorname{Ad}}

\newcommand\Span{\operatorname{Span}}
\newcommand\Ind{\operatorname{Ind}}

\numberwithin{equation}{subsection}
\numberwithin{table}{subsection}

\begin{document}

\title{Representations associated to small nilpotent orbits for real Spin Groups} 
\begin{abstract}
The results in this paper provide a comparison between the
$K$-structure of unipotent representations and regular sections of
bundles on nilpotent orbits. Precisely, let
$\tu{G_0} =\tu{Spin}(a,b)$ with $a+b=2n$, the nonlinear double cover of
$Spin(a,b)$, and let $\tu{K}=Spin(a, \bbC)\times Spin(b, \bbC)$ be the complexification of the maximal
compact subgroup of $\tu{G _0}$. We consider the nilpotent orbit
$\calO_c$ parametrized by 
$[3 \ 2^{2k} \ 1^{2n-4k-3}]$ with $k>0$. We provide a list of
unipotent representations that are genuine, and prove that the list is
complete using the coherent continuation representation. Separately we compute 
$\tu{K}$-spectra of the regular functions on certain real forms
$\calO$ of $\calO_c$ transforming according to appropriate characters
$\psi$ under $C_{\tu{K}}(\calO)$, and then match them with the
$\tu{K}$-types of the genuine  unipotent  representations.  
The results provide instances for the orbit philosophy.
\end{abstract}

\author{Dan Barbasch}
      \address[D. Barbasch]{Department of Mathematics\\
               Cornell University\\Ithaca, NY 14850, U.S.A.}
        \email{barbasch@math.cornell.edu}
\thanks{D. Barbasch was supported by an  NSA grant}

\author{Wan-Yu Tsai}
\address[Wan-Yu Tsai]{Institute of Mathematics, Academia Sinica, 6F, Astronomy-Mathematics Building, No. 1, Sec. 4, Roosevelt Road, Taipei 10617, TAIWAN}
\email{wytsai@math.sinica.edu.tw}

\maketitle
\section{Introduction}\label{ss:0} 

Let $G_0\subset G$ be the real points of a complex linear reductive algebraic
group $G$ with Lie algebra $\fk g_0$ and 
maximal compact subgroup $K_0$. Let $\fk g_0=\fk k_0+\fk s_0$
be the Cartan decomposition, and $\fk g=\fk k+\fk s$ be the
complexification. Let $K$ be the complexification of $K_0.$

  \begin{definition}
Let $\calO:= K\cdot e\subset \fk s$, with $e\in \fk s$ a nilpotent element. We say that an irreducible
admissible representation $\Xi$ of $G_0$ is associated to $\calO,$ if $\calO$
occurs with nonzero multiplicity in the associated cycle in the sense
of \cite{V2}.

An irreducible module $\Xi$ of $G_0$ is called unipotent
associated to a nilpotent orbit $\calO\subset \fk s$ and
infinitesimal character $\la_{\calO}$, if it satisfies
\begin{description}
\item[1] It is associated to $\calO$ and its annihilator
  $Ann_{U(\fk g)}\Xi$ in the universal  enveloping algebra $U (\fk g)$ of $\fk g$ is the unique maximal primitive ideal with
  infinitesimal character $\la_{\calO}$,
\item[2] $\Xi$ is unitary.
\end{description}
Denote by $\calU _{G_0}(\calO,\la_{\calO})$  the set of unipotent
representations of $G_0$ associated to $\calO$ and $\la_{\calO}$.
  \end{definition}

\bigskip
{Let $C_K(\calO):= C_K(e)$ denote the centralizer of $e$ in $K$, and 
let $A_K(\calO):=C_K(\calO)/C_K(\calO)^0$ be the component group.} 
Assume that $G_0$ is connected, and a complex
group viewed as a real Lie group. In this case $G\cong G_0\times G_0,$
and  $K\cong G_0$ as complex groups. Furthermore $\fk s\cong\fk g_0$
as complex vector spaces, and the action of $K$ is the adjoint action.  
In this case it is conjectured that there exists an infinitesimal character
$\la_{\calO}$ such that in addition,
\begin{description}
\item[3] There is a 1-1 correspondence $\psi\in
 \wht{ A_K(\calO)}\longleftrightarrow \Xi(\calO,\psi)\in
  \calU_{G_0}(\calO,\la_{\calO})$ satisfying the additional condition 
$$
\Xi(\calO,\psi)\bigb_{K}\cong R(\calO,\psi),
$$
where
\begin{equation}\label{def-reg-fun}
\begin{aligned}
R(\calO, \psi) &= \Ind_{C_{K}(\calO)} ^{K}  (\psi) \\
&= \{f: K\to V_{\psi}  \mid f(gx) =\psi(x) f(g) \ \forall g\in K, \ x\in C_K (\calO)\}
\end{aligned}
\end{equation}
is the ring of regular functions on $\calO$ transforming according to
$\psi$. Therefore, $R(\calO,\psi)$ carries a $K$-representation.

\end{description}
 
Conjectural parameter $\la_\calO$ satisfying this additional condition
are studied in \cite{B}, along with results establishing the validity of 
this conjecture for large classes of nilpotent orbits in the classical
complex groups. Such parameters $\la_\calO$ are available for the
exceptional groups as well, \cite{B} for $F_4$, and to appear elsewhere
for type $E.$

\bigskip
In this paper we investigate this conjecture for representations which
are attached to \textit{small} orbits
in the real case. 
\begin{definition}
A nilpotent $K$-orbit in $\fk s$ is called small if
\begin{equation}\label{orbit-dim-cond}
\dim \calO \leq \text{rank } (\fk k) + |\Delta ^+(\fk k,\fk t)|,
\end{equation}
where $\fk t\subset \fk k$ is a Cartan subalgebra, and $\Delta^+(\fk
k,\fk t)$ is a positive system. 
\end{definition}
These orbits have the property that there is a chance that the
multiplicity of any $\mu\in \widehat{K}$ be uniformly 
bounded.  The reason is as follows. Let $(\Pi,X)$ be an
admissible representation of $G_0$, and  $\mu$ be the highest weight of a
representation $(\pi,V)\in \wht{K}$ which is dominant for
$\Delta^+(\fk k,\fk t)$. Assume that $\dim\Hom_K[\pi,\Pi]\le C$, and
$\Pi$ has associated variety $\ovl{\calO}$ ({\cf \cite{V2})}. Then 
$$
\dim\{ v\ :\ v\in X \text{ belongs to an isotypic component with
} ||\mu||\le t\}\le Ct^{|\Delta^+(\fk k,\fk t)|+\dim \fk t}.
$$
The dimension of $(\pi,V)$ grows like $t^{|\Delta^+(\fk k,\fk t) |}$, the
number of representations with highest weight $||\mu||\le t$ 
grows like $t^{\dim\fk t},$ and the
multiplicities are assumed uniformly bounded. On the other hand, considerations
involving primitive ideals imply that the dimension of this set grows
like $t^{\dim   G\cdot e/2}$ with $e\in\calO,$ and half the
dimension of (the complex orbit) $G\cdot e$ is  the
dimension of the ($K$-orbit) $K\cdot e\in\fk s.$

In the complex case the appropriate condition for bounded
multiplicities is that the orbit be
\textit{spherical}, \ie the Borel subgroup has an open orbit, and results
of \cite{P} provide a classification. 
   
\bigskip  
In this paper we treat the case of real groups of type D in detail. In these cases,  there is  no difference between small and
spherical, and we do not elaborate on this. 

More precisely, we choose infinitesimal characters correponding to the
complex case for which the orbits are minimal such that a $(\fk g,
K)$-module associated to such an orbit exits. We then compute the genuine representations for 
$$
(\wti{G_0},\wti{K})=\big(\wti{Spin}(a,b), Spin(a,\bbC)\times
  Spin(b,\bbC)\big)
$$ 
satisfying these conditions and compute the $\wti{K}$-spectrum and
compare to the spectrum rational functions on the corresponding orbits. (Here we write $(\tu G_0, \tu K)$ to emphasize that $\tu{G_0}$ is the nonlinear double  cover of $Spin(a,b)$.)

\medskip
Write $2n=a+b$ and 
$$
\begin{aligned}
&a=2p,\qquad &&a=2p+1,\\
&b=2q,\qquad &&b=2q-1,
\end{aligned}
$$
The representations  are associated to real forms of the complex
nilpotent orbit  
$$
\calO_c=[3\ 2^{2k}\ 1^{2n-4k-3}]\qquad k>0.
$$ 
The condition $k>0$ insures that these orbits are not special in the
sense of Lusztig. So there are no  
representations with integral infinitesimal character associated to 
$\calO _c$. 
The infinitesimal character is
\begin{equation}
  \label{eq:inflchar}
  \la=(n-k-2,\dots ,1,0;k+1/2,\dots ,3/2, 1/2),
\end{equation}
same as in \cite{B}. 

\medskip
Here is a summary of the results. 

In Section 2 we list the real forms of the nilpotent orbit and
describe the (component groups) of their centralizers. In Section 3 we
analyze the $K$-structure of certain $R(\calO,\psi).$ In Section 4 we
match them with a set of representations obtained by restriction from
those listed in \cite{LS}. { It is not clear that certain of
  these restrictions are irreducible. An
alternative way to  construct a set of representations with the
required properties is to  apply the derived functors construction to
highest weight modules with the appropriate infinitesimal character
and annihilator. The calculations are in the spirit of
\cite{Kn} and \cite{T}. A comparison of the restrictions with the
alternate construction shows that indeed certain of these restrictions
are reducible. 

 Section 5 contains technicalities
about $Spin$ groups used to prove some of the results. Section 6 computes the
coherent continuation representation and shows that the list of
represenations in Section 4 is complete; these are all the genuine
representations with the given infinitesimal character associated to
real forms of $\calO _c.$ Section 7 provides a construction of the
representations using cohomological indcution from highest weight
modules. 

\medskip
The representations all satisfy conditions (1) and (2) necessary to be
called unipotent
representations. As to condition (3), there is a  significant
difference in the real case; it cannot hold in its stated form. 
This can already be seen for $SL(2,\bb R)$. The spherical principal
series with infinitesimal character zero is unipotent, and its 
associated cycle contains two nilpotent orbits. So its $K$-structure
does not match any $R(\calO,\psi).$  The phenomenon is
analyzed in detail in \cite{V2}. A necessary condition for it to hold
is that $\calO$  have codimension bigger than one in its closure. This
is the case for the orbits studied in this paper. In particular this
condition implies that the associated cycle only contains one orbit. 
Even so, because we are dealing with a
nonlinear cover, the \textit{correct} $\psi$ turn out to be 1-dimensional
characters of $C_{\tu{K}}(e)$ which are not trivial on the connected component. 

\medskip
Some of the results, particularly counting the representations and
restricting from the odd Spin groups to the even ones,  have their
origin  in \cite{Ts}. There are relations to the work in \cite{KO1}
and \cite{KO2} which we intend to pursue in future research.

\medskip
Much of this work was done while the second author visited
Cornell University, and continued later while the first author
visited Academia Sinica in Taiwan. 
We  would like to thank the institutions for their support.\bigskip
\section{Preliminaries}
\subsection{Nilpotent Orbits}\label{ss:1}

{We follow \cite{CM}. Nilpotent orbits in  $\fk{so} (a,b)$ are
  parametrized by orthogonal signed Young diagrams of  
signature $(a,b)$ with numerals. We write a real orbit of the diagram
$[3\ 2^{2k}\ 1^{2n-4k-3}]$ as $[3^{\ep} 2^{2k} 1^{+,c} 1^{-,d}]$  
(possibly with $I, II$), where $3^{\ep}$ denotes the block of size 3
starting with sign $\ep$; $1^{+,c}$ denotes  $c$ blocks of size 1
labeled $+$, and $c$ is omitted when $c=1$; similarly for $1^{-,d}$.} 

{The following eight cases of signed diagrams are treated in this
  paper. 
$$
\begin{aligned}
Case\ 1:\  &a=2p&&=2k+2,&&\calO=[3^+2^{2k}1^-]_{I,II}\quad &&\calO=[3^-2^{2k}1^+]_{I,II}\\     
  &b=2p&&=2k+2, \quad   &\\
Case\ 2:\  &a=2p&&=2k+2+2r_+,&&\calO=[3^-2^{2k}1^{+,2r_++1}]_{I,II} &&\calO=[3^+2^{2k}1^-1^{+,2r_+}] \\     
  &b=2q&&=2k+2, \quad    &&&&\\
Case\ 3:\  &a=2p&&=2k+2,&&\calO=[3^+2^{2k}1^{-,2r_-+1}]_{I,II} &&\calO=[3^-2^{2k}1^+1^{-,2r_-}]\\     
  &b=2q&&=2k+2+2r_-, \quad    &&&&\\
Case\ 4:\  &a=2p+1&&=2k+1,&&\calO=[3^+2^{2k}1^+1^{-,2}]&&\calO=[3^-2^{2k}1^-1^{+,2}]\\
  &b=2q-1&&=2k+1,\quad &&\\
Case\ 5:\  &a=2p+1&&=2k+3+2r_+,&&\calO=[3^+2^{2k}1^{+,2r_+ +1}]&&\\
  &b=2q-1&&=2k+1,\quad &&\\
Case\ 6:\  &a=2p+1&&=2k+1,&&\calO=[3^-2^{2k}1^{-,2r_- +1}]&&\\
  &b=2q-1&&=2k+3+2r_-,\quad &&\\
Case\ 7:\  &a=2p+1&&=2k+1+2r_+,&&\calO=[3^-2^{2k}1^-1^{+,2r_+}] &&\footnotesize (\text{with } r_+\ge 2)  \\
  &b=2q-1&&=2k+3.\quad &&\\ 
Case\ 8:\  &a=2p+1&&=2k+3,&&\calO=[3^+2^{2k}1^+1^{-,2r_-} ] && \footnotesize (\text{with } r_-\ge 2) \\
  &b=2q-1&&=2k+1+2r_-.\quad &&\\
\end{aligned}
$$
As will become apparent at the end, these are the only $\wti K$-orbits
that are associated to genuine representations.  Cases 1 and 4 are
invariant under exchanging $+$ and $-,$ Cases $2,3$, $5,6$ and $7,8$
correspond under exchanging $+$ and $-.$ Nilpotent orbits  $I,II$ in
Cases 1,2,3 are treated the same way. We will omit details for cases
that match under these correspondences.   

The proof of the next Proposition, and details about the nature of the
component groups,  are  in Section \ref{ss:clifford}.

\begin{prop} \label{comp-group-real}\   
 \begin{description}
\item[Case 1] If $\calO=[3^+ 2^{2k}1^-]_{I,II}$ or
  $[3^-2^{2k}1^+]_{I,II},$  then $A_{\wti
    K}(\calO) \cong  \bbZ_2\times \bbZ_2$.
\item[Case 2,3] If $\calO=[3^- 2^{2k}1^{+,2r_+ +1}]_{I,II},$ with
  $r_+>0$, then  $A_{\wti K}(\calO)\cong {\bbZ _2}$.
\item[Case 2,3] If $\calO=[3^+2^{2k}1^-1^{+,2r_+}]$, with $r_+>0$, then $A_{\wti K}(\calO)\cong \bbZ_2$.
\item[Case 4] If $\calO=[3^+2^{2k}1^-1^{+,2}],$ then $A_{\wti
    K}(\calO)\cong { \bbZ_2}$.
\item[Case 5,6] If $\calO=[3^+ 2^{2k}1^+],$ with $r_+=0,$ then $A_{\wti K}(\calO) \cong \bbZ_2$.
\item[Case 5,6] If $\calO=[3^+ 2^{2k}1^{+,2r_+ +1}]$ with $r_+>0$, then
  $A_{\wti K}(\calO)= {1}$.
\item[Case 7,8] If $\calO=[3^- 2^{2k}1^-1^{+,2r_+}],$ with $r_+\ge 2$, then  $A_{\wti K}(\calO)\cong {\bbZ_2 }$.
  \end{description}
The cases are paired according to  the $+$ and $-$ interchanged. 
\end{prop}

\section{Regular Sections} \label{s:regsec}
We compute the centralizers needed for $R(\calO,\psi)$ in $\fk k$ 
and in $\wti K.$  We use the standard roots and basis for
$\fk{so}(a,b).$ The Cartan subalgebra is the fundamental one, a basis is given by
$H(\ep_i)$, and the  root vectors are $X(\pm\ep_i\pm \ep_j),\ X(\pm
\ep_i).$ Realizations in terms of the Clifford algebra, and explicit 
calculations are in Section \ref{ss:clifford}.

\medskip
{Let $\{e,h,f\}$ with $e\in\calO$ be a Lie triple such that
  $h\in\fk k$ and $f\in\fk s$.  We denote by
\begin{itemize}
\item $C_{\fk k} (h)_i$  the $i$-eigenspace of $ad(h)$ in $\fk k$,
\item  $C_{\fk k} (e)_i$ the $i$-eigenspace of $ad(h)$ in the
  centralizer of $e$ in $\fk k$,
\item $C_\fk k (h)^+:= \sum \limits _{i>0} C_\fk k (h) _i$, and $C_\fk k (e)^+:= \sum \limits _{i>0} C_\fk k (e) _i.$  
\end{itemize}
}

\subsubsection*{$\mathbf{\wti{Spin}(2p,2q)}$} These are Cases 1,2,3, so 
{$p=k+1$, $q=k+1+r_-$.} 
The compact Cartan subalgebra has coordinates
$$
(x_1,\dots ,x_{k+1}\bigb y_1,\dots y_{k},y_{k+1},\dots , y_{k+r_-+1})
$$
with Cartan involution 
$$
\theta(x_i)=x_i,\ \theta(y_j)=y_j.
$$

We describe the centralizer for  $[3^+2^{2k}1^{-,2r_-+1}]_I$ in $\fk k$ in
detail. Representatives for $e$ and $h$ are
$$
\begin{aligned}
&e=X(\ep_1-\ep_{p+k+1})+X(\ep_1+\ep_{p+k+1}) +\sum_{2\le i\le {k+1}} X(\ep_i+\ep_{p+i-1})\\
&h=H(2\ep_1)+\sum_{2\le i\le {k+1}} H(\ep_i+\ep_{p+i-1})=(2,\underbrace{1,\dots ,1}_k\bigb \underbrace{1,\dots ,1}_k,0,\dots ,0).
\end{aligned}
$$
Then
\begin{equation}
\label{eq:ch}
\begin{aligned}
C_\fk k(h)_0\cong &\fk{gl}(1)\times \fk{gl}(k)\times \fk{gl}(k)\times \fk{so}(2r_-+2) \\
C_\fk k (h)_1=&\text{Span}\{ X(\ep_1-\ep_i), \ {2\le i\le k+1},\\
&X(\ep_{p+j}\pm \ep_{p+l}),\ 1\le j\le k<l\le q\},\\ 
C_\fk k(h)_2=&\text{Span}\{ X(\ep_i+\ep_j),\ 2\le i< j\le k+1,\\ 
&X(\ep_{p+i}+\ep_{p+j}),\  1\le i<j\le k\},\\
C_\fk k(h)_3=&\text{Span}\{ X(\ep_1+\ep_i),\ 2\le i\le k+1\}. 
\end{aligned}
\end{equation}
Similarly
\begin{equation}
\label{eq:ce}
\begin{aligned}
C_\fk k(e)_0\cong &\fk{gl}(1)\times \fk{gl}(k)\times \fk{so}(2r_-+1) \\
C_\fk k (e)_1=&\text{Span}\{ X(\ep_1-\ep_i)+X(\ep_{p+i-1}\pm\ep_{p+k+1}),\  2\le i\le k+1,\\ 
& {X(\ep_{p+j}\pm \ep_{p+l}), \ 1\le j\le k, \ k+2\le l \le q } \}, \\
C_\fk k(e)_2=&C_\fk k(h)_2,\\
C_\fk k(e)_3=&C_\fk k(h)_3.
\end{aligned}
\end{equation}
The $\fk{gl}(k)\subset C_\fk k(e)_0$ is embedded in $\fk{gl}(k)\times
\fk{gl}(k)\subset C_\fk k(h)_0$ via $x\mapsto (x,-x^t),$ and \newline
$\fk{so}(2r_-+1)\subset \fk{so}(2r_-+2)$ is the standard inclusion.

We denote by $Det^\chi$  a character of
$C_{\fk k}(e)$, a power of the determinant of $\fk{gl}(p-1)=\fk{gl}(k).$ 
\textbf{Asume
$\mathbf{p}$ is even throughout.} This has the effect that for an irreducible
representation, $V^*\cong V,$ and details can easily be filled in for
the other case.  Because we are considering   genuine
representations of the nonlinear double cover, we need to compute
regular functions for $\psi$ which are not trivial on the connected
component of the identity.  So $\psi=Det^\chi$ where $\chi$ is a
half-integer. \textbf{This holds for all cases}.

\subsection{Case 1}   As already noted, $p=k+1, q=k+1.$ 
We treat the orbit $\calO=[3^+2^{2k}1^-]_I$ only. The other orbits in
this Case are related by outer automorphisms as follows.

{
Let $\zeta,\eta$ be the outer automorphisms determined by
\begin{equation}\label{eq:out-auto}
\begin{aligned}
\zeta: & (x_1,\dots,x_p\mid y_1,\dots,y_p) \mapsto (x_1,\dots, x_{p-1}, -x_p\mid y_1,\dots,,y_{p-1}, -y_p),\\
\eta : & (x_1,\dots,x_p\mid y_1,\dots,y_p) \mapsto (y_1,\dots,y_p\mid x_1,\dots,x_p). 
\end{aligned}
\end{equation}
The other three orbits in Case 1 are conjugate to $\calO$ by an outer automorphism and are denoted by $\calO^{\zeta}, \calO^{\eta}, \calO^{\zeta\eta}$.
}

The centralizer $C_\fk k(h)$ is isomorphic to $\fk{gl}(1)\times
\fk{gl}(p-1)\times \fk{gl}(p-1)\times \fk{so}(2)$. 

\medskip
 
 A representation of $\tu{K}$ will be denoted by its highest weight, 
$$
V=V(a_1,\dots ,a_p\mid b_1,\dots ,b_p),\quad a_1\ge
a_2\ge \dots\ge |a_p|, \ b_1\ge b_2\ge \dots\ge |b_p|.
$$
All  $a_i,b_j\in\bbZ$ or $a_i,b_j\in\bbZ +\frac{1}{2}$, but $a_i-b_j$
need not be integers; $V$ is genuine precisely when $a_i-b_j\notin\bb Z.$

\medskip
We will compute 
\begin{equation}
\label{eq:mreal}
\Hom_{C_{\fk k}(e)}[V^*, \chi]=
\Hom_{C_{\fk k}(e)_0}\left [V^*/(C_{\fk k}(e)^+V^*), \chi \right ]:=
\left (V^*/(C_{\fk k}(e)^+V^*\right )^\chi
\end{equation} 
in two steps.  In the first step we define a parabolic
subalgebra $\fk p=\fk m+\fk n$ such that $e\in \fk n$, $C_{\fk
  k}(e)\in \fk p,$ and in addition $\fk n\subset C_{\fk k}(e)^+.$ By
Kostant's theorem $V^*/\big(\fk nV^*\big)$ is known, and the computation of (\ref{eq:mreal})
reduces to a similar computation in $\fk m.$ This  is done in step 2.

\subsubsection{Step 1} 
Let $\xi:=H(\ep_1+\dots +\ep_p) =(1, \dots, 1\mid 0, \dots, 0).$ It
determines a parabolic subalgebra $\fk p=\fk m +\fk n\subset\frakk$} where  
\begin{eqnarray*}
\fk m&=& C_{\frakk}(\xi) \cong \frakgl (p)\times \frakso (2p), \\ 
 \fk n &=& {\text{Span}\{X (\ep_i +\ep _j ),\ 1\le {i \neq j}
   \le p  \}\subset C_\fk k (e)_2+ C_\fk k (e)_3.} 
\end{eqnarray*}
Kostant's theorem on cohomology of finite dimensional representations
implies that  $V^*/(\fk n V^*)$ is the {irreducible $\fk m$-module
  generated by its lowest weight.}  We denote it 
\begin{equation}
  \label{eq:hwt1}
\calW(-a_p,-a_{p-1},\dots ,-a_1\mid b_1,\dots ,b_p).  
\end{equation}
The  assumption $p$ even implies $V^*\cong V.$  Since $C_\fk k (e)_0+C_\fk k (e)_1\subset\fk m$ 
and $C_{\frakk} (e) ^+ \cap \frakm = C_{\frakm}(e) ^+$, it is enough to compute
$$
\big[\calW/ (C_{ \fk k} (e) ^+ \cap \frakm )\calW\big]^{\chi}=\big[\calW/ (C_{\fk
  m}(e)^+ \calW )\big]^{\chi}.
$$

\subsubsection{{Step 2}} Let $\fk q=\fk l+\fk u\subset\fk m$ be the
parabolic subalgebra in $\fk m$ determined by $h$, \ie  
\begin{eqnarray*}
\fk l &\cong&\frakgl (1)\times  \frakgl (p-1)\times \frakgl (p-1)
\times \frakso (2)\\  
 {\fk u }&=&{\text{Span} \{  X(\ep_1- \ep _i),  X(\ep _{p-1+i} \pm
\ep_{2p} ), X(\ep _{p-1+i} +\ep_{p-1+j}),\  2\le{i\neq j} \le  p \}} ,
\end{eqnarray*}
with $C_{\fk m}(h)=\fk l,$ $C_{\fk m}(h)^+=\fk u.$ Then 
$$
\begin{aligned}
&C_{\fk m}(e)_0=\text{Span}\{ H(\ep_i-\ep_{p-1+i}),
X(\ep_i-\ep_j){+}X(-\ep_{p-1+i}+\ep_{p-1+j})\},  \\
&C_{\fk m}(e)^+=\text{Span}\{ X(\ep_1-\ep_i){-}X(\ep_{p-1+i}-\ep_{2p}), 
X(\ep_1-\ep_i){-}X(\ep_{p-1+i}+\ep_{2p})\}.
\end{aligned}
$$
As in the case of $\fk g,$ $C_{\fk m}(e)_0\cong
\frakgl(p-1)$ embeds in $\frakgl(p-1)\times \frakgl(p-1)\subset\frakl$ 
as $x\mapsto (0;x\mid -x^t;0).$ 

\subsubsection{} The module $\calW$ is a quotient of a
(generalized) Verma
module $M(\la)=U(\fk m)\otimes_{U(\ovl{\fk q})} F_\la$ with  $\la$ 
the weight of $\calW$ made dominant for $\ovl{\fk q}:$ 
$$
(-a_1;-a_p,\dots ,-a_2\mid -b_{p-1},\dots ,-b_1;-b_p).
$$
The $;$ denotes the fact that this is a (highest) weight of $\frakl
\cong\frakgl (1)\times 
\frakgl (p-1)\times \frakgl (p-1)\times \fk{so} (2).$ The positive
system for $\triangle ^+(\frakl)$ is the standard one for the Levi component. 
The nilradical decomposes $\fk u={C_{\fk m}(e)^+}  \oplus \fk s$ 
where $\fk s=\text{Span}\{
X(\ep_1-\ep_i),\ 2\le i\le p\}$ is a representation of
$\frakgl(1)\times\frakgl(p-1)\times \fk{so}(2p).$  {The (generalized) 
Bernstein-Gelfand-Gelfand resolution is 
\begin{equation}
\begin{aligned}
0\cdots\longrightarrow \bigoplus _{w\in W^+,\ \ell(w)=k} M(w\cdot\la)
\longrightarrow \cdots \longrightarrow 
\bigoplus _{w\in W^+,\ \ell(w)=1} M(w\cdot \la)
\longrightarrow M(\la)\longrightarrow \calW\longrightarrow 0,
\end{aligned}
\end{equation}
with $w\cdot \la := w(\la+\rho(\frakm))-\rho(\frakm)$, and $w\in W^+$,
the $W(\frakl)$-coset representatives that make $w\cdot \la$
dominant for $\triangle ^+ (\frakl)$.  This is a free $C_{\fk
  m}(e)^+$-resolution so we can compute homology by considering 
\begin{equation}\label{eq:coh1}
\begin{aligned}
0\cdots\longrightarrow \bigoplus _{w\in W^+,\ \ell(w)=k} 
\overline{M(w\cdot \la)} \longrightarrow \cdots 
\longrightarrow \bigoplus _{w\in W^+,\ \ell(w)=1} \ovl{M(w\cdot \la)}
\longrightarrow\ovl{ M( \la)}\longrightarrow 0,
\end{aligned}
\end{equation}
where for an $\frakm$-module X, $\ovl{ X }$ 
denotes $X/ \big(( {C_{\fk m}(e)^+}   ) X\big)$.}

{
As a module for $\frakgl (1)\times \frakgl (p-1)\times \fk{so}(2p),$ 
$\fk s$ has highest weight
$(1;0,\dots ,0,-1\mid 0,\dots ,0).$ Thus $S^m(\fk s)\cong (m;0,\dots
,0,-m\mid 0,\dots ,0;0).$ 

Let $\mu:=(-\al_1;-\al_p,\dots ,-\al_2\mid -\beta_{p-1}.\dots
,-\beta_1;-\beta_p)$ be the highest weight of an $\fk l$-module. 
By the  Littlewood-Richardson rule,

\begin{equation}
\label{eq:LRrule}
\begin{aligned}
S^m(\fk s)\otimes F_\mu=\sum \calW(-\al_1+m;-\al_p-m_p,\dots ,-\al_2-m_2\mid
-\beta_{p-1},\dots ,-\beta_1;-\beta_p).
\end{aligned}
\end{equation}

The sum in (\ref{eq:LRrule}) is taken over the set $$\{  m_i \ |  \
  m_i\geq 0, \  \sum \limits _{i=2} ^p m_i =m, \ m_{i} \leq
  \al_{i-1}-\al_i ,\ 3\leq i \leq p \}$$.} 
\begin{lemma}\label{l:1} $\Hom_{C_{\fk  m}(e)_0}[S^m(\fk s)\otimes F_\mu:\chi]\ne 0$ if and only if
$$
\beta_1\ge \al_2+\chi\ge \beta_2\ge \dots \ge \al_{p-1 }+\chi\ge
 \beta_{p-1}\ge \al_p+\chi.
$$
The  multiplicity is $1.$   
\end{lemma}
\begin{proof}
The multiplicity of $\chi$ is nonzero precisely when 
$$
-\al_i-m_i+\beta_{i-1}=\chi \ \text{ for some } m_i\geq 0, \quad 2\le i\le p.
$$
The condition $0\le m_i\le \al_{i-1}-\al_{i}$ implies $\beta_{i -1}\le \al_{i-1}+\chi$ for
$3\le i\le p.$
\end{proof}
\begin{cor}\label{c:1}
$\Hom_{C_{\fk k}(e)}[V,\chi]\ne 0$ only if
$$
b_1\ge a_2+\chi\ge b_2\ge \dots \ge a_{p-1 }+\chi\ge b_{p-1}\ge a_p+\chi.
$$  
The multiplicity is $\le 1,$ and the action of $\ad h$ is $-2\sum\limits_{1\le i\le p}a_i.$
\end{cor}
\begin{proof}
The first two statements follow from the surjection 
\begin{equation}
  \label{eq:bgg}
\begin{aligned}
&  \ovl{M(\la)} \cong S(\fk s)\otimes_{\bb C} F_\la\longrightarrow
\ovl{\calW}\longrightarrow 0.
\end{aligned}
\end{equation}
The action of $\ad h$ is computed from the module
$\calW(-a_1+m;-b_{p-1},\dots, -b_1\mid-b_{p-1},\dots ,-b_1;-b_p)$ with 
$m=\sum\limits_{2\le i\le p}(a_i-b_{i-1}).$ The value is
$$
2(-a_1+\sum_{2\le i\le p}-a_i+b_{i-1})+2(\sum_{1\le j\le p-1} -b_j)=
-2\sum_{1\le i\le p}a_i.
$$ 
\end{proof}

We will need  Lemma \ref{l:tensor} in order to prove Proposition \ref{p:spectrum}.

\bigskip
Let $\fk g=\fk{gl}(r,\bb C)$, $V(a)=V(a_1,\dots ,a_r)$ an irreducible finite dimensional representation with highest weight $(a_1,\dots ,a_r)$ with $a_i\in\bb N.$ Let $\fk s=\bb C^n$ be the standard representation with basis $e_1,\dots , e_r.$ The Littlewood-Richardson rule implies that
$$
S(\fk s)\otimes V(a)=
\sum V(a_1+m_1,\dots ,a_r+m_r)
$$
with multiplicity 1, sum over
$$
\begin{aligned}
m_2&\le a_1-a_2\\ 
&\dots\\ 
m_r&\le a_{r-1}-a_r 
\end{aligned}
$$
We need explicit information about the highest weights that occur in the decomposition. A typical weight of the tensor product will be denoted by
$ e^{\ul{m'}}\otimes v(\ul{a'})$ with $e^{\ul{m}}:=e_1^{m_1'}\dots e_r^{m_r'}$ and 
$v(\ul{a'})$ an eigenvector of $V$ with weight $(a_1',\dots ,a_r').$ The weight $\ul{m}+\ul{a}=(m_1+a_1,\dots ,m_r+a_r)$ is spanned by monomials of the form
$$
e^{\ul{m'}}\otimes v(\ul{a'})\qquad \ul{m'}+\ul{a'}=\ul{m}+\ul{a}.
$$
Order the monomials by $\ul{m'}$ lexicographically. The \textbf{lowest term} is $e^{\ul{m}}\otimes v(\ul{a}).$
\begin{lemma}\label{l:tensor}
The highest weight corresponding to $V(\ul{m}+\ul{a})$ has \textbf{lowest term} 
$e^{\ul{m}}\otimes v(\ul{a}).$
\end{lemma}
\begin{proof}
Let $e^{\ul{m'}}\otimes v(\ul{a'})$ be a lowest monomial occuring in the highest weight. Since the factor occurs, a lowest term must occur. It must be annihilated by the action of all root vectors $X(\ep_i-\ep_{j})$ with $i<j.$ The formula is
$$
X\cdot e^{\ul{m'}}\otimes v=Xe^{\ul{m'}}\otimes v +e^{\ul{m'}}\otimes Xv.
$$   
The formula is
$$
X(\ep_i-\ep_j)e^{\ul{m'}}=(m'_j)e_1^{m'_1}\dots e_i^{m'_i+1}\dots e_j^{m'_j-1}\dots e_r^{m'_r}.
$$
Thus each monomial is mapped into a term that is strictly higher (possibly zero) plus the term $e^{\ul{m'}}\otimes Xv$ which is lexicographically at the same level. 

Applied to the highest weight expression, this implies that $Xv=0$ for any $X=X(\ep_i-\ep_j).$ Thus the lowest term must be a multiple of the highest weight $v(\ul{a})$, which occurs with multiplicity one. It follows that $\ul{m'}=\ul{m}$ as well.
\end{proof}

\begin{prop}\label{p:spectrum}
$\Hom_{C_\fk k(e)}[V,\chi]\ne 0$ only if
$$
a_1+\chi\ge b_1\ge a_2+\chi\dots \ge a_p+\chi\ge |b_p|.
$$
The multiplicities are $\le 1.$

\begin{proof}
We need to prove three inequalities:
$$
a_1+\chi\ge b_1,\qquad a_p+\chi\ge \pm b_p.
$$  
{The module $M(w\cdot\la)$ decomposes according to the Baker--Campbell--Hausdorff formula
$$
M(w\cdot\la)=S(\fk s)\otimes F_{w\cdot\la}\oplus C_{\fk m}(e)^+U(\fk u)\otimes F_{w\cdot\la},
$$ 
and the  summands are $C_{\fk m}(e)_0$-stable. We say the factors are transverse to each other. 
Suppose $b_1<a_1+\chi,$ and $w$ corresponds to the reflection $s_{\ep_1-\ep_2}.$  
Then the complement of $\ovl{M(w\cdot\la)}$ contains a $C_{\fk m}(e)^+$-component with weight 
$$
(-a_2+1+m';-a_p+m'_2,\dots -a_3+m'_3,-a_1+m'_2\mid -b_{p-1},\dots ,-b_1;-b_p),
$$ satisfying $m'=m_2'+\dots +m_p'$  and 
$-a_1-m'_2=-b_1+\chi$;  similar equalities as in Lemma \ref{l:1} hold for the other coordinates. The assumption $b_1>a_1+\chi$ guarantees that such a weight occurs. The vectors are of the form 
$S\otimes Av_{w\cdot\la}$ with $A$ in the universal enveloping algebra
of the first factor $gl(p-1),$ and $X\in S^{m'}(\fk s)$. The image
under the differential is $ASX(\ep_1-\ep_p)^{a_1-a_2+1}\otimes
v_{\la}.$ This is nonzero and in the transverse to $C_{\fk
  m}(e)^+U(\fk u)\otimes F_{w\cdot\la}.$ The proof uses Lemma \ref{l:tensor}.

For the other two inequality, the analogous argument applies using the fact that the spaces 
$$
\fk s_\pm = \Span\{X(\ep_{p+i}\pm \ep_{2p})\}_{1\le i\le p-1}
$$
are also transverses. }
\end{proof}
\end{prop}

\subsubsection{$\ell(w)=1$} To prove that the weights in
Proposition \ref{p:spectrum} actually occur, it is enough to show that these
weights do not occur in the term in the BGG resolution (\ref{eq:bgg}) 
with $\ell(w)=1.$
Recall 
\begin{equation}
  \label{eq:wone}
\begin{aligned}
\rho &= \rho (\frakm)= (-\frac{(p-1)}{2};\frac{p-1}{2}, \frac{p-3}{2},\dots
,-\frac{(p-3)}{2}\mid -1,\dots ,-(p-1);0),\\
\la+\rho &= (-a_1- \frac{(p-1)}{2}; -a_p+ \frac{(p-1)}{2},\dots ,-a_2
-\frac{(p-3)}{2}\bigb\\
& -b_{p-1}-1,-b_{p-2}-2,\dots ,-b_1-(p-1);-b_p)
\end{aligned}
\end{equation}

For the case $\ell(w)=1,$ there are three elements.
We enumerate them as $w_1, w_2, w_3$, with $w_1 =s_{\ep _1-\ep_p},
w_2= s_{\ep _{p+1} -\ep _{n}}, w_3 =s_{\ep _{p+1} + \ep _{n}}$.  Then 
$$
\begin{aligned}
&w_1\cdot\la=(-a_2+1;-a_p,\dots ,-a_3,-a_1-1\mid -b_{p-1},\dots ,-b_1;-b_p), \\
&w_2\cdot\la= (-a_1;-a_p,\dots ,-a_2\mid -b_p+1,\dots ,-b_1;-b_{p-1}-1),\\
&w_3\cdot\la= (-a_1;-a_p,\dots ,-a_2\mid b_p+1,\dots ,-b_1;-b_{p-1}+1).
\end{aligned}
$$
\begin{lemma}\ 
  \label{l:spectrum}
  $\ovl{M(w_i\cdot \la)}$ has vectors transforming according to $\chi$  of
$C_{\fk m}(e)$ (trivial on $C_\fk k(e)^+$), only if
  
\begin{eqnarray*}
w_1:&-a_1 -1-m_2 =-b_1+\chi\ \text{ for some }  m_2 \ge 0, \  \text{ \ie } &b_1>a_1+\chi\\
&&\\ 
w_2:&-a_p -m_p = -b_p +1 +\chi\ \text{ for some } m_p \ge 0,  \
 \text{ \ie } &b_p>a_p+\chi\\
 &&\\
w_3:& -a_p-m_p =b_p+1 +\chi\ \text{ for some } m_p \ge 0\  \text{ \ie } &-b_p>a_p+\chi.
\end{eqnarray*}

  \medskip
The multiplicities are 1, and the eigenvalue of $\ad h$ is
$-2\sum\limits_{1\le i\le p-1} a_i.$ 
\end{lemma}
\begin{proof}
As in (\ref{eq:LRrule}), the weights in $\ovl{ M (w_1 \cdot \la)}$, $\ovl{ M (w_2 \cdot \la)}$ and $\ovl{ M (w_3 \cdot \la)}$ are of the form 
\begin{eqnarray*}
&(-a_2+1+m;-a_p -m_p,\dots ,-a_3-m_3,-a_1-1 -m_2\mid -b_{p-1},\dots ,-b_1;-b_p),\\
&( -a_1 +m ; -a_p-m_p, \dots, -a_2 -m_2\mid -b_p+1,\dots -b_1 ; -b_{p-1} -1 ),\\
&( -a_1 +m ; -a_p-m_p, \dots, -a_2 -m_2\mid b_p+1,\dots -b_1 ; -b_{p-1} +1 ),
\end{eqnarray*}
respectively.   The proof is completed as in the case $\ell(w)=0.$ 
\end{proof}

\begin{theorem}
\label{t:1}
A representation $V(a_1,\dots ,a_p\mid b_1,\dots ,b_p)$  has vectors
transforming as $\chi$ of $C_{\fk k}(e)$  if and only if 
\begin{equation}\label{eq:halfinteger} 
a_1+\chi\ge b_1\ge\dots \ge a_p+\chi\ge |b_p|,
\end{equation}
and the multiplicity is 1. In summary,
$$
Ind_{C_{\wti K}(e)^0}^{\wti K}(\chi)=\bigoplus V(a_1,\dots,a_p\mid b_1,\dots, b_p), 
$$
satisfying  
$$
a_1+\chi\ge b_1 \ge \dots\ge a_p+\chi \ge |b_p|.
$$
\end{theorem}
\begin{proof}
The proof is straightforward from the BGG resolution (\ref{eq:bgg}), Proposition
\ref{p:spectrum}, and Lemma \ref{l:spectrum}.
\end{proof}

\subsubsection{} \label{ss:cover}{ Theorem \ref{t:1} can be
  interpreted as computing regular functions on the universal cover
  $\wti\calO$  of
  $\calO$ transforming according to $\chi$ under $C_{\fk k}(e)_0$.  
We decompose it further:
\begin{eqnarray}\label{eq:decomp}
R(\tu{\calO}, Det^{\chi}):= \Ind_{C_{\tu{ K} } (e)^0} ^{K} (Det
^{\chi}) =  \Ind _{C_{\tu{K}} (e)} ^{\tu{K}}  \left
  [\Ind_{C_{ \tu{K}} (e)^0} ^{  C_{ \tu{K}} (e)  } (Det ^{\chi})
\right ]. 
\end{eqnarray}
The inner induced module splits into
$$
\Ind_{C_{ \tu{K}} (e)^0} ^{  C_{ \tu{K}} (e)  } (Det ^{\chi})=\sum\psi
$$
where $\psi $ are the irreducible representations of $C_{\wti K}(e)$
restricting to $Det^\chi$ on $C_{\wti K}(e)^0.$
}

\subsubsection{} \label{ss:cover2} We compute $R(\calO,\psi):= \Ind_{C_{\tu{ K} } (e)} ^{\tu K}(\psi)$  for
$\wti{K}$ for $\chi=-1/2$; these are the cases matching representations.

\medskip
{The formula in Theorem \ref{t:1} specializes to  
\begin{equation*}
\begin{aligned}
R(\tu{\calO},Det^{-1/2}) &= \bigoplus V(a_1,\dots,a_p\mid b_1,\dots, b_p)
\end{aligned}
\end{equation*}
satisfying  
$$
a_1\ge b_1+1/2 \ge \dots\ge a_p\ge |b_p|+1/2.
$$
}
{
\begin{lemma}\label{indchar}
Let $\nu_i$, $1\le i \le 4$, 
be the following  $\tu{K}$-types parametrized by their highest weights: 
\begin{eqnarray*}
&\nu _1 = (1/2, \dots, 1/2\mid 0,\dots, 0),  \nu _2 = (3/2, 1/2, \dots, 1/2\mid 0,\dots, 0),\\
&\nu_3 = (1,\dots , 1 \mid1/2, \dots, 1/2), \nu_4 = (1,\dots , 1 \mid1/2, \dots, 1/2 , -1/2).
\end{eqnarray*}
Let  $\psi_i $  be  the restriction of the highest weight of $\nu _i$ to $C_{\tu{K}}(e)$. Then  
\begin{eqnarray*}
\Ind_{C_{\tu{K}} (e)^0} ^{  C_{\tu{K}} (e)  } (Det ^{-1/2})=\sum \limits_{i=1} ^4 \psi _i.
\end{eqnarray*}
\end{lemma}
We write $C_{\tu K}(\calO):= C_{\tu K} (e)$.
\begin{prop} \label{split-even-orbit}
The induced representation (\ref{eq:decomp}) decomposes as 
\begin{eqnarray*}
\Ind_{C_{ \tu{K}} (\calO)^0} ^{\tu{K}} (Det ^{-1/2}) =\sum \limits _{i=1} ^4 R(\calO, \psi_i),
\end{eqnarray*}
where
$$ 
\begin{aligned}
R(\calO, \psi_1) &=& \Ind _{C_{\tu{K} }(\calO) }  ^{\tu{K} } (\psi_1)&=\bigoplus V(\beta_1+1/2,\dots, \beta_p+1/2\mid \delta_1, \dots, \delta_p),  &&  \sum(\beta_i+\delta_j)\in 2\bbZ, \\
R(\calO, \psi_2) &=& \Ind _{C_{\tu{K} }(\calO) }  ^{\tu{K} } (\psi_2)&= \bigoplus V(\beta_1+1/2,\dots, \beta_p+1/2\mid \delta_1, \dots, \delta_p), &&  \sum(\beta_i+\delta_j)\in 2\bbZ +1,
\end{aligned}
$$
satisfying $\beta_1\ge \delta_1 \ge \dots, \ge \beta_p\ge |\delta_p|$ and $\beta_i, \delta_j\in \bbZ$,

$$
\begin{aligned}
R(\calO, \psi_3) &=& \Ind _{C_{\tu{K} }(\calO) }  ^{\tu{K} } (\psi_3)&=\bigoplus V(\beta_1+1/2,\dots, \beta_p+1/2\mid \delta_1, \dots, \delta_p),  &&  \sum(\beta_i+\delta_j)\in 2\bbZ, \\
R(\calO, \psi_4) &=& \Ind _{C_{\tu{K} }(\calO) }  ^{\tu{K} } (\psi_4)&= \bigoplus V(\beta_1+1/2,\dots, \beta_p+1/2\mid \delta_1, \dots, \delta_p),  &&   \sum(\beta_i+\delta_j)\in 2\bbZ +1,
\end{aligned}
$$
satisfying $\beta_1\ge \delta_1 \ge \dots, \ge \beta_p\ge |\delta_p|$ and $\beta_i, \delta_j\in \bbZ +1/2$.

\medskip
{The analogouse results for the other orbits in case 1 follow by
  applying the outer  automorphisms in (\ref{eq:out-auto})}

\end{prop}

\subsection{Case 2, 3} It is enough to consider the three nilpotent orbits, 
$$
\begin{aligned}
&\calO=[3^+2^{2k}1^{-,2r_-+1}]_{I,II}\qquad &&h=H(2,\underbrace{1,\dots ,1,\pm 1}_{k}\bigb \underbrace{1,\dots ,1}_{k},0,\underbrace{0,\dots ,0}_{r_-}),\\
&\calO=[3^-2^{2k}1^+1^{-,2r_-}]\qquad &&h=H(\underbrace{1,\dots ,1}_k,0\bigb 2,\underbrace{1,\dots ,1}_{k},\underbrace{0,\dots ,0}_{r_-}).
\end{aligned}
$$ 
\subsubsection{$\calO=[3^+2^{2k}1^{-,2r_-+1}]_{I,II}$} We assume
$p=k+1, q=k+1+2r_-$ with $r_->0$.
{ Denote $\calO_{I,II} =[3^+2^{2k}1^{-,2r_-+1}]_{I,II} $ according to the
  semisimple elements in the Lie triple 
$h_{I,II}=(2,\underbrace{1,\dots ,1,\pm 1}_{k}\bigb 
\underbrace{1,\dots,1}_{k},0,\underbrace{0,\dots ,0}_{r_-})$. The
orbits are conjugate by the outer automorphim 
$$
\zeta: (x_1,\dots,x_p\mid y_1,\dots, y_q)\mapsto (x_1,\dots,-x_p\mid y_1,\dots, -y_q).
$$
We only treat the
case $\calO = \calO_I$. Similar result holds for $\calO^{\zeta}=\calO_{II}.$
}
\begin{prop}
\label{p:inv1}
A representation $V(a_1,\dots ,a_p\bigb b_1,\dots ,b_q)$ has invariant
vectors under $C_\fk k(e)^+$ which transform according to
$Det^\chi$ under $C_{\fk k}(e)_0\cong \fk{gl}(k)\times \fk{so}(2r_- -1)$ 
if and only if 
$b_{k+2}=\dots =b_q=0$ and 
$$
a_1+\chi\ge b_1\ge a_2+\chi\ge \dots \ge a_p+\chi\ge b_p\ge 0,
$$
{and the multiplicity is 1.} 
\end{prop}
\begin{proof} The representation
  $\calW(a_1,\dots ,a_{k+1}\bigb b_1,\dots b_{k};b_{k+1},\dots
  ,b_q)$ has $\fk{so}(2r_-+1)-$fixed vectors only if $b_{k+2}=\dots =b_q=0$
by Helgason's theorem; in that case, the fixed vector is the highest
weight $(b_{k+1},0,\dots ,0).$ Otherwise the proof is identical to
Case 1.
\end{proof}

\subsubsection{}
{
As in section \ref{ss:cover2}, we compute $R(\calO,\psi)$  for
$\wti{K}$ and $\chi=-1/2-r_-$; these are the cases matching
representations. From Proposition \ref{p:inv1}, 
\begin{equation}\label{cover-4}
R(\tu{\calO_I },Det^{-1/2-r_-}) = \Ind _ {C_{\tu{K}} (e)^0 } ^{\tu{K}}
(Det^{-1/2-r_-}) = \bigoplus V(a_1,\dots,a_p\mid b_1,\dots, b_p,
0,\dots,0)  
\end{equation}
satisfying 
$$
a_1\ge b_1+1/2 +r_-\ge \dots\ge a_p\ge b_p+1/2+r_- \ge 0.
$$


\begin{prop} \label{decom-I}
Let 
$$ 
\begin{aligned}
\psi _1 & =  (r_- + 1/2, \dots, r _- + 1/2) \mid 0,\dots, 0)\mid _{C_{\tu{K}}(e) }, \\
\psi _2 & =  (r_-+3/2, r_- + 1/2, \dots, r_- + 1/2 \mid 0,\dots, 0)\mid _{C_{\tu{K}}(e  ) }, \\
\end{aligned}
$$

The induced representation (\ref{cover-4}) decomposes as 
\begin{eqnarray*}
\Ind_{C_{ \tu{K}} (\calO_I)^0} ^{\tu{K}} (Det ^{-1/2-r_-}) = R(\calO_I, \psi_1 )+R(\calO_I, \psi_2 ),
\end{eqnarray*}
where
$$ 
\begin{aligned}
R(\calO _I, \psi_1) &=\Ind _{C_{\tu{K} }(\calO_I) }  ^{\tu{K} } (\psi_1)&=\bigoplus V(\beta_1+r_-+1/2,\dots, \beta_p+r_-+1/2\mid \delta_1, \dots, \delta_p,0, \dots,0)  \\
&&\text{ with }   \sum(\beta_i+\delta_j)\in 2\bbZ, \\
R(\calO_I, \psi_2 ) &= \Ind _{C_{\tu{K} }(\calO_I) }  ^{\tu{K} } (\psi_2)&= \bigoplus V(\beta_1+r_-+1/2,\dots, \beta_p+r_-+1/2\mid \delta_1, \dots, \delta_p,0,\dots,0)  \\
&& \text{ with }    \sum(\beta_i+\delta_j)\in 2\bbZ +1,
\end{aligned}
$$
satisfying $\beta_1\ge \delta_1 \ge \dots, \ge \beta_p\ge |\delta_p|$ and $\beta_i, \delta_j\in \bbZ$.

\medskip
{The corresponding results for  $R(\calO_{II},\psi ^{\zeta} )$
  follow by applying the automorphism $\zeta$.}


\end{prop}

}

\subsubsection{$\calO=[3^+2^{2k}1^-1^{+,2r_+}]$} We assume $p=k+1+r_+,\
q=k+1$ with $r_+>0$ A representative of the orbit is 
$$
\begin{aligned}
&e=X(\ep_1-\ep_{p+q})+X(\ep_1+\ep_{p+q})+{\sum_{2\le j\le k+1}}  X(\ep_j+\ep_{p+j-1}),\\
&h=(2,\underbrace{1,\dots ,1}_k,\underbrace{0,\dots
  ,0}_{r_+}\bigb \underbrace{1,\dots ,1}_k,0).
\end{aligned}
$$ 
\begin{prop}\label{p:inv5}
A representation $V(a_1,\dots ,a_p\bigb b_1,\dots ,b_q)$ has invariant
vectors under $C_\fk k (e)^+$ transforming according to $Det ^\chi$
under $C_\fk k(e)_0\cong \fk{gl}(k)\times \fk{so}(2r_+)$ if and only if
$a_{k+2}=\dots =a_p=0$ and 
$$
a_1+\chi\ge b_1\ge a_2+\chi\ge b_2\ge \dots \ge a_{q}+\chi\ge |b_{q}|
$$  

\begin{proof} The proof follows Case 1. The fact that $a_{k+2}=\dots
  =a_p=0$ follows from the requirement that the character be trivial
  on the $\fk{so}(2r_+)$-factor.  
\end{proof}
\end{prop}

\subsubsection{}
{
As before, $\chi=r_+-1/2$ is the case corresponding to
representations. Proposition \ref{p:inv5} specializes to
\begin{equation}\label{c-d}
R(\tu{\calO},Det^{r_+-1/2})= \Ind _ {C_{\tu{K}} (e)^0 } ^{\tu{K}}  (Det^{r_+-1/2}) = \bigoplus V(a_1,\dots,a_q,0,\dots,0\mid b_1,\dots, b_p) 
\end{equation}
satisfying 
$$
a_1 +r_+-1/2\ge b_1\ge \dots\ge a_q+r_+-1/2\ge |b_q|.
$$

\begin{prop} 
Let 
$$
\begin{aligned}
\phi_1= (0, \dots, 0 \mid r_+-1/2,\dots, r_+-1/2) \mid _{C_{\tu{K}}(e)},\\
\phi_2= (0, \dots, 0\mid r_+-1/2,\dots,-(r_+-1/2))\mid_ {C_{\tu{K}}(e)}
\end{aligned}
$$
Then the induced representation (\ref{c-d}) decomposes as 
\begin{eqnarray*}
\Ind_{C_{ \tu{K}} (\calO)^0} ^{\tu{K}} (Det ^{r_+-1/2}) = R(\calO, \phi_1)+R(\calO, \phi_2),
\end{eqnarray*}
where
$$ 
\begin{aligned}
R(\calO, \phi_1) &=\Ind _{C_{\tu{K} }(\calO) }  ^{\tu{K} } (\phi_1)&=\bigoplus V(\beta_1,\dots, \beta_q,0,\dots, 0\mid \delta_1, \dots, \delta_q)  
&\text{ with }   \sum(\beta_i+\delta_j)\in 2\bbZ, \\
R(\calO, \phi_2) &= \Ind _{C_{\tu{K} }(\calO) }  ^{\tu{K} } (\phi_2)&= \bigoplus V(\beta_1,\dots, \beta_q,0,\dots, 0\mid \delta_1, \dots, \delta_q)
&\text{ with }    \sum(\beta_i+\delta_j)\in 2\bbZ +1,
\end{aligned}
$$
satisfying 
$
\beta_1+r_+-1/2\ge \delta_1 \ge \dots, \ge \beta_q+r_+-1/2\ge
|\delta_q| \text{ and } \beta_i \in \bbZ, \ \delta_j\in \bbZ+1/2.
$
\end{prop}
}

\bigskip
\subsubsection*{$\mathbf{\wti{Spin}(2p+1,2q-1)}$} These are Cases 4--8. 
{ The two orbits in Case 4 are obtained from Case 7 and Case 8
  by putting  $r_+=1$ and $r_-=1$, and they are 
related by the automorphisms in (\ref{eq:out-auto}).  
So we deal with Cases 5, 6 and Cases 7, 8.}
\subsection{Case 5, 6} The orbit is $\calO=[3^+2^{2k}1^{+,2r_++1}]$ and
  $2p+1=2k+3+2r_+,\ 2q-1=2k+1.$ 
{
The fundamental Cartan subalgebra has coordinates 
$$
(x_1,\dots ,x_{k+1},x_{k+2},\dots,x_{k+1+r_+} \bigb y_1,\dots y_{k} ; z)
$$
with Cartan involution 
$$
\theta(x_i)=x_i,\ \theta(y_j)=y_j,\ \theta(z)=-z.
$$

Representatives for $e$ and $h$ are
$$
\begin{aligned}
&e=X(\ep_1)_n+\sum_{2\le i\le k+1} X(\ep_i+\ep_{p+i-1})\\
&h=H(2\ep_1)+\sum_{2\le i\le k+1}
H(\ep_i+\ep_{p+i-1})=(2,\underbrace{1,\dots
  ,1}_{k},\underbrace{0,\dots ,0}_{r_+} \bigb \underbrace{1,\dots ,1}_k;0)
\end{aligned}
$$
where the last coordinate after the $";"$ is the $z.$ Then
\begin{equation}
\label{eq:ch2}
\begin{aligned}
C_\fk k(h)_0\cong &  \fk{gl}(1)\times {\fk{gl}(k)\times {\fk{so}(2r_++1)}\times \fk{gl}(k) }, \\
C_\fk k (h)_1=&\text{Span}\{ X(\ep_1-\ep_i), X(\ep_i)_c,\ X(\ep_{p+i-1})_c,\ 2\le i\le k+1,\\
&X(\ep_{j}\pm \ep_{l}),\ 1\le j\le k+1<l \le p\}, \\
C_\fk k(h)_2=&\text{Span}\{ X(\ep_1)_c,\  X(\ep_i+\ep_j), \ X(\ep_{p+i-1} +\ep_{p+j-1}) , \ 2\le i< j\le k+1\},\\ 
C_\fk k(h)_3=&\text{Span}\{ X(\ep_1+\ep_i),\ 2\le i\le k+1\}. 
\end{aligned}
\end{equation}
Similarly
\begin{equation}
\label{eq:ce2}
\begin{aligned}
C_\fk k(e)_0\cong &\fk{gl}(1)\times \fk{gl}(k)\times {\fk{so}(2r_++1)} \\
C_\fk k (e)_1=&\text{Span}\{
X(\ep_1-\ep_i)+X(\ep_{p+i-1})_c,\ 2\le i\le k+1,\\
&X(\ep_{i}\pm\ep_{l}), 2\le i\le k+1< l \le p,\  
X(\ep_{p-1+j})_c \ 2\le i\le k+1,\ 1\le j\le k\},\\ 
C_\fk k(e)_2=&C_\fk k(h)_2\\
C_\fk k(e)_3=&C_\fk k(h)_3.
\end{aligned}
\end{equation}
The $\fk{gl}(k)$ embeds in $\fk{gl}(k)\times\fk{gl}(k)$ as before
$x\mapsto (x,-x^t).$
\begin{prop}
\label{p:inv4}
A representation $V(a_1,\dots ,a_p\bigb b_1,\dots ,b_{q-1})$ has
vectors invariant for $C_{\fk k}(e)^+$ which transform according to
$Det^{\chi}$ under 
$C_\fk k(e)_0\cong \fk{gl}(1)\times\fk{gl}(k)\times \fk{so}(2r_++1)$ if and only if 
$a_{k+2}=\dots =a_{p}=0$, and 
\begin{equation}  \label{eq:morbit4}
a_1+\chi\ge b_1\ge \dots \ge a_{k}+\chi\ge b_{k}\ge a_{k+1}+\chi. 
\end{equation}
\end{prop}
}
\begin{proof}
\noindent\textit{Step 1.\ } Let $\fk p=\fk m+\fk n\subset \fk k$ be the
parabolic subalgebra determined by  
$$
\xi=(\underbrace{1,\dots ,1}_{k+1},\underbrace{0,\dots ,0}_{r_+}\bigb
\underbrace{0,\dots ,0}_k;0). 
$$ 
Then $\fk n\subset C_{\fk k}(e)^+,$ and we can apply Kostant's theorem
to reduce the computation to \newline 
$\fk m\cong \fk{gl}(k+1)\times \fk{so}(2r_++1)\times\fk{gl}(k).$ The
$\fk n$-coinvariants are the module  
$$
\calW(-a_1,\dots
,-a_{k+1};a_{k+2},\dots a_p\bigb
b_1,\dots, b_k).
$$ We have assumed $k$ even for simplicity. By
Helgason's theorem, $a_{k+2}=\dots =a_p=0.$ We need to compute the
multiplicity of a character $\chi$ trivial on the nilradical of $C_\fk m(e).$ 

\medskip
\noindent\textit{Step 2.\ }  Let $\fk q=\fk l+\fk u\subset \fk m$ be the
parabolic subalgebra determined by
the (restriction of) $h:$ 
$$
\fk l\cong \fk{gl}(1)\times \fk{gl}(k)\times \fk{gl}(k)\times \fk{so}(2k+1).
$$
The proof proceeds as in Case 1 and Cases 2,3 ; see also Cases 7, 8.
\end{proof}
When $r_+>0,$ $a_i\in\bb Z$ and $b_j\in\bb Z+1/2.$ When $r_+=0,$
$A_{\wti K}(\calO)$ has two components. The character $\chi$ is not
determined by its differntial; there are two possibilities,
corresponding to 
$a_i\in \bb Z, b_j\in \bb Z+1/2$ and $a_i\in\bb Z+1/2,\ b_j\in \bb Z.$
\subsection{Case 7, 8}
The fundamental Cartan subalgebra has coordinates 
$$
(x_1,\dots ,x_{k+1}\bigb y_1,\dots y_{k},y_{k+1},\dots , y_{k+r_-},z)
$$
with Cartan involution 
$$
\theta(x_i)=x_i,\ \theta(y_j)=y_j,\ \theta(z)=-z.
$$
We describe  the centralizer for
$[3^+2^{2k}1^{+}1^{-, 2r_- }]$ with $r_-\ge 2$ in $\fk k$ in detail. 
Representatives for $e$ and $h$ from the previous section are
$$
\begin{aligned}
&e=X(\ep_1)+\sum_{2\le i\le k+1} X(\ep_i+\ep_{p+i-1})\\
&h=H(2\ep_1)+\sum_{2\le i\le k+1} H(\ep_i+\ep_{p+i-1})=(2,\underbrace{1,\dots ,1}_k\bigb \underbrace{1,\dots ,1}_k,\underbrace{0,\dots ,0}_{r_-};0)
\end{aligned}
$$
where the last coordinate after the $";"$ is the $z.$ 
Then
\begin{equation}
\label{eq:ch1}
\begin{aligned}
C_\fk k(h)_0\cong &\fk{gl}(1)\times \fk{gl}(k)\times \fk{gl}(k)\times \fk{so}(2r_-+1), \\
C_\fk k (h)_1=&\text{Span}\{ X(\ep_1-\ep_i), X(\ep_i),\ X(\ep_{p+i-1}),\ 2\le i\le k+1,\\
&X(\ep_{p+j}\pm \ep_{p+l}),\ 1\le j\le k<l \le q\}, \\
C_\fk k(h)_2=&\text{Span}\{ X(\ep_1),\  X(\ep_i+\ep_j), \ X(\ep_{p+i-1} +\ep_{p+j-1}) , \ 2\le i< j\le k+1\},\\ 
C_\fk k(h)_3=&\text{Span}\{ X(\ep_1+\ep_i),\ 2\le i\le k+1\}. 
\end{aligned}
\end{equation}
Similarly
\begin{equation}
\label{eq:ce1}
\begin{aligned}
C_\fk k(e)_0\cong &\fk{gl}(1)\times \fk{gl}(k)\times \fk{so}(2r_-) \\
C_\fk k (e)_1=&\text{Span}\{
X(\ep_1-\ep_i)+X(\ep_{p+i-1}\pm\ep_{p+k+1}),\ 2\le i\le k+1,\\
&X(\ep_{p+j}\pm\ep_{p+k+l}), 1\le j\le k,\  {k+2\le l \le q-1} \},\\ 
C_\fk k(e)_2=&C_\fk k(h)_2\\
C_\fk k(e)_3=&C_\fk k(h)_3.
\end{aligned}
\end{equation}
The $\fk{gl}(k)\subset C_\fk k(e)_0$ is embedded in $\fk{gl}(k)\times
\fk{gl}(k)\subset C_\fk k(h)_0$ as before, $x\mapsto (x,-x^t)$, and
$\fk{so}(2r_-)\subset \fk{so}(2r_-+1)$ in the standard way. 
\begin{prop}
\label{p:inv2}
A representation $V(a_1,\dots ,a_p\bigb b_1,\dots ,b_{q-1})$ has
vectors invariant under $C_{\fk k}(e)^+$ and transforming according to
$Det^{\chi}$ under $C_\fk k(e)_0\cong \fk{gl}(k)\times \fk{so}(2r_-)$ if and only if 
$b_{k+2}=\dots =b_{q-1}=0$, and 
\begin{equation}
  \label{eq:morbit}
a_1+\chi\ge b_1\ge a_2+\chi\ge \dots \ge b_k\ge a_{k+1}+\chi\ge b_{k+1}.  
\end{equation}


\end{prop}
\begin{proof}
The proof is essentially the same as for the other cases.

\noindent\textit{Step 1.\ } Let $\fk p=\fk m+\fk n\subset \fk k$ be the
parabolic subalgebra determined by  
$$
\xi=(\underbrace{1,\dots ,1}_{k+1}\bigb 0,\dots ,0;0).
$$ 
Then $\fk n\subset C_{\fk k}(e)^+,$ and we can apply Kostant's theorem
to reduce the computation to \newline $\fk m\cong \fk{gl}(k+1)\times \fk{so}(2k
+1+2r_-).$ Let $\calW(a_1,\dots , a_{k+1}\bigb b_1,\dots
b_k,b_{k+1},\dots , b_{k+r_-})$
($q-1=k+r_-$) be an irreducible representation of $\fk m$ parametrized by its highest
weight, and $\chi$ be a character of $C_{\fk m}(e)$ trivial on the
nilradical.  We will  compute 
$$
\big[\calW(a_1,\dots a_{k+1}\bigb b_1,\dots ,b_k,b_{k+1},\dots ,b_{k+r_-})\big]^\chi.
$$

\noindent\textit{Step 2.\ } Let $\fk q=\fk l+\fk u\subset \fk m$ be the
parabolic subalgebra determined by
the (restriction of) $h:$ 
$$
\fk l\cong \fk{gl}(1)\times \fk{gl}(k)\times \fk{gl}(k)\times\fk{so}(2r_-+1).
$$
Then $C_{\fk m}(e)_0\cong \fk{gl}(k)\times \fk{so}(2r_-),$ $C_\fk m(e)_2=C_\fk m(h)_2$ and 
$C_{\fk m}(e)_3=C_\fk m(h)_3.$  

$C_\fk m(e)_1\subset C_\fk m(h)_1$  
has complements $\fk s_0,\fk s_\pm$ spanned by 
$$
\begin{aligned}
&\fk s=\text{Span}\{X(\ep_1-\ep_i),\ 2\le i\le k+1\},\\
&\fk s_-=\text{Span}\{X(\ep_{p+i}-\ep_{p+q-1}),\ 1\le i\le k\}\\
&\fk s_+=\text{Span}\{X(\ep_{p+i}+\ep_{p+q-1}),\ 1\le i\le k\}.
\end{aligned}
$$ 
Then $S^m(\fk s)=V(m;0,\dots ,0,-m\bigb \underbrace{0,\dots
  0}_{k};0,\underbrace{0,\dots ,0}_{r_- -1})$ as before.
{The (generalized) 
Bernstein-Gelfand-Gelfand resolution, using ${\ovl{\fk q}}$, is 
\begin{equation}
\begin{aligned}
0\cdots\longrightarrow \bigoplus _{w\in W^+,\ \ell(w)=k} M(w\cdot\la)
\longrightarrow \cdots \longrightarrow 
\bigoplus _{w\in W^+,\ \ell(w)=1} M(w\cdot \la)
\longrightarrow M(\la)\longrightarrow \calW\longrightarrow 0,
\end{aligned}
\end{equation}
with $w\cdot \la := w(\la+\rho(\frakm))-\rho(\frakm)$, and $w\in W^+$,
the $W(\frakl)$-coset representatives that make $w\cdot \la$
dominant for $\triangle ^+ (\frakl)$.  This is a free $C_{\fk
  m}(e)^+$-resolution so we can compute cohomology by considering 
\begin{equation}\label{eq:coh2}
\begin{aligned}
0\cdots\longrightarrow \bigoplus _{w\in W^+,\ \ell(w)=k} 
\overline{M(w\cdot \la)} \longrightarrow \cdots 
\longrightarrow \bigoplus _{w\in W^+,\ \ell(w)=1} \ovl{M(w\cdot \la)}
\longrightarrow\ovl{ M( \la)}\longrightarrow 0,
\end{aligned}
\end{equation}
where for an $\frakm$-module X, $\ovl{ X }$ 
denotes $X/ \big(( {C_{\fk m}(e)^+}   ) X\big)$.}
The weight $\la$ is 
$$
(-a_1,-a_{k+1},\dots ,-a_k \bigb -b_{k},\dots ,-b_1,b_{k+1},\dots b_{q-1}).
$$
The fact that $b_{k+2}=\dots =b_{q-1}=0$ follows from Helgason's
theorem for the pair $\fk{so}(2r_-)\subset \fk{so}(2r_-+1)$. The fixed vector is
the highest weight.

\end{proof}

{
\subsubsection{} The $\chi$ relevant to matching with representations
are
$$
\begin{aligned}
&Case\ 4,\qquad &&\chi=-1/2,\\
&{ Case\ 5,6  \text{ with } r_+=0,} \qquad &&{\chi=1/2},\\  
&Case\ 5,6 \text{ with } r_+> 0,\qquad &&\chi=r_+ + 1/2\\  
&Case\ 7,8,\qquad &&\chi=-r_- +1/2. 
\end{aligned}
$$
\begin{prop}\ The $\wti K-$ structure in Cases 4-8 is as follows.
\begin{description}
\item[Case 4] $\calO=[3^+2^{2k}1^-1^{+,2}]$. In this case
  $A_{\wti K}(\calO)\cong \bb Z_2.$ 
   Let 
$$
\begin{aligned}
&\psi_1=( 1/2,\dots,1/2\mid 0,\dots, 0)|_{C_{\tu{K}}(e) },\\
&\psi_2= (3/2, 1/2,\dots,1/2\mid 0,\dots,0)|_{C_{\tu{K}}(e) }.  
\end{aligned}
$$ 

  Then 

$$
R(\tu{\calO}, Det ^{1/2}) = R(\calO,\psi_1)+R(\calO,\psi_2),
$$

with
$$ 
\begin{aligned}
R(\calO, \psi_1) &= \Ind _{C_{\tu{K} }(\calO) }  ^{\tu{K} }
(\psi_1)&=\bigoplus V(   \beta_1+1/2,\dots, \beta_p+1/2\mid \delta_1,
\dots, \delta_p)\\ && \text{ with } \sum(\beta_i+\delta_j)\in 2\bbZ,
\\ 
R(\calO, \psi_2) &= \Ind _{C_{\tu{K} }(\calO) }  ^{\tu{K} } (\psi_2)&=
\bigoplus V(   \beta_1+1/2,\dots, \beta_p+1/2\mid \delta_1, \dots,
\delta_p)\\ && \text{ with }  \sum(\beta_i+\delta_j)\in 2\bbZ +1, 
\end{aligned}
$$
satisfying $\beta_1\ge \delta_1 \ge \dots, \ge \beta_p\ge \delta_p\ge 0$ and $\beta_i, \delta_j\in \bbZ$.

\medskip
\noindent {The automorphism $\eta$ in (\ref{eq:out-auto}) relates the
  result for $R(\calO,\psi)$, with $\calO^{\eta} =[3^-2^{2k}
  1^+1^{-,2}]$ and the corresponding $\psi^\eta$.
}
\medskip
{
\item[Case 5, 6] $\calO=[3^+2^{2k}1^+].$ In this case 
$A_{\wti K}(\calO)\cong \bb Z_2.$ 
{
 Let 
 \begin{equation}
   \label{eq:case561} 
\begin{aligned}
&\psi_1=( 0,\dots, 0\mid 1/2,\dots,1/2 )|_{C_{\tu{K}}(e) },\\
&\psi_2= (1/2,\dots,1/2 \mid 1,\dots,1)|_{C_{\tu{K}}(e) }.  
\end{aligned}  
 \end{equation}
}

Then  

$$
R(\tu{\calO} ,Det^{1/2})  = R(\calO,\psi_1)+R(\calO,\psi_2)
$$
with
$$
\begin{aligned}
R(\calO, \psi_1) &=& \Ind _{C_{\tu{K} }(\calO) }  ^{\tu{K} }
(\psi_1)&=\bigoplus V(   \beta _1 ,\dots, \beta_p \bigb \delta_1,
\dots, \delta_{p-1})  &&\beta_i\in \bbZ,\ \delta_j\in \bbZ +1/2,
\\ 
R(\calO, \psi_2) &=& \Ind _{C_{\tu{K} }(\calO) }  ^{\tu{K} } (\psi_2)&=
\bigoplus V(   \beta _1 ,\dots, \beta_p \bigb \delta_1,
\dots, \delta_{p-1})  &&\beta_i\in \bbZ+1/2,\ \delta_j\in \bbZ, 
\end{aligned}
$$
satisfying $\beta_1+1/2\ge \delta_1 \ge \dots, \ge \beta_{p-1} +1/2 \ge \delta_{p-1}\ge
\beta_{p}+1/2.$
}

\medskip
\item[Case 5, 6] $\calO=[3^+2^{2k}1^{+,2r_+ +1}]$ with $r_+>0.$ In this
  case $A_{\wti K}(\calO)\cong 1.$ Then

$$
\begin{aligned}
R(\tu{\calO} ,Det^{r_++1/2}) &=R(\calO ,Det^{r_++ 1/2}) \\
&= \bigoplus V(  \beta_1,\dots,\beta_q, 0, \dots , 0 \mid \delta_1, \dots, \delta_{q-1} )  
\end{aligned}
$$
satisfying $\beta_1 +r_++1/2 \ge \delta_1 \ge \dots, \ge \beta_{q-1} +1/2\ge \delta_{q-1}\ge \beta _q $ and $\beta_i\in \bbZ, \delta_j\in \bbZ+1/2$.

\medskip
\item[Case 7, 8] $\calO=[3^+2^{2k}1^+1^{-,2r_-}]$. {Let 
$$
\begin{aligned}
&\psi_1=( r_- - 1/2,\dots,r_- -1/2\mid 0,\dots, 0)|_{C_{\tu{K}}(e) },\\
&\psi_2= (r_- +1/2, r_- -1/2, \dots,r_- -1/2\mid 0,\dots,0)|_{C_{\tu{K}}(e) }.  
\end{aligned}
$$ 
}

 In this case
  $A_{\wti K}(\calO)\cong \bb Z_2.$ Then
  
  $$
R(\tu{\calO} ,Det^{-r_- +1/2})  = R(\calO,\psi_1)+R(\calO,\psi_2),
$$

with
$$ 
\begin{aligned}
R(\calO, \psi_1) &= \Ind _{C_{\tu{K} }(\calO) }  ^{\tu{K} }
(\psi_1)&=\bigoplus V(  \beta_1+r_--1/2,\dots, \beta_p +r_- -1/2 \mid \delta_1, \dots, \delta_p, 0,\dots, 0 )   \\ &&\text{ with }   \sum(\beta_i+\delta_j)\in 2\bbZ,
\\ 
R(\calO, \psi_2) &= \Ind _{C_{\tu{K} }(\calO) }  ^{\tu{K} } (\psi_2)&=
\bigoplus V(  \beta_1+r_--1/2,\dots, \beta_p +r_- -1/2 \mid \delta_1, \dots, \delta_p, 0,\dots, 0 )\\ && \text{ with }    \sum(\beta_i+\delta_j)\in 2\bbZ +1, 
\end{aligned}
$$
satisfying $\beta_1\ge \delta_1 \ge \dots \ge \beta_p\ge \delta_p\ge 0$ and $\beta_i, \delta_j\in \bbZ$.
  
\end{description}
\end{prop}

\begin{proof} {The calculations of $R(\wti\calO,Det^\chi)$ are
  essentially the same for all Cases 4, 5,6 and 7,8. The calculations
  of the $R(\calO,\psi)$ are different. In Cases 4 and 7,8 the
  disconnectedness of the centralizer is already present for
  $K=SO(2p+1,2q-1).$ Precisely, $A_{\wti K}(\calO)=A_K(\calO)$ is
  nontrivial. In Cases 5,6 with $r_-=0$, $A_{\wti K}(\calO)\ne
  A_{K}(\calO).$ Finally in Cases 5,6 with $r_->0$,  $A_{\wti
    K}(\calO)=1$, and there is nothing further to prove}.

\medskip
For $R(\wti\calO,Det^\chi)$ in Cases 5,6, we give details for  
$\calO=[3^+2^{2k}1^+].$ Then 
$$
\begin{aligned}
&e=X(\ep_1)_n +\sum_{2\le j\le k+1}X(\ep_j+\ep_{p+j-1}),\\
&h=(2,\underbrace{1,\dots ,1}_k\bigb \underbrace{1,\dots ,1}_ k {; 0}) .  
\end{aligned}
$$
{ The centralizers $C_\fk k(h)$ and $C_\fk k(e)$ are as in
  (\ref{eq:ch1})  and (\ref{eq:ce1}) with $r_-=0$}.

Let $V(a_1,\dots ,a_{k+1}\bigb b_1,\dots ,b_k)$ be a $\wti
K$-type. Then Steps 1 and 2 imply that $V$ has a $C_\fk k(e)^+$-fixed
vector transforming according to $Det^\chi$ if and only if 
$$
a_1\ge b_1+\chi\ge \dots b_k+\chi\ge a_{k+1}.
$$
The genuine $\wti K$-types must satisfy $a_i\in \bb Z, b_j\in\bb Z+\frac12$ or
$a_i\in \bb Z+\frac12, b_j\in\bb Z$ and $\chi$ a half-integer. The case $\chi=-1/2$ is relevant to the representations.

\bigskip
{The element $(-I,-I)$ acts by $-1$ on the representation. The two elements $(I,-I)$ and $(-I,I)$ therefore act by opposite signs. Then
$$
R(\wti \calO,Det^\chi)=R(\calO,\psi_1)+R(\calO,\psi_2),
$$
{ where 
$$
\begin{aligned}
\psi_1 =(0,\dots,0\mid 1/2,\dots,1/2)|_{C_{\tu{K}}(e) },\\
\psi_2 =(1/2,\dots,1/2\mid 1,\dots,1)|_{C_{\tu{K}}(e) },
\end{aligned}
$$
where $\psi_1, \psi_2$ are again the corresponding $\tu{K}$-types restricting to $C_{\tu{K}}(e)$.
This coincides with the result in the statement if replacing $+$ by $-$. 
}

\medskip
For Cases 4 and 7,8 we give details for $R(\wti\calO,Det^\chi)$ with
$\calO=[3^+2^{2k}1^-1^{+,2}]$ in Case 4.  Other cases are similar.  The component group satisfies $A_{\wti
  K}(\calO)=A_K(\calO)\cong \bb Z_2.$  We
use the realization
\beq
\begin{aligned}
e=&X(\ep_1-\ep_{p+k+1})+ X(\ep_1 +\ep_{p+k+1})+\sum_{2\le i\le k+1}
X(\ep_i+\ep_{p+i-1})\\
h=&2H(\ep_1) +\sum_{2\le i\le k+1}
H(\ep_i+\ep_{p-1+i})=(2,\underbrace{1,\dots ,1}_{k}\bigb
\underbrace{1,\dots ,1}_k,0 { ; 0 } ).  
\end{aligned}
\eeq

\begin{equation}
\label{eq:ch1.3}
\begin{aligned}
C_\fk k(h)_0&\cong \fk{gl}(1)\times \fk{gl}(k)\times \fk{gl}(k)\times \fk{so}(3),\\
C_\fk k (h)_1&=\text{Span}\{ X(\ep_1-\ep_i),\ {X(\ep_{p+i-1} \pm \ep_{p+k+1}  ) }, \ 2\le i\le k+1,\\
&X(\ep_i)_c,\ X(\ep_{p+i-1})_c,\ 2\le i\le k+1\},\\
C_\fk k(h)_2&=\text{Span}\{ X(\ep_1)_c,\ X(\ep_i+\ep_j),\ 2\le i< j\le k+1,\\ 
&X(\ep_{p+i-1} +  \ep_{p+j-1}),\  2\le i<j\le k+1\},\\
C_\fk k(h)_3&=\text{Span}\{ X(\ep_1+\ep_i),\ 2\le i\le k+1\}. 
\end{aligned}
\end{equation}
Similarly
\begin{equation}
\label{eq:ce1.3}
\begin{aligned}
C_\fk k(e)_0&\cong \fk{gl}(1)\times \fk{gl}(k)\times \fk{so}(2),\\
C_\fk k (e)_1&=\text{Span}\{
X(\ep_1-\ep_i)+X(\ep_{p+i-1}\pm\ep_{p+k+1}),\ 2\le i\le k+1,\\
&X(\ep_i)_c,\ X(\ep_{p+i-1})_c, \ 2\le i\le k+1\},\\ 
C_\fk k(e)_2&=C_\fk k(h)_2,\\
C_\fk k(e)_3&=C_\fk k(h)_3.
\end{aligned}
\end{equation}
The $\fk{gl}(k)\subset C_\fk k(e)_0$ is embedded in $\fk{gl}(k)\times
\fk{gl}(k)\subset C_\fk k(h)_0$ as before, $x\mapsto (x,-x^t)$.

\medskip
The element $e^{\pi iH(\ep_1\pm \ep_6)}$ represents the nontrivial
element in the component group. The vector in $V(a_1,\dots
,a_{k+1}\bigb b_1,\dots b_k,b_{k+1})$ which is 
$C_{\fk k}(e)^+$-invariant and transforms according to $Det^\chi$, has weight
$$
(-a_1+ k\chi-\sum_{2\le i\le k+1}a_i+\sum _{1\le j\le  k}b_j,-b_2+\chi,
\dots ,-b_k+\chi \bigb -b_1,\dots ,-b_k,b_{k+1}).
$$ 
The nontrivial element of $A_{\wti K}(\calO)$ acts by 
$$
e^{\pi i(\sum_{1\le i\le k+1}a_i +\sum_{1\le j\le k+1} b_j +k\chi)},
$$
and has different values according to the \textit{parity} of the sum in
the exponent. This accounts for the decomposition
$$
R(\wti\calO,Det ^\chi)=R(\calO,\psi_1)+R(\calO,\psi_2).
$$
}
\end{proof}
}


\section{Representations}
We will obtain representations associated to the various $\calO$ 
by restricting the representations
of $(\fk g'=\fk{so}(p',q'),\wti K' = Spin(p')\times Spin(q'))$ constructed in \cite{LS}. 
They are unitary, associated  to the orbit
$\calO'=[2^{2k+2}\ 1^{2n-4k-3}]$, and have infinitesimal character
$$
\la'=\left(n-k-1-1/2,\dots ,1/2 ;  k+1,\dots ,1\right). 
$$
We recall their $\wti K'$-spectrum  from
\cite{LS}. Let $\tu{G'}=\tu{Spin}(p',q')$
be such that $p'$ is odd and $q'$ even. 
\subsubsection{$p'-1=q'$} There are four representations
$$
\begin{aligned}
  &\left(\la_1,\dots ,\la_{q'/2}\ \big|\  \la_1+1/2,\dots
    ,\la_{q'/2}+1/2\right),\quad &&\la_i\in\bb Z,\\\
  &\left(\la_1,\dots ,\la_{q'/2}\ \big|\  \la_1+1/2,\dots
    ,-\la_{q'/2}-1/2\right),\quad &&\la_i\in\bb Z,\\
  &\left(\la_1,\dots ,\la_{q'/2}\ \big|\  \la_1+1/2,\dots
    ,\la_{q'/2}+1/2\right),\quad &&\la_i\in\bb Z+1/2,\\
  &\left(\la_1,\dots ,\la_{q'/2}\ \big|\  \la_1+1/2,\dots
    ,-\la_{q'/2}-1/2\right),\quad &&\la_i\in\bb Z+1/2.\\
\end{aligned}
$$
\subsubsection{$p'-1>q'$} There are two representations,
$$
\begin{aligned}
  &\left(\la_1,\dots ,\la_{q'/2},0,\dots ,0\bigb \la_1+\frac{p'-q'}{2},\dots ,\la_{q'/2}+\frac{p'-q'}{2}\right),\\
  &\left(\la_1,\dots ,\la_{q'/2},0,\dots ,0\bigb \la_1+\frac{p'-q'}{2},\dots ,-\la_{q'/2}-\frac{p'-q'}{2})\right),\\
\end{aligned}
$$
\subsubsection{$p'-1<q'$} One representation,
$$
\left(\la_1+\frac{q'-p'}{2},\dots ,\la_{(p'-1)/2}+\frac{q'-p'}{2}\bigb
  \la_1,\dots ,\la_{(p'-1)/2},0,\dots ,0\right).
$$ 

\begin{theorem}\ 
  \label{t:main}
The representations attached to $\calO$ have the following $\wti{K}$-structure.

\begin{description}

\item[Case 1. $\tu{G_0}=\tu{Spin}(2p,2p)$, $2p=2k+2$, $(r_+=0)$ ]\

There are eight representations obtained by restriction from $\tu{Spin}(2p+1,2p)$, with $\tu{K}$-structure:
$$
\begin{aligned}
\pi_3: & (\delta_1,\dots,\delta_p\bigb  \beta_ 1+1/2, \dots, \beta_p +1/2) &&\text{ with } \sum(\delta_i+\beta_j)  \in 2\bbZ , \\
\pi_4:   & (\delta_1,\dots,\delta_p\bigb  \beta_ 1+1/2, \dots, \beta_p +1/2) &&\text{ with } \sum (\delta_i+\beta_j) \in 2\bbZ +1,
\end{aligned}
$$
satisfying $\beta_1\ge \delta_1\ge \dots\ge\beta_p\ge |\delta_p|, \beta_i, \delta_j \in \bbZ$;

$$
\begin{aligned}
\tau_3: & (\delta_1,\dots,\delta_p\bigb  \beta_ 1+1/2, \dots,  -\beta_p -1/2) &&\text{ with } \sum( \delta_i+\beta_j )  \in 2\bbZ+1, \\
\tau_4:   & (\delta_1,\dots,\delta_p\bigb  \beta_ 1+1/2, \dots, -\beta_p -1/2) &&\text{ with } \sum(\delta_i+\beta_j ) \in 2\bbZ ,
\end{aligned}
$$
satisfying $\beta_1\ge \delta_1\ge \dots\ge\beta_p\ge |\delta_p|, \beta_i, \delta_j \in \bbZ$;

$$
\begin{aligned}
\sigma _1: & (\delta_1,\dots,\delta_p\bigb  \beta_ 1+1/2, \dots,  \beta_p +1/2)  &&\text{ with } \sum(\delta_i+\beta_j)  \in 2\bbZ, \\
\sigma _2:   & (\delta_1,\dots,\delta_p\bigb  \beta_ 1+1/2, \dots, \beta_p +1/2) &&\text{ with } \sum(\delta_i+\beta_j  )\in 2\bbZ +1,
\end{aligned}
$$
satisfying $\beta_1\ge \delta_1\ge \dots\ge\beta_p\ge |\delta_p|, \beta_i, \delta_j \in \bbZ +1/2$;

$$
\begin{aligned}
\xi_1: & (\delta_1,\dots,\delta_p\bigb  \beta_ 1+1/2, \dots,  -\beta_p -1/2) &&\text{ with } \sum(\delta_i+\beta_j) \in 2\bbZ+1, \\
\xi_2:   & (\delta_1,\dots,\delta_p\bigb  \beta_ 1+1/2, \dots, -\beta_p -1/2) &&\text{ with } \sum(\delta_i+\beta_j) \in 2\bbZ ,
\end{aligned}
$$
satisfying $\beta_1\ge \delta_1\ge \dots\ge\beta_p\ge |\delta_p|, \beta_i, \delta_j \in \bbZ +1/2$.

Another eight representations are obtained by restriction from $\tu{Spin}(2p,2p+1)$, with $\tu{K}$-structure:
$$
\begin{aligned}
\pi_1: & ( \beta_ 1+1/2, \dots, \beta_p +1/2\bigb \delta_1,\dots,\delta_p) &&\text{ with } \sum(\delta_i+\beta_j  ) \in 2\bbZ , \\
\pi_2:   & ( \beta_ 1+1/2, \dots, \beta_p +1/2\bigb \delta_1,\dots,\delta_p) &&\text{ with } \sum(\delta_i+\beta_j )\in 2\bbZ +1,
\end{aligned}
$$
satisfying $\beta_1\ge \delta_1\ge \dots\ge\beta_p\ge |\delta_p|, \beta_i, \delta_j \in \bbZ$;

$$
\begin{aligned}
\tau_1: & ( \beta_ 1+1/2, \dots,  -\beta_p -1/2\bigb \delta_1,\dots,\delta_p) &&\text{ with } \sum(\delta_i+\beta_j) \in 2\bbZ+1, \\
\tau_2:   & (  \beta_ 1+1/2, \dots, -\beta_p -1/2\bigb \delta_1,\dots,\delta_p) &&\text{ with } \sum(\delta_i+\beta_j ) \in 2\bbZ ,
\end{aligned}
$$
satisfying $\beta_1\ge \delta_1\ge \dots\ge\beta_p\ge |\delta_p|, \beta_i, \delta_j \in \bbZ$;

$$
\begin{aligned}
\sigma _3: & (   \beta_ 1+1/2, \dots,  \beta_p +1/2\bigb \delta_1,\dots,\delta_p) &&\text{ with } \sum( \delta_i+\beta_j) \in 2\bbZ, \\
\sigma _4:   & (  \beta_ 1+1/2, \dots, \beta_p +1/2\bigb \delta_1,\dots,\delta_p) &&\text{ with } \sum(\delta_i+\beta_j )\in 2\bbZ +1,
\end{aligned}
$$
satisfying $\beta_1\ge \delta_1\ge \dots\ge\beta_p\ge |\delta_p|, \beta_i, \delta_j \in \bbZ +1/2$;

$$
\begin{aligned}
\xi_3: & (  \beta_ 1+1/2, \dots,  -\beta_p -1/2\bigb \delta_1,\dots,\delta_p) &&\text{ with }\sum( \delta_i+\beta_j) \in 2\bbZ+1, \\
\xi_4:   & ( \beta_ 1+1/2, \dots, -\beta_p -1/2 \bigb \delta_1,\dots,\delta_p) &&\text{ with } \sum(\delta_i+\beta_j)\in 2\bbZ ,
\end{aligned}
$$
satisfying $\beta_1\ge \delta_1\ge \dots\ge\beta_p\ge |\delta_p|, \beta_i, \delta_j \in \bbZ +1/2$.

\smallskip
The representations with the same subscripts have the same central character.\\

\item[Case 2. $\tu{G_0}=\tu{Spin}(2p,2q), 2p=2k+2+2r_+, 2q=2k+2$, {$(r_+=p-q>0)$} ]\ 

{(Case 3 is corresponding to Case 2 with $+$ replaced by $-$.)}

There are {four} representations {obtained by restriction from $\tu{Spin}(2p+1,2q)$}, with $\tu{K}$-structure:
$$
\begin{aligned}
{\pi_1: }\ &(\delta_1,\dots ,\delta_q,0,\dots ,0\bigb \beta_1+r_++1/2,\dots ,\beta_q+r_++1/2) &&\text{with  }  \sum(\delta_i+\beta_j )\in 2\bbZ, \\
{\pi_2: }\ &(\delta_1,\dots ,\delta_q,0,\dots ,0\bigb \beta_1+r_++1/2,\dots ,\beta_q+r_++1/2) &&\text{with  }  \sum(\delta_i+\beta_j)\in 2\bbZ +1, \\
{\sigma_1:} \   &\big(\delta_1,\dots ,\delta_q,0,\dots ,0\bigb \beta_1+r_++1/2,\dots ,-(\beta_q+r_++1/2)\big)&&\text{with  }  \sum(\delta_i+\beta_j)\in 2\bbZ +1, \\
{\sigma_2:} \   &\big(\delta_1,\dots ,\delta_q,0,\dots ,0\bigb \beta_1+r_++1/2,\dots ,-(\beta_q+r_++1/2)\big)&&\text{with  }  \sum(\delta_i+\beta_j)\in 2\bbZ, \\
\end{aligned}
$$
satisfying $\beta_1\ge \delta_1\ge \beta_1\ge \dots \ge \beta_q\ge \delta_q\ge 0$, ${\beta_i, \delta_j\in \bbZ}$.

\smallskip
There {are two} representations  {obtained by restriction from $\tu{Spin}(2p,2q+1)$},  with $\wti K$-structure:
$$
\begin{aligned}
{\tau_1:} \ & (\beta_1,\dots ,\beta_q,0,\dots ,0   \bigb \delta_1, \dots, \delta_q )   &&\text{ with } \sum (\beta_i+\delta_j)\in 2\bbZ, \\
{\tau_2:} \ &  (\beta_1,\dots ,\beta_q,0,\dots ,0   \bigb \delta_1, \dots, \delta_q )  &&\text{ with } \sum (\beta_i+\delta_j)\in 2\bbZ +1, 
\end{aligned}
$$
satisfying  {$\beta_1 +r_+ -1/2 \ge \delta_1\ge \dots \ge \beta_q +r_+ -1/2 \ge | \delta_q |$ , $\beta_i\in \bbZ, \delta_j\in \bbZ+1/2$.}

\medskip
{The representations $\pi_i, \sigma_i, \tau_i$ have the same central character for $i=1,2$. }

\smallskip

\item[Case 4. $\tu{G_0}=\tu{Spin}(2p+1,2p+1)$, $2p+1=2k+1$, $(r_+=1 \text{ or  } r_- =1)$ ]\

{There is one representation (which may decompose further) obtained by restriction from $\tu{Spin}(2p+2,2p+1)$}, with  $\wti K$-structure
$$
 \pi_1:(\delta_1,\dots ,\delta_{p}\bigb \beta_1 +1/2 ,\dots ,\beta_p+1/2)
$$
{satisfying $\beta_1\ge \delta _1\ge \dots\ge \beta_p\ge \delta_p\ge 0,$ and $\beta_i, \delta _j \in\bb Z.$ }

\medskip
There is another representation (which may decompose further) {obtained by restriction from $\tu{Spin}(2p+1,2p+2)$}, with $\wti K$-structure 
$$
{\pi_2: }(\beta_1 +1/2 ,\dots ,\beta_p+1/2 \bigb\delta_1,\dots ,\delta_{p})
$$
{satisfying $\beta_1\ge \delta _1\ge \dots\ge \beta_{p}\ge \delta_{p}\ge 0,$ and $\beta_i, \delta _j \in\bb Z.$ }
The representations $\pi_1$ and $\pi_2$ have different central characters. 

\smallskip

\item[Case 5. $\tu{G_0}=\tu{Spin}(2p+1, 2q-1)$, $2p+1=2k+3+2r_+, 2q-1=2k+1$, {($r_+=p-q$)} ]\
{(Case 6 is corresponding to Case 5 with $+$ replaced by $-$.)}

 {
 When $r_+=0$, there are two representations obtained by restriction from $\tu{Spin}(2p+1,2p)$, with $K$-structure:
$$
\begin{aligned}
{\pi_1:} \ & (\beta_1,\dots ,\beta_p  \bigb \delta_1, \dots, \delta_{p-1} )   &&\text{ with } \beta_i\in \bbZ, \delta_j\in \bbZ+1/2\\
{\pi_2:} \ &  (\beta_1,\dots ,\beta_p  \bigb \delta_1, \dots, \delta_{p-1} )  &&\text{ with }  \beta_i\in \bbZ +1/2, \delta_j\in \bbZ, 
\end{aligned}
$$ 
 satisfying { $\beta_1 +1/2\ge \delta _1\ge \dots\ge \beta_{p-1} +1/2 \ge \delta_{p-1}\ge \beta_p+1/2$}. The representations $\pi_1$ and $\pi_2$ have different central characters. }
 
 \medskip
 When $r_+>0$,
there is a representation  obtained by restriction from $\tu{Spin} (2p+1, 2q)$, with $K$-structure:
{
$$
\pi: (\beta_1,\dots ,\beta_{q},\underbrace{0,\dots ,0} _{p-q}\bigb \delta_1,\dots ,\delta _{q-1})
$$

satisfying $\beta_1+r_++1/2\ge \delta _1 \ge \dots \ge \beta_{q-1}+r_++1/2\ge \delta_{q-1}\ge \beta_q+r_+ +1/2,$ and $\beta_i\in \bbZ, \delta_j\in \bbZ+1/2$. 
}

\smallskip
\item[Case 7. $\tu{G_0}=\tu{Spin}(2p+1,2q-1)$, $2p+1=2k+1+2r_+, 2q-1=2k+3$, {($r_+=p-q+2$)}] \
{ (Case 8 is corresponding to Case 7 with $+$ replaced by $-$.)}

There is one representation  {(which may decompose further)} obtained by restriction from $\tu{Spin}(2p+2,2q-1)$, with $\wti K$-structure 
$$
{\pi: }(\delta_1,\dots ,\delta_{q-1},\underbrace{0,\dots ,0} _{p-q+1}\bigb \beta_1 +r_+ -1/2 ,\dots ,\beta_{q-1}+r_+ -1/2)
$$
{satisfying $\beta_1\ge \delta _1\ge \dots\ge \beta_{q-1}\ge \delta_{q-1}\ge 0,$ and $\beta_i, \delta _j \in\bb Z.$ }

\end{description}
\end{theorem}
\subsection{Proof of Theorem \ref{t:main}}
\subsubsection*{Case 1} Let $p'=2p+1, q'=2p$, so 
$p'-1=q',$ and 
$\frac{p'-q'}{2}=\frac12.$ The  restrictions of the four representations of
$\tu{Spin}(2p+1,2p)$ are:
$$
\begin{aligned}
&\left(\delta_1,\dots ,\delta_{p}\bigb \beta_1+1/2,\dots ,\pm(\beta_p+1/2)\right),
&&\beta_i,\delta_j\in\bb Z,\\
&\left(\delta_1,\dots ,\delta_{p}\bigb \beta_1+1/2,\dots ,\pm(\beta_p+1/2)\right),
&&\beta_i,\delta_j\in\bb Z+1/2,  
\end{aligned}
$$
satisfying $\beta_1\ge \delta_1\ge\dots \ge \beta_p\ge |\delta_p|.$
Similarly we get another four representations by restricting from
$\tu{Spin}(2p,2p+1).$  

 The center of $\tu{Spin}(2p,2p)$ does not act by a scalar, so these
  representations decompose further into the sixteen listed in the
  theorem. Also, the highest weights of the $\wti K$-types of an
  irreducible representation must differ by the root lattice.

\subsubsection*{Case 2, 3} We consider  $a=2p=2k+2+2r_+, b=2q=2k+2, r_+
=p-q>0$ only.

Let {$p'= 2p+1, q'=2q.$} 
This is the case $p'-1>q'$, and so $\frac{p'-q'}{2}=r_++1/2$. The
restrictions of the two representations of $\tu{Spin}(p',q')$ are 
$$
\begin{aligned}
&(\delta_1,\dots ,\delta_{q},0,\dots ,0\bigb \beta_1+r_++1/2,\dots ,\beta_{q}+r_++1/2)\\
&\left(\delta_1,\dots ,\delta_{q},0,\dots ,0\bigb \beta_1+r_++1/2,\dots ,-(\beta_{q}+r_++1/2)\right)
\end{aligned}
$$
satisfying $\beta_1\ge \delta_1\ge \beta_1\ge \dots \ge \beta_q\ge \delta_q\ge 0$, $\beta_i,\delta_j\in \bbZ$.

\medskip
Let $p'=2k+3 =2q+1, q'=2k+2+2r_+ =2p.$ This is the case $p'-1<q'$, and
$\frac{q'-p'}{2}=r_+-1/2$. The restriction of the single representation
 of $\tu{Spin}(p',q')$ is 
$$
(\beta_1,\dots ,\beta_{q},0,\dots ,0   \bigb \delta_1,\dots,\delta_q)
$$
satisfying {$\beta_1 +r_+ -1/2 \ge \delta_1\ge \dots \ge \beta_q +r_+ -1/2 \ge | \delta_q |$ , $\beta_i\in \bbZ, \delta_j\in \bbZ+1/2$ .}

The center of $\tu{Spin}(2p,2q)$ does not act by a scalar, so these
  representations decompose further into the {six} listed in the
  theorem. Also, the highest weights of the $\wti K$-types of an
  irreducible representation must differ by the root lattice.

\subsubsection*{Case 4} Thus  $a=2p+1=2k+1$ and
$b=2q-1=2k+1$. 

Let $p'=2p+2$ and $q'=2p+1.$ There are two representations, and they
restrict to the same
$$
(\delta_1,\dots \delta_p\bigb \beta_1+1/2,\dots ,\beta_p+1/2) 
$$ 
satisfying $\beta_1\ge\delta_1\dots \beta_p\ge\delta_p\ge 0.$ 

Similarly for $p'=2p+1$ and $q'=2p+2.$ 
{ These representations 
decompose further, not detected by the action of the center; see
Conjecture \ref{conj} and the introduction. Their
$\wti K$-structure differs by whether $\sum \delta_i +\sum\beta_j$ is
in the root lattice or not. We write $\pi=\pi^e +\pi^o.$

}
\subsubsection*{Case 5, 6} We consider $a=2p+1=2k+3+2r_+, b=2q-1=2k+1$,
$r_+=p-q\ge 0$ only.} 

Let $p'=2p+1= 2k+3+2r_+, q'=2q= 2k+2.$ 
When $r_+>0,$  $p'-1> q'$ and $\frac{p'-q'}{2}=r_++1/2.$ The
restrictions of the two representations that occur for  $Spin(p',q')$ coincide:
$$
(\beta_1,\dots ,\beta_{q},\underbrace{0,\dots ,0} _{p-q}\bigb \delta_1,\dots ,\delta _{q-1})
$$
such that $\beta_1+r_++1/2\ge \delta _1 \ge \dots \ge \beta_{q-1}+r_++1/2\ge \delta_{q-1}\ge \beta_q+r_+ +1/2,$ and $\beta_i\in \bbZ, \delta_j\in \bbZ+1/2$. 

{ The case when $r_+=0$ satisfies $p'-1=q'.$ In addition to the
representation above,  there are two more representations. Their
restriction has  $\wti K$-structure
$$
(\beta_1,\dots ,\beta_{p}\bigb \delta_1,\dots ,\delta_{p-1})
$$
satisfying $\beta_1+1/2\ge\delta_1\ge\beta _2+1/2\ge \dots \ge \delta_{p-1}\ge \beta_p+1/2.$
}


\subsubsection*{Case 7, 8}  
We consider  $a=2p+1=2k+1+2r_+, b=2q-1=2k+3,  r_+=p-q+2\ge 0$ only. 

\medskip
Let $p'= 2q-1= 2k+3, q'=2p+2= 2k+2+2r_+.$ 
In this case, $p'-1<q'$, and $\frac{q'-p'}{2} = r_+-1/2$. The
  representation of $\tu{Spin}(p',q')$ restricts to
$$
{(\delta_1,\dots ,\delta_{q-1}, \underbrace{0,\dots ,0} _{p-q+1} \bigb \beta_1 +r_+-1/2,\dots ,\beta_{q-1}+r_+-1/2 )}
$$
satisfying $\beta_1\ge \delta_1\ge \beta_1\ge \dots \ge \beta_{q-1}\ge \delta_{q-1}\ge 0,$ $\beta_i, \delta_j\in \bbZ$.
As in Case 4, this representation
decomposes further, not detected by the action of the center; see
Conjecture \ref{conj} and the introduction. We write $\pi=\pi^e +\pi^o.$

{
\begin{conj}\label{conj}
Each representation in Case 4, Cases
7 and 8  decomposes into two irreducible factors; we write
$\pi=\pi^e+\pi^o.$ 
\end{conj}

The derived functors construction of the representations verifies this
conjecture. Since we have omitted the details of this alternate
construction, we list the above as a conjecture.

}

\subsection{Infinitesimal Character and Restriction} Let $\fk g=\fk{so}(2n,\bb C)\subset\fk
g'=\fk{so}(2n+1,\bb C),$ and $G=SO(2n,\bb C)$ and $G'=SO(2n+1,\bb C)$ the
corresponding groups  sharing a ($\theta$-stable) Cartan subgroup
$H=TA$. Let $\calI'$ be the unique maximal primitive ideal with
infinitesimal character 
$$
\la'=\left(n-k-1-1/2,\dots
  3/2,1/2 ; k+1,\dots , 1\right).
$$
There is a unique $(\fk g',K')$-module $\pi'$ with these properties,
and it is spherical unitary. In particular $\pi'=U(\fk g')/\calI'$. 
Let $\pi$ be any module with annihilator $\calI'.$ 
Then $\fk g$ acts  via the map $X\in\fk g\mapsto
X\cdot 1\in U(\fk g')/\calI'.$ Write $\pi'=\pi_0+\pi_1$ where $\pi_0$
is the unique spherical irreducible $(\fk g,K)$-submodule. The image
of $U(\fk g)$ is contained in $\pi_0.$  
We aim to show that $\pi_0$ has infinitesimal character
$\la=(n-k-2,\dots ,0 ;  k+1/2,\dots ,3/2, 1/2).$ Then all the factors of the
restriction of $\pi$ to $\fk g$ have this infinitesimal character as
well.
In particular this is true for the factors in the restrictions of the
modules of  $\fk{so}(p',q')$ considered 
in Theorem \ref{t:main}. 

\smallskip
It is enough to check the action of $\fk g$ on the
spherical function correponding to $\pi'$. By \cite{H} pages 31-32,
its restiction to $A$  is (up to a multiple), 
$$
\phi'=\frac{\sum_{w\in W(B_n)}\ep(w)e^{w\la'}}{\prod
\left(e^{\ep_i/2}-e^{-\ep_i/2}\right)\cdot\Delta}
$$   
with $\Delta=\prod_{\al\in R(D_n)} \left(e^{\al/2}-e^{-\al/2}\right)$
and $R(D_n)$ the
standard positive roots for type $D_n.$ The claim follows if we show that the
restriction of $\phi'$  to $G$ is the spherical function 
$$
\phi=\frac{\sum_{s\in W(D_n)}\ep(s)e^{s\la}}{\Delta}.
$$  
The next Lemma completes the proof.
\begin{lemma}
$$
\sum_{s\in  W(D_n)}\ep(s)e^{s\la}\cdot\prod\left(e^{\ep_i/2}-e^{-\ep_i/2}\right)=
\sum_{w\in W(B_n)}\ep(w)e^{w\la'}.  
$$
\end{lemma}
\begin{proof}
Both sides are skew invariant under $W(D_n).$ It is enough to count
the occurences of the dominant regular weights on the right. On the
left there are only two such weights, $(n-k-1/2, n-k-3/2,\dots ,
k+3/2,k+1,\dots ,3/2,\pm 1/2)$ occuring with opposite signs. On the
left, the weights are of the form 
$$
(n-k-2,\dots , k+1,k+1/2,\dots ,1,0)+(\pm 1,2,\dots ,\pm 1/2).
$$     
The parity of the number of $-1/2$ in the weight being added determines the
sign. The only weights that give a dominant regular sum are $(1/2,\dots
,1/2,1/2)$ and $(1/2,\dots ,1/2,-1/2).$ 
\end{proof}

\subsection{Matchup between regular sections on orbits and representations}\

{
We match the $\tu{K}$-spectra of the representations in Theorem \ref{t:main}
and the regular sections on nilpotent orbits computed in Section \ref{s:regsec}. We 
do this for Cases 1, 2, 4, 5, 7, and use the notation from Section \ref{s:regsec} (with possible
change from $-$ to $+$) and Theorem \ref{t:main}. The notation $\chi_i$  distinguishes different central characters.
In each table, the representations in the same row have the same central character; the representations in the same column 
are attached to the same orbit.
\bigskip

\begin{description}
\item[Case 1] \
\bigskip

$
\begin{array}{c|cccc }
&\calO =[3^+2^{2k}1^-] _I& \calO^{\zeta}=[3^+2^{2k}1^-] _{II}& \calO^{\eta}=[3^-2^{2k}1^+]_I &\calO^{\zeta\eta}= [3^-2^{2k}1^+]_{II} \\\hline
\chi_1& \pi _1|_{\tu{K}} = R(\calO,\psi_1)&   \tau _1|_{\tu{K}} = R(\calO ^{\zeta},\psi_2 ^{\zeta}) &   \sigma _1|_{\tu{K}} = R(\calO ^{\eta},\psi_3 ^{\eta}) &  \xi_1|_{\tu{K}} = R(\calO ^{\zeta\eta},\psi_4 ^{\zeta\eta})\\
\chi_2& \pi _2|_{\tu{K}} = R(\calO,\psi_2)&   \tau _2|_{\tu{K}} = R(\calO ^{\zeta},\psi_1 ^{\zeta})&   \sigma _2|_{\tu{K}} = R(\calO ^{\eta},\psi_4 ^{\eta})&  \xi_2|_{\tu{K}} = R(\calO ^{\zeta\eta},\psi_3 ^{\zeta\eta})\\
\chi_3& \sigma _3|_{\tu{K}} = R(\calO,\psi_3)&  \xi _3|_{\tu{K}} = R(\calO ^{\zeta},\psi_4 ^{\zeta})&   \pi _3|_{\tu{K}} = R(\calO ^{\eta},\psi_1 ^{\eta})& \tau_3|_{\tu{K}} = R(\calO ^{\zeta\eta},\psi_2 ^{\zeta\eta})\\
\chi_4& \sigma _4|_{\tu{K}} = R(\calO,\psi_4)&  \xi _4|_{\tu{K}} = R(\calO ^{\zeta},\psi_3 ^{\zeta})& \pi _4|_{\tu{K}} = R(\calO ^{\eta},\psi_2 ^{\eta})&\tau_4 |_{\tu{K}} = R(\calO ^{\zeta\eta},\psi_1 ^{\zeta\eta} ) 
\end{array}
$

\bigskip

\item[Case 2]\

\bigskip

\begin{tabular}{c|ccc}
& $\calO_I = [3^-2^{2k}1^{+,2r_++1}]_I$ & $\calO_{II}=\calO_I ^{\zeta} = [3^-2^{2k}1^{+,2r_++1}]_{II}$ & $\calO = [3^+2^{2k}1^-1^{+,2r_+}]$\\\hline
$\chi_1$& $\pi _1|_{\tu{K}} = R(\calO _I,\psi_1)$& $\sigma _1|_{\tu{K}} = R(\calO _{II},\psi_2 ^{\zeta})$&  $\tau _1|_{\tu{K}} = R(\calO,\phi_1)$\\
$\chi_2$&$\pi _2|_{\tu{K}} = R(\calO_I,\psi_2)$& $\sigma _2|_{\tu{K}} = R(\calO _{II},\psi_1 ^{\zeta})$  &$\tau _2|_{\tu{K}} = R(\calO,\phi_2)$
\end{tabular}

\bigskip

\item[Case 4] \ 

\bigskip

\begin{tabular}{c|cc}
& $\calO = [3^+2^{2k}1^+1^{-,2}]$ & $\calO^{\eta}= [3^-2^{2k}1^-1^{+,2}]$ \\\hline
$\chi_1$& \small ${ \pi _1| _{\tu{K}}=(\pi_1^e+\pi_1^o)}|_{\tu{K}} = R(\calO ,\psi_1)  + R(\calO ,\psi_2)$&  \\
$\chi_2$&    & \small ${\pi _2 |_{\tu{K}}=(\pi_2^e+\pi_2^o)}|_{\tu{K}} = R(\calO^{\eta} ,\psi_1 ^{\eta}) +R(\calO ^{\eta},\psi_2 ^{\eta})  $ 
\end{tabular}

\bigskip

\item[Case 5 with $r_+=0$]\

\bigskip

\begin{tabular}{c|c}
& $\calO = [3^+2^{2k}{1^{+}}]$  \\\hline
$\chi_1$& $\pi _1|_{\tu{K}} = R(\calO ,\psi_1)  $ \\
$\chi_2$& $\pi _2|_{\tu{K}} = R(\calO ,\psi_2)  $ 
\end{tabular}

\bigskip

\item[Case 5 with $r_+>0$]\

\bigskip

\begin{tabular}{c|c}
& $\calO = [3^+2^{2k}{ 1^{+, 2r_++1}}]$  \\\hline
$\chi_1$& $\pi _1|_{\tu{K}} = { R(\calO , Det ^{r_++1/2})}$ 
\end{tabular}

\bigskip

\item[Case 7]\

\bigskip

\begin{tabular}{c|c}
& $\calO = [3^-2^{2k} 1^- 1^{+,2r_+} ]$  \\\hline
$\chi_1$& ${ \pi | _{\tu{K}}=(\pi^e+\pi^o)} |_{\tu{K}} = R(\calO , \psi _1 ) +R(\calO, \psi_2)$ 
\end{tabular}
\end{description}
}
\bigskip

\section{Clifford algebras and Spin groups} \label{ss:clifford}
Since the main interest is in the case of $Spin(V),$ the simply
connected groups of type $D,$ we realize everything in the context of
the Clifford algebra.  

\subsection{Structure} Let $(V,Q)$ be a quadratic space of even dimension $2n$, with a basis
$\{e_i,f_i\}$ with $1\le i\le n,$ satisfying $Q(e_i,f_j)=\delta_{ij},$
$Q(e_i,e_j)=Q(f_i,f_j)=0$. Occasionally we will replace $e_j,f_j$ by
two orthogonal vectors $v_j,w_j$ satisfying  
$Q(v_j,v_j)=Q(w_j,w_j)=1,$ and orthogonal to the $e_i,f_i$ for $i\ne
j.$  Precisely they will satisfy $v_j=(e_j+f_j)/\sqrt{2}$ and 
$w_j=(e_j-f_j)/(i\sqrt{2})$ (where $i:=\sqrt{-1},$ not an index). Let
$C(V)$ be the Clifford algebra with automorphisms $\al$ defined by 
$\al(x_1\cdots x_r)=(-1)^r x_1\cdots x_r$ and $\star$ given by $(x_1\cdots
x_r)^\star=(-1)^r x_r\cdots x_1, $ subject to the relation
$xy+yx=2Q(x,y)$ for $x,y\in V$. The double cover of $O(V)$ is  
$$
Pin(V):=\{ x\in C(V)\ \mid\ x\cdot x^\star=1,\ \al(x)Vx^\star\subset V\}.
$$
The double cover $Spin(V)$ of $SO(V)$ is given by the elements in $Pin(V)$ which are in $C(V)^{even},$ \ie  $\disp{Spin(V):=Pin(V)\cap C(V)^{even}}.$
For $Spin,$ $\al$ can be suppressed from the notation since it is the identity.

\smallskip
The action of $Pin(V)$ on $V$ is given by $\rho(x)v=\al(x)vx^*.$ The
element $-I\in SO(V)$ is covered by 
\begin{equation}
\label{eq:-spin}
\pm \Ep_{2n}=\pm i^{n-1}vw\prod_{1\le j\le n-1} [1-e_jf_j]=\pm
i^n\prod_{1\le j\le n} [1-e_jf_j].  
\end{equation}
These elements satisfy 
$$
\Ep_{2n}^2=
\begin{cases}
  +Id &\text{ if } n\in 2\bb Z,\\
  -Id&\text { otherwise.}
\end{cases}
$$
The center of $Spin(V)$ is 
$$
Z(Spin(V))=\{\pm I, \pm\Ep_{2n}\}\cong
\begin{cases}
\bb Z_2\times \bb
Z_2 &\text{ if }n \text{ is even,}\\ 
\bb Z_4 &\text{ if } n \text{ is odd}.
\end{cases}
$$
The Lie algebra of $Pin(V)$ as well as $Spin(V)$ is formed of elements
of even order $\le 2$ satisfying
$$
x+x^\star=0.
$$ 
The adjoint action is {$\ad x(y)=xy-yx$}. A
Cartan subalgebra and the root vectors corresponding to the usual
basis in Weyl normal form are  formed of the elements 

\begin{equation}
\begin{aligned}
\label{eq:liea}
&(1-e_if_i)/2&&\longleftrightarrow &&H(\ep_i)\\
&e_ie_j/2&&\longleftrightarrow &&X(-\ep_i-\ep_j),\\
&e_if_j/2&&\longleftrightarrow &&X(-\ep_i+\ep_j),\\
&{f_i} f_j/2&&\longleftrightarrow &&X(\ep_i+\ep_j).
\end{aligned}
\end{equation}
\subsubsection*{Root Structure}\ We use $1\le i, l\le p$ and $1\le j, m \le
q-1$ consistently. We give a realization of the Lie algebra for
$Spin(2p+1,2q-1)$. The case $Spin(2p,2q)$, is (essentially) obtained
by suppressing the short roots.    

\medskip
\begin{tabular}{ll}
{Compact}  &Noncompact\\
&\\
${\fk t}=\{(1-e_if_i), (1-e_{p+j}f_{p+j})\}$ & $\fk a=\{v^+v^-\}$\\
$H(\ep_i),\  H(\ep_{p+j})$ &$H(\ep_{p+q})$\\
$f_iv^+, e_iv^+, f_{p+j}v^-, e_{p+j}v^-$ &$f_iv^-, e_iv^-, v^+f_{p+j}, v^+e_{p+j}$\\
$X(\ep_i)_c, X(-\ep_{i})_c, X(\ep_{p+j})_c, X(-\ep_{p+j})_c$&$X(\ep_i)_{n}, X(-\ep_i)_{n} X(\ep_{p+j})_{n}, X(-\ep_{-p+j})_{n}$\\
$f_if_l, f_ie_l, e_ie_l, e_if_l $& $f_if_{p+j}, f_ie_{p+j}, e_ie_{p+l}, e_if_{p+l}$\\
$f_{p+j}f_{p+m},f_{p+j}e_{p+m}, e_{p+j}f_{p+m}, e_{p+j}e_{p+m}$& \\
$X(\ep_i+\ep_l),X(\ep_i-\ep_l), X(-\ep_i-\ep_l), X(-\ep_i+\ep_l)$&
$X(\ep_i+\ep_{p+j}), X(\ep_i-\ep_{p+j})$,\\
& $X(-\ep_i-\ep_{p+j}), X(-\ep_i+\ep_{p+j})$,\\
$X(\ep_{p+j}+\ep_{p+m}), X(\ep_{p+j}-\ep_{p+m}),$ &\\ 
$X(-\ep_{p+j}-\ep_{p+m}), X(-\ep_{p+j}+\ep_{p+m}).$&
\end{tabular}

\subsubsection*{Nilpotent Orbits, Complex Case} 
 In this case, we write $\tu{K}=Spin(V)=Spin (2n,\bbC) $, $K=SO(V)=SO(2n,\bbC)$.  A nilpotent orbit of an element $e$  will have Jordan blocks denoted by
\begin{equation}
\label{eq:blocks}
\begin{aligned}
&e_1\longrightarrow e_2\longrightarrow\dots \longrightarrow
e_k\longrightarrow v\longrightarrow -f_k\longrightarrow
f_{k-1}\longrightarrow {-f_{k-2} \longrightarrow}  \dots \longrightarrow \pm f_1\longrightarrow 0\\
&\begin{matrix}
&e_1\longrightarrow &e_2&\longrightarrow&\dots &\longrightarrow
&e_{2\ell}\longrightarrow 0\\
&f_{2\ell}\longrightarrow &-f_{2\ell-1}&\longrightarrow&\dots &\longrightarrow
&-f_1\longrightarrow 0
\end{matrix} 
\end{aligned} 
\end{equation}
with the conventions about the $e_i,f_j,v$ as before. There is an even
number of odd sized blocks, and any two blocks of equal odd size $2k+1$
can be replaced by a pair of   blocks of the form as the even ones.  A realization of the odd block is given by $ \displaystyle\frac{1}{2}\left ( {\sum \limits _{i=1} ^{k-1} }e_{i+1}f_i +vf_k\right ),$ and a realization of the even blocks by $\dpfr \left( {\sum \limits _{i} ^{2l-1}}e_{i+1}f_{i}\right).$ When there are only even blocks, there are two orbits; one block of the form 
$\big(\sum_{1\le i< \ell-1} e_{i+1}f_{i}+e_\ell f _{\ell-1}\big)/2$ is replaced by  $\big(\sum_{1\le i< \ell-1} e_{i+1}f_{i}+ f_\ell f_{\ell-1}\big)/2.$ 

\medskip
The centralizer of $e$ in $\fk{so}(V)$  has Levi component isomorphic to a product 
of $\fk{so}(r_{2k+1})$ and $\fk{sp}(2r_{2\ell})$ where $r_j$ is the number of
blocks of size $j.$ The centralizer of $e$ in $SO(V)$ has Levi
component $\prod Sp(2r_{2\ell})\times S[\prod O(r_{2k+1})]$. 
For each odd sized block define
\begin{equation}
  \label{eq:epsilon}
\Ep_{2k+1}=i^{k}v\prod (1-e_jf_j).  
\end{equation}
This is an element in $Pin(V),$ and acts by $-Id$ on the block. Even
products of $\pm \Ep_{2k+1}$  belong to $Spin(V),$ and represent the
connected components of $C_{\wti{K}}(e).$   
\begin{prop}
Let $m$ be the number of distinct odd blocks. Then 
$$A_K(\calO) \cong  \begin{cases} \bbZ _2 ^{m-1} & \mbox{ if }  m >  0 \\ 1 & \mbox{ if } m=0.  \end{cases} $$ Furthermore, 
\begin{enumerate}

\item If $E$ has an odd block of size $2k+1$ with $r_{2k+1}>1,$ then
$A_{\wti{K}}(\calO)\cong A_K(\calO).$

\item If all $r_{2k+1}\le 1,$ then there is an exact sequence
$$
1\longrightarrow \{\pm I\}\longrightarrow
A_{\wti{K}}(\calO)\longrightarrow A_K(\calO)\longrightarrow 
0.
$$
\end{enumerate}
\end{prop}
\begin{proof}

Assume that there is an ${r_{2k+1}>1}.$ Let 
$$
\begin{matrix}
&e_1&\rightarrow    &\dots &\rightarrow &e_{2k+1}&\rightarrow 0\\
&f_{2k+1}&\rightarrow&\dots &\rightarrow &-f_1   &\rightarrow 0
\end{matrix}
$$   
be two of the blocks. In the Clifford algebra this element is
$e=(e_2f_1+\dots +e_{2k+1}f_{2k})/2.$ The element ${\sum \limits _{ j=1} ^{2k+1}}(1-e_{j}f_{j})$ in the Lie
algebra commutes with $e$. So its exponential
\begin{equation}\label{path1}
\prod \exp\big( i\theta(1-e_{j}f_{j})/2\big)=
\prod [\cos\theta/2 + i\sin\theta/2 (1-e_{j}f_{j})]
\end{equation}
also commutes with $e.$  {At $\theta=0$, 
the element in (\ref{path1})  is $I$; at $\theta =2\pi$, it is $-I$.} Thus $-I$ is in the connected component of the identity of
$A_{\wti{K}}(\calO)$ (when $r_{2k+1}>1$), and therefore $A_{\wti{K}}(\calO)=A_K(\calO).$  

\medskip
Assume there are no blocks of odd size. Then $C_K(\calO) {\cong \prod Sp (r_{2l})}$ is simply
connected, so $C_{\wti{K}}(\calO)\cong C_K(\calO)\times\{\pm I\}.$  { Therefore $A_{\tu{K}} (\calO) \cong \bbZ _2$.}

\medskip
Assume there are {$m$ distinct odd blocks with $m\in 2\bbZ _{>0}$ and $r_{2k_1+1}=\cdots =r_{2k_m +1}=1.$ In this case, $C_K(\calO)\cong \prod Sp(r_{2l}) \times S[  \underset{m}{\underbrace{O(1) \times \cdots \times O(1)}} ]$
, and hence $A_{\tu{K}} (\calO)\cong \bbZ_2^{m-1}$. 
Even products of $\{ \pm \Ep _{2k_j +1} \}$ are representatives of elements in $A_{\tu{K} }(\calO)$.} They satisfy  

$$
\Ep_{2k+1}\cdot\Ep_{2\ell+1}=
\begin{cases}
-\Ep_{2\ell +1}\cdot\Ep_{2k+1}  &k\ne \ell,\\
{ (-1)^kI}  &k=\ell.
\end{cases}
$$  
\end{proof}

\begin{cor} \ 
\label{c:cgp}
  \begin{enumerate}
  \item If $\calO=[3\ 2^{n-2}\ 1],$ then {$A_{\wti{K}}(\calO)\cong
\bb Z_2\times\bb Z_2=\{\pm\Ep_3\cdot\Ep_1,\pm
I\}$}. 
\item If $\calO=[3\ 2^{2k}\ 1^{2n-4k-3}]$ with  $2n-4k-3>1,$ then  $A_{\wti{K}}(\calO)\cong
\bb Z_2.$
\item If  $\calO=[2^{n}]_{I,II}$ ($n$ even), then $A_{\wti{K}} (\calO) \cong\bb Z_2.$ 
\item If $\calO=[2^{2k}\ 1^{2n-4k}]$ with $2k<n,$ then
  $A_{\wti{K}} (\calO) \cong 1.$ 
  \end{enumerate}
In all cases  $C_{\wti K}(\calO)=Z(\wti K)\cdot C_{\wti K}(\calO)^0.$ 

\end{cor}

\subsubsection*{Nilpotent Orbits, Real Case}

Write $V=V^+ \oplus V^-$ a sum
of two (complex) spaces, each endowed with a nondegenerate quadratic
form. Recall the notation in Section \ref{ss:1}. The spaces $V^\pm$
have dimensions $a$ and $b.$ We use bases as in the complex case,
$e^\pm_j,f^\pm_j,$ and plus $v^\pm$ when $a,b$ are both odd. 
In this case, we write $\tu{K}=Spin(V^+)\times Spin(V^-)=Spin(a,\bbC)\times Spin(b,\bbC)$,
$\tu{K}=SO(V^+)\times SO(V^-)=SO(a,\bbC)\times SO(b,\bbC)$.
Write
$\fk g=\fk k+\fk s$ for the (complexification of the) Cartan
decomposition. Nilpotent orbits of $\tu{K}$ (as well as $K$) in $\fk s$ are parametrized by
signed partitions where the basis elements alternate between   
$V^+$ and $V^-$ but otherwise as in Equation (\ref{eq:blocks}).
The centralizer of $e$ in $\fk k$ is a product $\prod
\fk{sp}(2r_{2\ell} ) \times\prod [\fk{so}(r^+_{2k+1} )\times \fk{so}(r^-_{2k+1})]$ where
$r_{2l}$ is the number of blocks of even size $2\ell$, $r^\pm_{2k+1}$ is the number of blocks of odd size $2k+1$ starting with $\pm.$
We compute the centralizer of $e$ in $\wti K,$ and its component
group.  

The even sized blocks do not contribute to the component group. They
can however be used to deduce that $(-I,-I)\in C_{\wti K}(\calO)^0$ as
in the complex case.

For the odd sized blocks,  
\begin{equation}
  \label{eq:epsilonpm}
  \begin{aligned}
&\Ep_{2k+1}^+=i^{k}(v^+\prod (1-e^+_jf^+_j ) ,\prod (1-e_j^-f_j^-))\\  
&\Ep_{2k+1}^-=i^{k}(\prod (1-e^+_jf^+_j ),v^-\prod (1-e_j^-f_j^-))\\      
  \end{aligned}
\end{equation}
Products with an even number of both $\pm$ of such elements give
representatives of the component group.

\begin{lemma} \label{real-comp-le}\

\begin{enumerate} 
\item
If $r_{2k+1}^{+} >1$ for some $2k+1\ge 1,$  then 
$$
\begin{aligned}
&(-I,I)\in C_{\wti{K}}(\calO)^0, \quad k  \text {  even} ,\\
&(I,-I)\in C_{\wti{K}}(\calO)^0, \quad k  \text {   odd} .
\end{aligned}
$$
Similarly for $r_{2k+1}^-$ with $k$ even and odd interchanged.
 \item If $r_{2\ell} >1$ for some $2\ell>0,$ then $(-I,-I)\in C_{\tu{K}} (\calO)^0$.

\end{enumerate}

\begin{proof}
Assume that  $r_{2k+1} ^+>0$. 
$$
\begin{matrix}
&e_1^ + &\rightarrow & e_2 ^- & \rightarrow   &\dots &\rightarrow &e_{2k+1}^ +&\rightarrow 0\\
&f_{2k+1}^ + &\rightarrow& -f_{2k}^- & \rightarrow & \dots &\rightarrow & -f_1^+   &\rightarrow 0
\end{matrix}
$$   
represent two equal size blocks  starting with the same sign $+$.
The corresponding element in the Clifford algebra is   
$e=(e_2 ^- f_1^+ +\cdots+ e^+ _{2k+1} f^- _{2k})/2$. Similar to
the case $r_{2\ell}>0$ below, 
$$
\left (   \prod _{j=0} ^k [ \cos \theta + i\sin \theta (1-
  e_{2j+1}^+f_{2j+1}^+)],    \   \prod _{j=0} ^k [\cos \theta + i\sin
  \theta (1- e_{2j}^- f_{2j}^-) ]  \right  )\in Spin(V^+)\times Spin(V^{-})
$$ 
centralizes $e$ for all $\theta$. 
This   gives a continuous path between $(I,I)$ and $(I,-I)$ when $k$
is odd; and  a continuous path between $(I,I)$ and $(-I,I)$
when $k$ is even.

Assume that some $r_{2\ell}> 1$.  Let 
$$
\begin{matrix}
&e_1^ + &\rightarrow & e_2 ^- & \rightarrow   &\dots &\rightarrow &e_ {2\ell}^ -&\rightarrow 0\\
& f_{2\ell}^ -&\rightarrow& -f_{2\ell -1}^+ & \rightarrow & \dots &\rightarrow & -f_1^+   &\rightarrow 0
\end{matrix}
$$   represent two blocks of size $2\ell$. Again, let $e=(e_2 ^- f_1^+  +e_3 ^+f_2 ^-  +\cdots+ e^- _{2\ell} f^+ _{2\ell -1})/2$
be a representative. Then  $e$ is centralized by $[(1-e_1^+f_1^+)+ (1-e_2 ^- f_2 ^-)+\cdots + (1-e^- _{2\ell} f^- _{2\ell})]/2 $ in the Lie algebra.
 Exponentiating, 
\begin{eqnarray*}
\left (   \prod _{j=1} ^{\ell } [\cos \theta/2 + i\sin \theta/2 (1-
  e_{2j-1}^+f_{2j-1}^+)],    \   \prod _{j=1} ^{\ell} [\cos \theta/2 +
  i\sin \theta/2 (1- e_{2j}^- f_{2j}^-) ]  \right  )  \\  \in
Spin(V^+)\times Spin(V^-)
\end{eqnarray*}
centralizes $e$ for all $\theta$. This is $(I,I)$ when $\theta = 0$,
and is $(-I,-I)$ when $\theta =\pi$, and hence it   gives a continuous
path between $(I,I)$ and $(-I,-I)$.

\end{proof}
\end{lemma}

We apply the Lemma to the orbit of the diagram $[3\ 2^{2k}\ 1^{2n-4k-3}]$ with $k>0.$ 
\begin{prop} [Proposition \ref{comp-group-real}] \

  \begin{enumerate}
\item If $\calO=[3^+ 2^{2k}1^-]_{I,II},$ then $A_{\wti K}(\calO) \cong  \bbZ_2\times \bbZ_2$.
\item If $\calO=[3^+ 2^{2k}1^+],$ then $A_{\wti K}(\calO) \cong \bbZ_2$.
\item If $\calO=[3^+ 2^{2k}1^{+,2r_+ +1}]$ with $r_+>0$, then
  $A_{\wti K}(\calO)= {1}$.
\item If $\calO=[3^+ 2^{2k}1^-1^{+,2r_+}],$ with $r_+>0$, then $A_{\wti K}(\calO)\cong  {\bbZ_2 }$.
\item If $\calO=[3^- 2^{2k}1^{+,2r_+ +1}],$ with $r_+>0$, then  $A_{\wti K}(\calO)\cong {\bbZ _2}$.
\item If $\calO=[3^- 2^{2k}1^-1^{+,2r_+}],$ with $r_+\ge 2$, then  $A_{\wti K}(\calO)\cong {\bbZ_2 }$.
  \end{enumerate}
Similarly for the nilpotent orbits with the $+$ and $-$ interchanged. 
\end{prop}

\begin{proof}
In case (1),  the odd blocks can be represented by
$$
\begin{aligned}
&e_1^+\longrightarrow v_2^-\longrightarrow -f_1^+\longrightarrow 0,\\ 
&w_2^-\longrightarrow 0. 
\end{aligned}
$$
The
corresponding element in the Clifford algebra is $e=(v_2 ^- f_1
^+)/2$. The element $( i (1- e_1 ^+ f_1^+) ,  v_2 ^- w_2 ^- ) = (  i (1- e_1 ^+ f_1^+),  i (1- e_2 ^- f_2^-))$ is in
$Spin(V^+)\times Spin(V^-)$, acts by $-Id$ on the blocks and centralizes $e$.
 
Note that $A_{K}(\calO)\cong \bbZ_2.$ The inverse image in
$C_{\tu{K}}(\calO)$ of $C_{K}(\calO)$ contains $\{  (\pm I, \pm I) ,  (\pm  i (1- e_1 ^+ f_1^+), \pm  i (1- e_2 ^- f_2^-))\}$.
 By Lemma \ref{real-comp-le} (2), $(-I,-I)\in C_{\tu{K}} (\calO) ^0$, so $A_{\tu{K}} (\calO) \cong \bbZ_2\times \bbZ_2$. 

\medskip
In case (2),  there is only one orbit with this signed partition. The
odd blocks are represented by 
$$
\begin{aligned}
&e_1^+\longrightarrow  v_3^-\longrightarrow -f_1^+\longrightarrow 0,\\ 
&v_2^+\longrightarrow 0. 
\end{aligned}
$$
  The corresponding element in the
 Clifford algebra is $e=(v_3 ^- f_1 ^+)/2$.
The element $\big(iv_2^+(1-e_1^+f_1^+),v_1^-\big)$ acts by $-Id$ on the blocks and centralizes $e,$
but is in $Pin(V^+)\times Pin(V^-),$ so cannot contribute to the
centralizer $C_{\wti{K}}(e).$   As in the previous case, $(I,I)$ and $(-I,-I)$ are in the same connected component. Thus $A_{\tu{K}} (\calO)\cong \bbZ _2$.

\medskip
In cases (3)--(6), Lemma \ref{real-comp-le} implies that $(\pm I,\pm
I)\in C_{\wti K}(\calO)^0.$ So $A_{\wti K}(\calO)=A_K(\calO).$ 
\end{proof}

{
\begin{remark}
In Proposition \ref{comp-group-real}, the generators of $A_{\tu{K}}(\calO)$ can be
chosen as follows: (1) $(-I,-I)$,  $(  i (1- e_1 ^+ f_1^+),  i (1- e_2 ^- f_2^-))$; (2) $(I,-I)$
; (3) $(I,I)$; (4)  $(  i (1- e_1 ^+ f_1^+),  i (1- e_2 ^- f_2^-))$; (5)  $(  i (1- e_1 ^+ f_1^+),  i (1- e_2 ^- f_2^-))$; (6)  $(  i (1- e_1 ^+ f_1^+),  i (1- e_2 ^- f_2^-))$.
Furthermore, in cases (1), (2), (4), (5),  nontrivial representatives of $A_{\tu{K}}(\calO)$ can be chosen to be elements in $Z(\tu{K})$. 
\end{remark}
}

\begin{example}
Let $\calO=[3^-2^2 1^-1^{+,2}],$ \ie $k=1,$ $r_+=1.$  Then $A_{\wti
  K}(\calO)\cong A_K(\calO)$ has two connected components. The Jordan
blocks are
$$
\begin{aligned}
&e_1^-\longrightarrow v^+\longrightarrow -f_1^-\longrightarrow 0\\
&e_2^+\longrightarrow e_2^-\longrightarrow 0\\
&f_2^-\longrightarrow -f_2^+\longrightarrow 0\\
&v^-\longrightarrow 0\\
&v_1^+\longrightarrow 0\\
&w_1^+\longrightarrow 0.
\end{aligned}
$$
The group $C_{\wti K}(\calO)$ can be written as 
$$
\left(\pm v^+ v_1^+   ,\pm(1-e_1^-f_1^-)    \right)\cdot\left(
  \cos\frac{\theta_1}{2} +i\sin\frac{\theta_1}{2} v_1^+w_1^+,I\right)\cdot \left( \cos\frac{\theta_2}{2}(I,I)+i\sin\frac{\theta_2}{2}(1-e_2^+f_2^+,1-e_2^-f_2^-)\right)
$$

\end{example}

\bigskip

\section{Counting representations} 

In this section, we will count the number of representations attached
to the complex nilpotent orbit of the form $\calO _c=[3\ 2^{2k}\
1^{2n-4k-3}]$, with $k>0$. This orbit is not special in the sense of Lusztig. 
The infinitesimal character cannot be integral.
Considerations coming from primitive ideals imply that
the infinitesimal character must be regular, and have integrality
given by the system $D_{k+1}\times D_{n-k-1}.$ The
infinitesimal character with minimal length satisfying the above
conditions, and corresponding to $\calO _c$  must be

\begin{equation}
  \label{infchar-orbit}
 \la=(n-k-2,\dots ,1,0;k+1/2,\dots ,3/2, 1/2)
\end{equation}
 with $k+1\le n-k-1$, and hence  $0< k\le n/2 -1$.
These are the infinitesimal characters which (conjecturally) admit
unitary unipotent representations attached to $\calO_c$.

The setting is as in the previous sections. 
For convenience, in this section we use slightly different notation.   
We write $(\tu{G},\tu{K})=(\tu{Spin}(c,d), Spin(c)\times Spin(d))$ with  $2n=c+d$, and 
$$
\begin{aligned}
&c=2p,\qquad &&c=2p+1,\\
&d=2q,\qquad &&d=2q+1.
\end{aligned}
$$
We assume that $c\ge d$ in this section. We will classify all groups 
$\tu{G}$ that admit an admissible representation with infinitesimal character 
$\la$. It turns out that Cases 1, 2, 4, 5, 7 in Section \ref{ss:1} cover such groups. 
Let $n_{\calO} : = | \calU _{\tu{G }} (\calO_c , \la) |$ be the number of unipotent  representations of $\tu{G}$ attached to $\calO_c$ and $\la$.
We ultimately calculate $n_{\calO}$ for each case.

Before getting further, we need some structure theory.

\subsection{Cartan subalgebras of $\fk g_0$}\ 
{Conjugacy classes of Cartan
  subalgebras have the following representatives, with the given $\theta$:
$$
\begin{aligned}
&\fk h^{r^+,r^-, m,s}=
\{(x_1^+,\dots ,x^+_{r^+},x_1^-,\dots ,x_{r^-}^-,y_1,\dots y_m,y_{m+1},\dots ,y_{2m},z_1,\dots ,z_{s})\},\\
&\theta(x_i^\pm)=x_i^\pm,\ \theta(y_j)=y_{j+m},\ \theta (z_k)=-z_k.
\end{aligned}
$$
When $\fk g_0=\fk{so}(2p,2q),$ $s$ is even, when $\fk g_0=\fk{so}(2p+1,2q+1),$ $s$ is odd.
We write $s=2s'+\ep$, where $\ep=0$ when $n$ is even, and $\ep =1$ when $n$ is odd. 
Furthermore,
$m+r^++s'=p$ and $m+r^-+ s'=q.$  
The orthogonal space $V=V^+\oplus V^-$  has basis 
\begin{equation}
  \label{eq:vbasis}
e_i,v^+,f_i,\quad e_{j},v^-,f_{j}\qquad 1\le i\le p,\ {p+1}\le j\le {p+q } 
\end{equation}
with $e_i,f_i,v^+$ a basis of $V^+$ and $e_{j},f_{j},v^-$ a basis
of $V^-.$ The $e,f$ are isotropic and in duality, the $v^\pm$ unit
vectors orthogonal to the $e,f.$

We will use $1\le {i, l }\le p$ and ${p+1\le j,k \le p+q}$ consistently. When
convenient, we denote $c=2p, 2p+1$ and $d=2q,2q+1.$ 

The Lie algebra with respect to the fundamental Cartan subalgebra is realized as follows. For $\fk{so}(2p,2q)$, $v^\pm$ and the corresponding terms are missing.

\medskip
Recall the basis of $\fk g$ formed of the (complexification of the) 
fundamental Cartan subalgebra and its root vectors:

\medskip
\begin{tabular}{ll}
{Compact}  &Noncompact\\
&\\
${\fk t}=\{\sqrt{-1}(1-e_if_i)/2, \sqrt{-1}(1-e_{j}f_{j})/2\}$ & $\fk a=\{v^+v^-\}$\\
$\sqrt{-1}H(\ep_i),\  \sqrt{-1}H(\ep_{j})$ &$H(\ep_{  {p+q +1} })$\\
$f_iv^+, e_iv^+, f_{j}v^-, e_{j}v^-$ &$f_iv^-, e_iv^-, v^+f_{j}, v^+e_{j}$\\
$X(\ep_i)_c, X(-\ep_{i})_c, X(\ep_{j})_c, X(-\ep_{j})_c$&$X(\ep_i)_{n}, X(-\ep_i)_{n} X(\ep_{j})_{n}, X(-\ep_{j})_{n}$\\
$f_if_l, f_ie_l, e_ie_l, e_if_l $& $f_if_{j}, f_ie_{j}, e_ie_{j}, e_if_{j}$\\
${ f_{j}f_{k},f_{j}e_{k}, e_{j}f_{k}, e_{j}e_{k} }$& \\
$X(\ep_i+\ep_l),X(\ep_i-\ep_l), X(-\ep_i-\ep_l), X(-\ep_i+\ep_l)$&
$X(\ep_i+\ep_{j}), X(\ep_i-\ep_{j})$,\\
& $X(-\ep_i-\ep_{j}), X(-\ep_i+\ep_{j})$,\\
${X(\ep_{j}+\ep_{k}), X(\ep_{j}-\ep_{k}),}$ &\\ 
${X(-\ep_{j}-\ep_{k}), X(-\ep_{j}+\ep_{k})}.$&
\end{tabular}

Realizations of the other Cartan subalgebras $\fk h^{{r^+,r^-, m,s }}$ are
$$
\begin{aligned}
&\sqrt{-1}(1-e_if_i)/2, \sqrt{-1}(1-e_{j}f_{j})/2,\ 1\le i\le r^+, \
{p+1\le j\le p+ r^-},\\
&\sqrt{-1}H(\ep_i),\sqrt{-1}H(\ep_j),\\
&\sqrt{-1}(e_{r^++t}f_{r^++t}-e_{p+r^-+t}f_{p+r^-+t})/2,e_{r^+ +t}e_{p+r^- +t}+f_{r^++t}f_{p+r^- +t},\ 1\le t \le m\\ 
&{\sqrt{-1}H(\ep _{r^+ + t} -\ep _{p+r^- +t}),  X(\ep_{r^+ +t } +\ep _{p+r^- +t} ) + X( - \ep _{r^+ +t} -\ep _{p+r^- +t}) }\\
&e_{r^++m+\ell}e_{p+r^- +m+l}+f_{r^++m+\ell}f_{p+r^-+m+\ell}, 
e_{r^+ +m+\ell}f_{p+r^- +m+\ell}+f_{r^+ +m+\ell}e_{p+r^- +m+\ell},\
1\le \ell\le s,\\
&X(-\ep_{r^++m+\ell}-\ep_{p+r^-+m+\ell})+X(\ep_{r^++m+\ell}+\ep_{p+r^-+m+\ell}),\\
&X(\ep_{r^++m+\ell}-\ep_{p+r^-+m+\ell})+X(-\ep_{r^++m+\ell}+\ep_{p+r^-+m+\ell}).
\end{aligned}
$$
{When $p=q=m$, there are two nonconjugate Cartan subalgebras, denoted $\fk h ^{0, 0, m,0}_{I, II}$:

$\fk h^{0,0,m,0}_{I}$ is generated by
$$H(\ep _{ t} -\ep _{p+t}), H(\ep_{p}-\ep_{2p}), X(\ep_{t } +\ep _{p+t} ) +  X( - \ep _{t} -\ep _{p+t}), X(\ep_{p} +\ep _{2p} ) +  X( - \ep _{p} -\ep _{2p}),\ 1\le t \le p-1;$$
$\fk h^{0,0,m,0}_{II}$ is generated by
$$ H(\ep _{ t} -\ep _{p+t}),H(\ep_p+\ep_{2p}),X(\ep_{t } +\ep _{p+t} ) + X( - \ep _{t} -\ep _{p+t}) , X(\ep_{p } -\ep _{2p} ) -X( - \ep _{p} +\ep _{2p})  ,  1\le t\le p-1.
$$
}

}
\subsection{Center of $\mathbf{\tu{G}}$}

The center of $\wti G$ is contained in the maximal compact subgroup 
$\tu{K}$. 

\begin{lemma}\
\begin{itemize}
\item[(1)] When $c=2p+1, d=2q+1$,
 $Z(\tu{G})=\{ (\pm I, \pm I )\}\cong \bbZ_2\times \bbZ_2$.
\item[(2)] When $c=2p, d=2q$, 
\begin{equation}
\begin{aligned}
Z(\tu{G})&= \left \{ (\pm I, \pm I) , \left (\pm  i^p  \prod_{i=1} ^p (1-e_i f_i)  , \pm i^q \prod _{j=1} ^q  (1-e_{p+j} f_{p+j}) \right ) \right \} \\
& \cong \begin{cases}    \bbZ_2\times \bbZ_4 & \text{ if at least one of $p$ and $q$ is odd,} \\  \bbZ _2 \times \bbZ_2 \times \bbZ _2 &  \text{ otherwise.}   \end{cases}
\end{aligned}
\end{equation}
\end{itemize}
\end{lemma}

\begin{lemma}\label{spin-central-char}
Let $\mu= (a_1, \dots, a_p\mid b_1, \dots, b_q)$ be a 
$\wti K$-type parametrized by its highest weight, and let $\chi$ be the restriction of the highest 
weight of $\mu$ to $Z(\tu{G})$. Then 
\begin{equation}\label{center-spincover}
\begin{aligned}
&\chi  (\ep_1 I, \ep_2 I)  = \ep_1 ^{2a_1} \ep_2 ^{2b_1}, \\
& \chi  \left ( \ep_1 i^p  \prod_{i=1} ^p (1-e_i f_i)  , \ep_2
  i^q \prod _{j=1} ^q  (1-e_{p+j} f_{p+j}) \right ) = \ep_1 ^{2a_1}\ep_2
^{2b_1} i^{ 2(a_1+\cdots+a_p +b_1+\cdots +b_q)} 
\end{aligned}
\end{equation}
where $\ep_j=\pm 1.$
\begin{proof}
As a linear functional of the fundamental Cartan subalgebra, $\mu$
acts by 
\begin{eqnarray*}
\mu:\frac{\theta_1 \sqrt{-1} (1-e_1f_1)}{2} +\cdots +
\frac{\theta_{p+q} \sqrt{-1} (1-e_{p+q}f_{p+q})}{2} \\  \mapsto 
  \sqrt{-1}(a_1\theta _1+\cdots +a_p\theta_p+b_1\theta_{p+1}+\cdots
  +b_q \theta_{p+q}) .
\end{eqnarray*}

Setting $\theta_1=2\pi$ and $\theta_j=0 $ for all $j\neq 0$, 
$$
\chi:(-I, I) \mapsto e^{2\pi i a_1}=(-1)^{2a_1},
$$
since  
$$
{e^{\sqrt{-1}\theta _i (1-e_if_i)/2}=\cos\frac{\theta _i }{2} +
  \sqrt{-1}\sin\frac{\theta_i}{2}(1-e_if_i)}.
$$ 
The action of $\chi$ on $(\pm I, \pm I)$ is  similar. 

Similarly, setting $\theta_i=\pi$, 
$$
\mu: \left ( \ep_1 i^p  \prod_{i=1} ^p (1-e_i f_i)  , \ep_2
  i^q \prod _{j=1} ^q  (1-e_{p+j} f_{p+j}) \right )  \mapsto e^{i\pi (a_1+\cdots +a_p +b_1+\cdots + b_q)} =  i^{ 2(a_1+\cdots+a_p +b_1+\cdots +b_q)}.
$$

\end{proof}
\end{lemma}

Then it is clear that 
\begin{itemize}
\item $\chi$ factors through $SO(c,d)$ iff $a_i, b_j\in \bbZ$;
\item $\chi$ factors through $Spin(c,d)$ if $a_i,b_j \in\bbZ $ or $a_i, b_j\in \bbZ+\frac{1}{2}$;
\item $\chi$ is genuine for $\tu{Spin}(c,d)$ if and only if $a_i\in \bbZ$ and $b_j\in \bbZ+\frac{1}{2}$, or $a_i\in \bbZ + \frac{1}{2}$ and $b_j\in \bbZ$.
\end{itemize}

 Denote the set of genuine central characters of $Z(\tu{G})$ by $\prod_g (Z(\tu{G}))$. 
The next lemma characterizes  $\prod_g (Z(\tu{G}))$.

\begin{lemma}
Let $\mu_j$ be the $\tu{K}$-type parametrized by its highest weight:
\begin{equation}
\begin{aligned}
\mu_1= ( \underbrace{1/2 , \dots, 1/2 }_p \mid \underbrace{ 0, \dots, 0} _q), \ & \mu_2 = ( \underbrace{0 , \dots, 0 }_p \mid \underbrace{ 1/2, \dots, 1/2} _q),\\
\mu_3= ( \underbrace{1/2 , \dots, 1/2,-1/2 }_p \mid \underbrace{ 0, \dots, 0} _q), \ & \mu_4= ( \underbrace{0 , \dots, 0}_p \mid \underbrace{ 1/2, \dots, 1/2, -1/2} _q).
\end{aligned}
\end{equation}
Let $\chi_j$ be the restriction of the highest weight of $\mu_j$ to $Z(\tu{G})$. 
\begin{itemize}
\item[(a)]  If $c=2p+1$, $d=2q+1$, then $\prod_g (Z(\tu{G})) =\{\chi_1, \chi_2\}$.

\item[(b)] If $c=2p,$ $d=2q$, then $\prod_g (Z(\tu{G}))=\{\chi_1, \chi_2,\chi_3, \chi_4\}$. 
\end{itemize}
Moreover, given any $\tu{K}$-type $\mu$, the central character  of $\mu$ is $\chi_j$ iff $\mu-\mu_j$ is in the root lattice of $D_n.$ 
\begin{proof}
The proof easily follows from \eqref{center-spincover}.
\end{proof}
\end{lemma}

\subsection{Cartan subgroups of $\mathbf{\wti G}$}

The $\theta$-stable Cartan subgroup $$\tu{H} ^{r^+,
    r^-,m,s}=\wti T^{r^+,r^-,m,s}\cdot A^{r^+,r^-,m,s}$$  of $\tu{G}$ is the centralizer of $\fk h ^{r^+,r^-,m,s}$ in $\tu{G}$. We write 
$\fk h ^{r^+,r^-,m,s} = \fk t ^{r^+,r^-,m,s}+\fk a ^{r^+,r^-,m,s}$. We have that 
$$
\tu{H} ^{r^+, r^-,m,s}= \tu{T }^{r^+,r^-,m,s}  \exp _{\tu{G}} \fk a ^{r^+,r^-,m,s}.
$$
The goal is to  compute $\tu{T }^{r^+,r^-,m,s} = Z_{\tu{K}} (\fk h ^{r^+,r^-,m,s} )$.

We first do this for $\tu{G}=\tu{Spin} (2,2)$ in detail. As before, the orthogonal space $V=V^+ \oplus V^- $ has basis 
$$
e_1, f_1 ;\ e_2, f_2
$$
with $e_1, f_1\in V^+$, $e_2,f_2\in V^-$.

There are four conjugacy classes of Cartan subalgebras of $\fk{so}(2,2)$, denoted
$$\fk h ^{1,1,0,0}, \fk h^{0,0,1,0} _I, \fk h^{0,0,1,0} _{II},  \fk h^{0,0,0,2}.$$

Note that an element in $\tu{K}=Spin(2)\times Spin (2)$ is of the form $(e^{ i \theta_1 h(\ep_1) },  e^{i \theta_2 h(\ep_2) })$, where 
\begin{equation}\label{exp-h}
e^{i \theta_j h(\ep_j) } = \exp [i \frac{\theta_j}{2} (1-e_jf_j)] = \cos\frac{\theta _j }{2}+i \sin \frac{\theta_j}{2} (1-e_jf_j).
\end{equation}

\begin{itemize}
\item $\tu{H}^{1,1,0,0} = Z_{\tu{K}} ( h (\ep_1) ) \cap  Z_{\tu{K}} ( h (\ep_2) ) = \{ (e^{ i \theta_1 h(\ep_1) },  e^{i \theta_2 h(\ep_2) }) \} \cong (S^1)^2$. 

\item The cases $\tu{H}_{I}^{0,0,1,0}$ and $\tu{H}_{II}^{0,0,1,0}$ are similar, so we do the former one only. We shall calculate $Z_{\tu{K}} (\fk a)$.
Let $(e^{i\theta_1h(\ep_1)}, e^{ i\theta_2 h(\ep_2)} )\in Z_{\tu{K}}(\fk a)$. Then

\begin{eqnarray*}
 X(\ep_1+\ep_2)+X(-\ep_1-\ep_2) &=\Ad e^{\theta_1h(\ep_1)+\theta_2 h(\ep_2)}\big[ X(\ep_1+\ep_2)+X(-\ep_1-\ep_2)\big]\\
& = e^{i(\theta_1+\theta_2)}X(\ep_1+\ep_2) +e^{i(-\theta_1-\theta_2)}X(-\ep_1-\ep_2).
\end{eqnarray*}

This gives  that $\theta_1+\theta_2=2k\pi$, $k\in \bbZ$.
Therefore, 
\begin{eqnarray*}
(e^{i\theta_1h(\ep_1)}, e^{ i\theta_2 h(\ep_2)} )&=& (e^{i\theta_1h(\ep_1)}, e^{ i (2k\pi- \theta_1) h(\ep_2)} )\\
&=&\big (  \cos\frac{\theta _1 }{2}+i \sin \frac{\theta_1}{2} (1-e_1f_1), \cos\frac{\theta _2 }{2}-i \sin \frac{\theta_2}{2} (1-e_2f_2) \big ) \cdot (I , e^{i2k\pi h(\ep_2)})\\
&=&\big (  \cos\frac{\theta _1 }{2}+i \sin \frac{\theta_1}{2} (1-e_1f_1), \cos\frac{\theta _2 }{2}-i \sin \frac{\theta_2}{2} (1-e_2f_2) \big ) \cdot (I , \pm I) 
\end{eqnarray*}
\item For $\tu{H}^{0,0,0,2}$, let $(e^{i\theta_1h(\ep_1)}, e^{ i\theta_2 h(\ep_2)} ) \in Z_{\tu{K}}(\fk a)$. We have the relations
\begin{eqnarray*}
 X(\ep_1+\ep_2)+X(-\ep_1-\ep_2) &=\Ad e^{\theta_1h(\ep_1)+\theta_2 h(\ep_2)}\big[ X(\ep_1+\ep_2)+X(-\ep_1-\ep_2)\big]\\
& = e^{i(\theta_1+\theta_2)}X(\ep_1+\ep_2) +e^{i(-\theta_1-\theta_2)}X(-\ep_1-\ep_2),
\end{eqnarray*}
and 
\begin{eqnarray*}
 X(\ep_1-\ep_2)+X(-\ep_1+\ep_2) &=\Ad e^{\theta_1h(\ep_1)+\theta_2 h(\ep_2)}\big[ X(\ep_1-\ep_2)+X(-\ep_1+\ep_2)\big]\\
& = e^{i(\theta_1-\theta_2)}X(\ep_1-\ep_2) +e^{i(-\theta_1+\theta_2)}X(-\ep_1+\ep_2).
\end{eqnarray*}
This gives that $\theta _1 -\theta_2=2 k\pi, \ \theta_1+\theta _2 =2 l \pi$, $k,l \in \bbZ$ and hence we write $$\theta_1= ( k+l )\pi, \theta_2 = (k-l)\pi. $$
Therefore, 
\begin{eqnarray*}
&&(e^{i\theta_1h(\ep_1)}, e^{ i\theta_2 h(\ep_2)} )= (e^{i (k+l)\pi h(\ep_1)}, e^{ i  (k-l)\pi   h(\ep_2)} )\\
&=&\big (  \cos\frac{  (k+l)\pi }{2}+i \sin \frac{(k+l)\pi}{2} (1-e_1f_1), \cos\frac{(k-l)\pi}{2}+i \sin \frac{(k-l)\pi}{2} (1-e_2f_2) \big ).
\end{eqnarray*}
This gives eight elements in $Z_{\tu{K}}(\fk a)$: 
$$
(\pm I, \pm I), (\pm i(1-e_1f_1), \pm i(1-e_2f_2)).
$$

\end{itemize}

  Recall $\tu{G}=\tu{Spin}(c,d)$, $c=2p, d=2q$  or $c=2p+1, d=2q+1$. We change the  orthonormal basis to
$$
v_i, w_i, v^+ ; \ v_j, w_j, v^-, \quad 1\le i\le p, \ p+1\le j \le p+q
$$
 for $V=V^+ \oplus V^-$, where 
 $$
 v_i = \frac{e_i+f_i}{\sqrt{2}}, \ w_i= \frac{e_i-f_i}{\sqrt{-2}},
 $$
 and the same relations hold for $v_j,w_j,e_j,f_j$.
Again, when $c, d$ are even, the corresponding terms of $v^+$ and $v_-$ are missing.
 
Define  a finite  subgroup $F^{r^+,r^-, m ,s}$ of $\tu{K}$ as follows. 
 \begin{itemize}
 
 \item[(1)] When $m=0$ and $s= 0 $ or 1, define $F^{r^+,r^-, m ,s}=1$.
 \item[(2)] When $s>1$, define $F^{r^+,r^-, m ,s}$ to be the subgroup generated by the elements of order four of $\tu{K}$,
 
 \begin{eqnarray*}
&\{   (\pm v_i v_k, \pm v_{i+p}v_{k+p}   ),   (\pm w_i w_k, \pm w_{i+p}w_{k+p}   ), \quad r^++m+1 \le i< k \le p \\
   &(\pm v_l w_k, \pm v_{l+p}w_{k+p}),   \quad r^++m+1 \le l, k  \le p \\
 &  (\pm v_i v^+, \pm v_{i+p}v^-) , (\pm w_i v^+, \pm w_{i+p } v^-) , \quad r^++m+1 \le i\le p \}.
 \end{eqnarray*} 
 
 In this case, 
 $
 |F^{r^+,r^-, m ,s} |= 2^{s-1}\cdot 4.
 $
 It is an extension of $\bbZ ^{s-1} _2$ of $\bbZ_2 \times \bbZ_2$, the center of $Spin(c)\times Spin(d)$.
\item[(3)] When $m\neq 0$ and $s=0$ or 1, define $F^{r^+,r^-, m,s}=\{ ( I  ,\pm I )\}$. In this case $|F^{r^+, r^-, m ,s}|=2$. 

 \end{itemize}

\begin{lemma} \label{Cartan-decomp}
The Cartan subgroup $\wti H^{r^+,r^-,s,m}$ has the direct product decomposition
$$
\tu{H}^{r^+, r^-, m ,s}=F^{r^+,r^-, m ,s}\times \exp _{\tu{G}} ( \fk h^{r^+,r^-, m , s} ). 
$$ Thus, the number of components of $\tu{H}^{r^+, r^-, m ,s}$ 
is $|F^{r^+,r^-, m ,s}|$,  listed above. 
\end{lemma}

\begin{cor}
$\tu{H}^{r^+, r^-, m ,s}$ is abelian iff $s< 3$.
\begin{proof}
Since $ \exp _{\tu{G}} ( \fk h ^{r^+,r^-, m , s} ) \subset \tu{H}_0\subset Z(  \tu{H} ^{r^+,r^-, m , s})$, to determine whether $ \tu{H} ^{r^+,r^-, m , s}$ is abelian, we just need to look at $ F ^{r^+,r^-, m , s}$ by Lemma \ref{Cartan-decomp}.

When $s=0$, $|F ^{r^+,r^-, m , s}|=1$ or 2. So $F ^{r^+,r^-, m , s}$ is obviously abelian. 

When $s=2$, $F ^{r^+,r^-, m , s}=\{ (\pm I, \pm I), (\pm v_1w_1, \pm v_2w_2) \}$, and is clearly abelian. 
 
When $s\ge 3$, there exist sets of orthonormal vectors $\{  v^+, w^+, u^+\} \subset V^+$ and  $\{  v^-, w^-, u^-\} \subset V^-$ such that $(v^+w^+, v^-w^-), (w^+u^+, w^-u^-)\in  \tu{H} ^{r^+,r^-, m , s}$. Then 
$v^+w^+ w^+u^+ = v^+ u^+$, whereas $w^+u^+v^+w^+=u^+v^+ =-v^+u^+$. Therefore $ \tu{H} ^{r^+,r^-, m , s}$ is not abelian.

\end{proof}

\end{cor}

\subsection{Regular characters} See \cite{AT} or \cite{RT} for more detail in this section. Suppose $G$ is a real reductive group (possibly nonlinear). 

\begin{definition}
A regular character of $G$ is a triple $\gamma= (H,\Gamma,\la)$
consisting of a $\theta$-stable Cartan subgroup $H$, an irreducible
representation $\Gamma$ of $H$, and $\la\in \fk h^*$, satisfying the
following conditions.  

\begin{itemize}
\item[(a)] $ \langle \la,\alpha ^{\vee}\rangle \in \bbR ^{\times} $ for all imaginary roots $\alpha$;
\item[(b)] $d\Gamma =\lambda +\rho _i (\la)-2\rho _{i,c}(\lambda)$;
\item[(c)] $\langle \la,\alpha^{\vee}\rangle\neq 0\ \forall \alpha \in \Delta$. 
\end{itemize}
\end{definition}

The group under consideration is $\tu{G}=\tu{Spin}(c,d)$. In this case we
will write $\gamma=(\tu{H} , \Gamma, \la)$. We say that  $\gamma$ is
genuine if $\Gamma$ is a genuine representation of $\tu{H}$. 

Let $I(\gamma)$ denote that standard module corresponding to the
parameter $\gamma$, and let $J(\gamma)$ denote the unique irreducible
quotient of $I(\gamma)$. 

\begin{prop}[\cite{AT}] \
\begin{itemize}
\item[(a)] Let $\tu{H}$ be a Cartan subgroup of $\tu{G}$. Every
  representation $\Gamma$ of $\tu{H}$ is parametrized by a genuine
  character of $Z(\tu{H})$, i.e. $\Gamma |_{Z(\tu{H})}$, the
  restriction of $\Gamma$ to $Z(\tu{H}).$ 
\item[(b)] $Z(\tu{H})=Z(\tu{G}) \tu{H^0}$, so a genuine character of
  $Z(\tu{H})$ is determined by its restriction to $Z(\tu{G})$ and its
  differential. Moreover, if $\gamma=(\tu{H}, \Gamma, \la)$ is a
  genuine character of $\tu{G}$, then $\gamma$ is determined by $\la$
  and $\Gamma |_{Z(\tu{G})}$. 
\end{itemize}
\end{prop}


Let $\Lambda$ be a genuine representation of the fundamental Cartan
subgroup $\tu{H}$ with $d\Lambda = \la$.  When $c=2p,d=2q$, $\tu{H}=\tu{H}^{p,q,0,0}$; when $c=2p+1, d=2q+1$, $\tu{H}=\tu{H}^{p,q,0,1}$.

\begin{lemma}

Given $\tu{G}=\tu{Spin}(c,d)$, $c+d=2n$.
\begin{itemize}
\item[(a)] Suppose that $n\in 2\bbZ$ ($c=2p, d=2q$). Then the
  infinitesimal character of any genuine discrete series
  representation of $\tu{G}$ is conjugate to the form  
$$ 
(a_1, \dots, a_p\mid b_1,\dots , b_q),
$$ 
with 
$a_i\in \bbZ, b_j\in \bbZ+\frac{1}{2}$, or $a_i\in \bbZ +\frac{1}{2},b_j\in \bbZ$,
 and
$a_1>\cdots> |a_p|\ge 0, b_1>\cdots> |b_q|\ge 0$.  
\item[(b)] Suppose that $n\in 2\bbZ +1$ ($c=2p+1, d=2q+1$). Then the
  infinitesimal character of any genuine fundamental series
  representation of $\tu{G}$ is conjugate to the form  
$$ 
(a_1, \dots, a_p\mid b_1,\dots , b_q\mid x),
$$ 
with $a_i\in \bbZ, b_j\in \bbZ+\frac{1}{2}$ or $a_i\in \bbZ
+\frac{1}{2}, b_j\in \bbZ$ and $a_1>\cdots>a_p\ge 0, b_1>\cdots>b_q\ge
0$, and $x$ is either in $\bbZ$ or $\bbZ+\frac{1}{2}$.
\end{itemize}
\end{lemma}
Then the following corollary easily follows.
\begin{cor}\label{group-infchar}
The following groups are the only ones which  admit a representation with 
infinitesimal character defined in (\ref{infchar-orbit}).

\begin{itemize} 
\item[(a)]  $\tu{G}  = \tu{Spin} (2p,2q)$, with $p=n-k-1, \ q=k+1$; 
\item[(b)] $\tu{G}=\tu{Spin}(2p+1,2q+1)$, with $p+1=n-k-1,\ q=k+1$;
\item[(c)] $\tu{G}=\tu{Spin}(2p+1,2q+1)$, with $p=n-k-1,\ q+1=k+1$.
\end{itemize}

\end{cor}

By Corollary \ref{group-infchar}, given $\la$ in (\ref{infchar-orbit}),
we will deal with the eight cases listed in Section \ref{ss:1}.

\subsection{Coherent Continuation Action}
The number of representations with associated cycle $\calO$
  equals the multiplicity of the $sgn$ representation of $W(\la)$ in the coherent
  continuation representation. The orbit $\calO$ is the minimal orbit
  which can occur for the given infinitesimal character, and this
  corresponds to the $sgn$ representation. So we first study the coherent continuation action for the group $\tu{G}$.

The formulas of the coherent continuation action can be derived from those of the action of Hecke operators.
As in \cite{RT}, given $\la$ as in (\ref{infchar-orbit}), we define a family of infinitesimal character $\calF (\la)$ including $\la$. Note that every $\la ' \in \calF(\la)$ can be  indexed by some $w\in W/W(\la)$.  
 Write $\calB _{\la ', \chi}$ for the set of equivalence classes of standard representation parameters with infinitesimal character $\la ' \in\calF(\la)$ and a fixed central character $\chi$ of $\tu{G}$, and $$\calB := \coprod \limits _{\la '\in \calF(\la) , \chi \in   \widehat{Z(\tu{G}) }  } \calB_{\la ' , \chi}.$$ 
 As we will see later that the coherent continuation action is closely related to the cross action, we may use $\calF(\la)$ to define the cross action of $W$ on $\calB$, denoted $w\times \gamma$ for $w\in W$, $\gamma\in\calB$, as shown in \cite{RT}. In fact, fixing an infinitesimal character $\la ' \in \calF(\la)$ and a central character $\chi$, $W(\la)$ acts on $\calB_{\la ',\chi}$ by the cross action.
 
 We set $\calM = \bbZ [u^{ \frac{1}{2}}, u^{-\frac{1}{2}} ] [\calB]$.
We fix the abstract infinitesimal character $\la_a\in \calF(\la)$ corresponding to the   positive root system $\triangle ^+:=\triangle ^+ _a (\frakg, \frakh^a)$ (where $\frakh^a$ is an abstract Cartan subalgebra of $\frakg$) and the set of simple roots $\prod_a \subset \triangle ^+ _a$. For $s=s_{\alpha}$ with $\alpha\in \prod _a$, the action of $T_s$ on $\gamma\in \calM$ is defined in Section 9 of \cite{RT}. 

On the other hand, we consider $\triangle (\la)$, the integral root system for $\la$ and the integral Weyl group $W(\la)$. 
As we take the infinitesimal character $\la$ in (\ref{infchar-orbit}), the integral root system for $\la$ is $\Delta (D_{n-k-1} \times D_{k+1})$, and due to Corollary \ref{group-infchar}, this is 
\begin{equation}
\Delta(\la)=
\begin{cases}
\Delta(D_p)\times \Delta (D_q)& \text{ if $c=2p, d=2q$,}\\
\Delta (D_{p+1})\times \Delta (D_q) \text{ or } \Delta (D_p) \times \Delta (D_{q+1})& \text{ if $c=2p+1, d=2q+1$.}
\end{cases}
\end{equation}
Also we choose $\Pi(\la)$ to be a set of simple roots for $\Delta(\la)$. 

Given $\alpha\in \Pi(\la)$, we decompose $s_{\alpha}=s_{\alpha_1}\cdots s_{\alpha_m}$
with $\alpha_j\in \Pi_a$.  Replace $T_{s_{\beta}}$ with $T_{\beta}$ for each root $\beta$ for simplicity.

Then
 
\begin{equation} \label{hecke}
\begin{aligned}
T_{\alpha} (\gamma) &=T_{\alpha_1} \cdot  T_{\alpha_2}\cdot \cdots \cdot T_{\alpha_m} (\gamma) \\
&= p_1(u) \cdots p_m (u)  s_{\alpha_1}\times (s _{ \alpha _2}  \times \cdots (s_{\alpha_m }\times \gamma ) ) \\ 
&+ ( \text{terms from more split Cartan subgroups}),
\end{aligned}
\end{equation}

where   $p_j (u) \in \mathbb{Z}[u, u^{-1}]$.

By \cite{V1}, we can define the coherent continuation action of $W(\la)$ on $\bbZ[\calB]$, denoted $w\cdot \gamma$, with $w\in W(\gamma), \gamma\in \calB$, as follows.
  For $s_{\alpha}\in W(\la)$ with $\alpha\in \Pi(\la)$,
  $$ 
  s_{\alpha}\cdot \gamma := -T_{s_{\alpha}} (\gamma)|_{u=1}, \quad\text{ with each term $\delta$ on the right side multiplied by $(-1) ^{l (\gamma) -l (\delta)}$,  }
  $$
where  $l$ is a length function defined on parameters and it can be looked up in \cite{V1}.

Therefore, from each step $T_{\alpha_j}$ in (\ref{hecke}), we may define 
$$s_{\alpha_j} \cdot \delta = -T_{  \alpha_j } (\delta)| _{u=1}, \text{ if }  \alpha_ j  \text{ is real or imaginay for }  \delta$$ and 
$$s_{\alpha_j} \cdot \delta =T_{  \alpha_j } (\delta)|_{u=1}, \text{ if }  \alpha_j \text{ is  complex for } \delta.$$
Let $m({\gamma}, s_{\alpha})$ be the number of occurrences of imaginary roots in $ \{ \alpha_j , 1\le j \le m  \}$. An easy calculation shows that 
\begin{equation} \label{key}
s_{\alpha}\cdot \gamma = (-1)^{m (\gamma,s_{\alpha}) } s_{\alpha}\times \gamma +\text{(terms from more split Cartan subgroups)}.
\end{equation}

Now fix a block $\calB _{\la,\chi}$ of regular characters of $\tu{G}$, then $W(\la)$ acts on
$\bbZ[\calB _{\la,\chi}]$ by the  coherent continuation action, since  $w\times \gamma \in \calB _{\la,\chi}$ for all $w\in W(\la), \gamma\in \calB$. Due
to the reason stated in the beginning of the section, the goal is to compute 
$[sgn_{W(\la)}   :  \bbZ[\calB_{\la,\chi}]]$, the multiplicity of the sign representation in $\bbZ[\calB _{\la,\chi}]$ when considered as $W(\la)$-representations.

Notice that two $\lambda$-regular characters $\gamma _i =(\widetilde{H_i}, \Gamma _i, \overline{\gamma_i})$ and $\gamma _j =(\widetilde{H_j}, \Gamma _j, \overline{\gamma_j})$ from $\mathcal{D}$ are in the same cross action orbit if and only if $\widetilde{H_i} = \widetilde{H_j}$. We enumerate the Cartan subgroups of $\widetilde{G}$ as $\{
\widetilde{H_1}, \cdots , \widetilde{H_l}\}$, and pick a  regular
character $\gamma _j$ specified by $\widetilde{H_j}$, then $\{\gamma
_1, \cdots, \gamma _l \}$ is a set of representatives of the cross
action orbits of $W(\lambda)$ on  $\mathbb{Z}[\mathcal{B_{\la,\chi}}]$.  

Let $W_{\gamma _j} = \{ w \in W(\lambda) \thinspace | \thinspace
w\times \gamma_j =\gamma _j  \}$ be the cross stabilizer of $\gamma _j
$ in $W(\lambda)$. 
Then we have the following Proposition.

\begin{prop}  \label{coh:decomp}
$\mathbb{Z}[\mathcal{B _{\la,\chi}}] \simeq \bigoplus _j  Ind ^{W(\lambda)}
_{W_{\gamma _j}   }  (\epsilon _j)$, where $\epsilon _j$ is a
one-dimensional representation of $W_{\gamma _j}$ such that for $w\in
W_{\gamma_j}$, $w \cdot \gamma _j = \epsilon _j (w) \gamma _j +$ other
terms from more split Cartan subgroups. 
\begin{proof} 
This can be easily proved by the formulas given in \cite{RT} and (\ref{key}).
\end{proof}

\end{prop}

By Proposition \ref{coh:decomp} and Frobenius reciprocity, the multiplicity of $sgn_{W(\lambda)}$ in $\mathbb{Z} [\calB_{\la,\chi}]$ is  
$[sgn_{W(\lambda)}:   \mathbb{Z}[\calB_{\la,\chi}]] =[sgn_{W(\lambda)} | _{W_{\gamma_j}} : \epsilon _j   ]$, which is equal to 0 or 1, 
since $sgn_{W(\lambda)} | _{W_{\gamma_j}} $ is one-dimensional. This means that we have reduced our goal to count the number of $\gamma _j$'s 
making $[sgn_{W(\lambda)} | _{W_{\gamma_j}} : \epsilon _j   ] = 1$.
 Equivalently, we
calculate the number of $\gamma_j$ such that  

\begin{equation}  \label{cond-star}
sgn_{W(\la)} |_{W_{\gamma_j}} = \ep_j.
\end{equation}
Due to Proposition \ref{coh:decomp} and (\ref{cond-star}), we have to
analyze $W_{\gamma_j}$ and $\ep_j$ for each $\gamma_j$. 

\subsubsection*{Some notation for Weyl group elements}

Let $\rho ' = (\rho_1, \dots, \rho_n)=w\rho$ for some $w\in W=W(D_n)$. We define the notations for elements in $W$ as follows. For $i\neq j $, write 
$ s_{i,j} = s_{\ep_i-\ep_j} $ and $s_{\ovl{i,j}}= s_{\ep_i+\ep_j} $ to be the reflections with respect to the corresponding roots. Moreover, for $1\leq i  <  j \leq n$, 
let $t_{i,j}$ and $t_{\ovl{i,j}}$ denote the corresponding Weyl group elements such that 

\begin{eqnarray*}
 t_{i,j} (\rho' ) =
\begin{cases}
 \mbox{ interchanging  the numbers $i$ and $j$ in } \rho'& \mbox{ if } i j >0 \\
 \mbox{ interchanging and changing both signs of  the numbers $i$ and $j$ in } \rho'& \mbox{ if } ij <0
 \end{cases}\\
  t_{\ovl{ i,j } } (\rho' ) =
\begin{cases}
 \mbox{ interchanging  the numbers $i$ and $j$ in } \rho'& \mbox{ if } i j <0 \\
 \mbox{ interchanging and changing both signs of  the numbers $i$ and $j$ in } \rho'& \mbox{ if } ij >0
 \end{cases}
\end{eqnarray*}
 
For  $0< j \leq n -1$, let $t_{0,j}$ and $t_{\ovl{0,j}}$ denote the corresponding Weyl group elements such that 
 \begin{eqnarray*}
 t_{0,j } (\rho' ) =
\begin{cases}
 \mbox{ interchanging  the numbers 0 and $j$ in } \rho'& \mbox{ if } j >0 \\
 \mbox{ interchanging and changing the sign of  the number $j$ in } \rho'& \mbox{ if } j <0
 \end{cases}\\
  t_{\ovl{ 0,j } } (\rho' ) =
\begin{cases}
 \mbox{ interchanging  the numbers 0 and $j$ in } \rho'& \mbox{ if } j <0 \\
 \mbox{ interchanging and changing the sign of  the number $j$ in } \rho'& \mbox{ if } j >0
 \end{cases}
\end{eqnarray*}
 
Note that $s_{i,j}$ and $s_{\ovl{i,j}}$, $1\le i\neq j \le n$, are fixed Weyl groups elements, whereas $t_{i,j}$ and $t_{\ovl{i,j}}$, $0\le i\neq j\le n-1$,  are dependent of $\rho'$. We have following  advantages of using the notation for $t$'s: 

(a) $t_{i,i+1}$ (or $t_{ \ovl{i,i+1}}$) is a simple reflection no matter what $\rho'$ it is acting on.  

(b) Let $\delta\in \calB$ be a parameter in the chamber of $\rho_{\delta} =w_{\delta}\rho$ for some $w_{\delta}\in W$. Then $t_{i,j}$ (or $t_{\ovl{i,j}}$) is integral for $\delta$ if and only if $i$ is in the $k$-th position of $\rho_{\delta}$ and $j$ is in the $l$-th position in $\rho_{\delta}$, and $k-l\in 2{\bbZ}$.

\subsubsection*{} 
Given a parameter $\gamma = (\tu{H}^{r^+,r^-, m ,s}, \Gamma, \la) \in\calB_{\la,\chi}$. Write $s=2s'$ if $n$ is even; $s=2s'+1$ if $n$ is odd.  Then $\gamma$ can be expressed as follows: 

\begin{equation} \label{para-even}
 \footnotesize(  \underbrace{a_1,\dots, a_{r^+}; b_1,\dots, b_{r^-} }_{r^++r^-}\mid \underbrace{\underline{a_{r^++ 1}\ b_{r^-+1}} , \dots, \underline{a_{r^++m}\ b_{r^-+m}} } _{2m}\mid
\underbrace{ a_{r^+ +m+1},\dots, a_{r^++m+s'}, b_{r^-+m+1},\dots, b_{r^-+m+s'} , x} _s), 
\end{equation}
where $a_i\in \bbZ, \ b_i\in \bbZ +\frac{1}{2}$ (or the other way
round), $\ep _i\pm \ep_j$ are imaginary for $1\le i<j \le r^+ + r^-$;
 $\ep_i\pm \ep_j $ are real for $r^+ + r^- + 2m+1\le  i < j\le n$; for $r^++r^-+1\le
i<j\le r^++r^- +2m$, $\ep _i - \ep_j $ is imaginary and $\ep_i +\ep _j$ is
real.  The coordinate $x$ is missing when $n$ is even, and it is either integral or half integral. 


We compute the cross stabilizer for such parameters. 

\begin{lemma}\label{cross-stab}
Let $\gamma$ be a parameter given as in (\ref{para-even}). Then 
from \cite{AT},
\begin{eqnarray}\label{crossstab}
W_\gamma=  W^C (\la)^{\theta} \rtimes (W^r(\la) \times W^i(\la)).
\end{eqnarray}
Furthermore, each group in (\ref{crossstab}) is explicitly expressed as follows:
\begin{eqnarray}\label{crossstab-1}
W^C (\la)^{\theta} &= &[W(\triangle (D_m\times D_m) )  \times (\bbZ_2\times \bbZ_2)^* ] \rtimes \bbZ _2 ^{\bullet}, \\
 W^ r (\la) & = & \begin{cases}  W(D_{s'})\times W(D_{s'}) & \text{ if $n=p+q$}   \\    W(D_{s'})\times W(D_{s'+1}) & \text{ if $n=p+q+1$}\end{cases} \\
 W^i(\la) &=& W(D_{ r^+})\times W(D_{r^-}).
\end{eqnarray}
In (\ref{crossstab-1}),
\begin{equation}
\begin{aligned}
(\bbZ_2\times\bbZ_2 )^* 
&=\begin{cases}
\bbZ_2\times \bbZ_2 & \text{ if $r^+,r^-$ are both nonzero, and $s\ge 2$,}\\
1& \text{ otherwise;}
\end{cases}\\
\bbZ_2 ^{\bullet}& =
\begin{cases}
\bbZ_2& \text{ if  $m\neq 0$ and one of conditions (i), (ii), (iii) below happens}. \\
1 & \text{ otherwise.}
\end{cases}
\end{aligned}
\end{equation}
Here are the conditions: (i) $s\ge 2$; (2) both $r^+$ and $r^-$ are nonzero; (3) $r_-=0$, $r_+\ge1, s=1$ with the coordinate $x$ is half-integral.

When $(\bbZ_2\times\bbZ_2 )^* =\bbZ _2\times \bbZ_2$, the generators can be taken to be $s_{1,  r^++r^-+2m+1} s_{\ovl{ 1,r^++r^-+2m+1  } } $ and
 $s_{r^++1 , r^++r^-+2m+s'+1} s_{\ovl{r^++1 , r^++r^-+2m+s'+1}} $;
 When $\bbZ_2 ^{\bullet}=\bbZ_2$, the generator can be taken to be:\\
 \begin{itemize} 
\item[(i)] $s_{r^++r^-+1, r^++r^-+2m+1} s_{ \ovl{r^++r^-+1, r^++r^-+2m+1}} s_{r^++r^-+2, r^++r^-+2m+2} s_{\ovl{r^++r^-+2, r^++r^-+2m+2 }}$,   or 
\item[(ii)] $ s_{1, r^++r^- +1} s_{\ovl{1, r^++r^- +1 }} s_{r^++1, r^++r^-+2}  s_{\ovl{r^++1, r^++r^-+2 }}$, or 
\item[(iii)] $s_{1,r^++1}s_{\ovl{1,r^++1}} s_{r^++2,r^+2m+1} s_{ \ovl{r^++2,r^+2m+1} }.$
 \end{itemize}
\end{lemma}

\begin{remark}\label{c-stab} \
\begin{itemize}
\item[(1)] $W^r(\la)$ is nontrivial iff $s\ge 3$.
\item[(2)] $W^i(\la)$ is nontrivial iff $r^+\ge 2$ and $r^-\ge 2$.
\end{itemize}
\end{remark}
Now we are ready to prove the following lemma.

\begin{lemma} \label{ep-rep}
Let $w\in W_{\gamma_j}$. Retain the notation in Proposition \ref{coh:decomp}  and Lemma \ref{cross-stab}. We have the analysis for $\ep_j$ corresponding to  $\gamma_j$:
\begin{itemize}
\item[(a)] If $w\in W^i(\la)$, then $\ep _j(w)=sgn(w)$.
\item[(b)]   If $w\in W^r (\la)$, then $\ep_j (w)=1$.
\item[(c)] If $w\in W(\triangle (D_m\times D_m) ) \subset W^C(\la)^{\theta}$, then $\ep _j (w)=1$.
\item[(d)] In the case that $(\bbZ_2\times \bbZ_2)^*=\bbZ_2\times \bbZ_2$, if $w$ is a generator of one of the $\bbZ_2$ factors, then $\ep_j(w)=1$. 
\item[(e)] In the case that $\bbZ_2 ^{\bullet} =\bbZ_2$,  if $w$ is the generator of the $\bbZ _2$ factor, then $\ep_j (w)=-1$.
\end{itemize}
\begin{proof}
By the comment given before Proposition \ref{coh:decomp}, we may choose any $\gamma_j$ specified by $\tu{H}_j$ to simplify the  computation.
In each case, we take a generator $w\in W_{\gamma_j}$ and decompose $s_{\alpha}= s_{\alpha_1}\cdots s_{\alpha_l}$ with $\alpha_i\in \Pi_a$. 
By (\ref{key}), we just need to count $m(\gamma, w)$, the number of occurrences of imaginary roots in $\{\alpha_1,\dots,\alpha_l\}$ (with respect to $\gamma$).

In case (a), we may choose the parameter $\gamma_j$ to be 
$$
(0 , 1,2 ,\dots, r^+-1 ; \frac{1}{2} , \frac{3}{2}, \dots, \frac{1}{2}+r^--1 \mid \cdots \cdots \mid \cdots \cdots),
$$
which is in the chamber of $\rho= (0, 2,4 ,\dots ; 1, 3,5, \dots \mid \cdots \mid \cdots )$.
A generator in $W^i (\la)$ is of the form $s_{i , i+1}$ with $1\le i \le r^+ -1 $ or $r^++1\le  i \le r^++r^--1$, 
$s_{\ovl{1,2}}$, or $s_{\ovl{r^++1,r^++2}}$. 

We treat $w= s_{i ,i+1}$ only, since the rest will be similar.  Decompose 
$$
\begin{aligned}
s_{i ,i+1}  &=   t_{k,k+2} & \text{ for some }  0\le k\le r^++r^- -1 \\
                                &=   t_{k+1, k+2} t_{k, k+1} t_{k+1, k+2}. & 
\end{aligned}
$$
It is clear that each $t_{i,l}$ is a simple reflection through an imaginary root, and hence $m(\gamma_j, s_{\alpha})=3.$ Therefore we conclude
that $\ep_j (w)=sgn(w)$ for  $w \in W^i (\la)$.

Case (b) is similar. We choose the parameter $\gamma_j$ to be 
$$
(\cdots\mid \cdots \mid  0, 1, \dots, \frac{1}{2}, \frac{3}{2},\dots),
$$
which is in the chamber of $\rho= (\dots\mid \dots\mid 0, 2,\dots, 1,3,\dots)$.
Everything is the same as in case (a), except that each $t_{il}$ is real for the parameter which is acted, and hence
$m(\gamma_j,s_{\alpha})=0$. Therefore we conclude
that $\ep_j (w)=1$ for  $w \in W^r (\la)$.

In case (c), we choose the parameter $\gamma_j$ to be 
$$
(\cdots \mid \underline{0 \  \frac{1}{2} }\quad   \underline{1 \ \frac{3}{2}} \dots \mid \cdots),
$$
which is in the chamber of $\rho= (\dots   \mid 0, 1, 2, 3, \dots \mid \dots)$.
Take $w$ to  be the Weyl group element such that 
$$
 w: \ \gamma_j \mapsto (\cdots \mid \underline{1 \  \frac{3}{2} }\quad   \underline{0 \  \frac{1}{2}} \dots \mid \cdots),
$$

then $w$ is one of the generators of  $W(\triangle (D_m\times D_m) ) \subset W^C(\la)^{\theta}$. 
We decompose $w$ in terms of $t$'s:
$$
w=t_{02}t_{13}=t_{12}t_{01}t_{12}t_{23}t_{12}t_{23}.
$$
It is easy to check that $m(\gamma_j, w)=2$.  

Another kind of generators in $W(\triangle (D_m\times D_m) )$ is of the form
$$
\ovl{w}: \   \gamma_j \mapsto (\cdots \mid \underline{ -1 \  -\frac{3}{2} }\quad   \underline{-0 \  -\frac{1}{2}} \dots \mid \cdots).
$$
We decompose $\ovl{w}$ in terms of $t$'s:
$$
\ovl w= t_{\ovl{02}}t_{\ovl{13} } = t_{12} t_{\ovl{01}} t_{12} t_{\ovl{10}} t_{10}t_{12} t_{23}t_{12}t_{10}t_{\ovl{10}},
$$
and get $m(\gamma_j, s_{\ovl\alpha} )=2$. Therefore, we conclude that $\ep_j(w)=1$ for $w\in W(\triangle (D_m\times D_m) ) \subset W^C(\la)^{\theta}. $

In case (d), we choose the parameter $\gamma_j$ to be 
$$
(\dots, 0 ; \frac{1}{2} \dots \mid \cdots \mid 1\dots \frac{3}{2}\dots).
$$
The generators of the $\bbZ_2\times \bbZ_2$ factor can be chosen to be 
$$
w_1:\ \gamma_j \mapsto ( \dots, -0 ; \frac{1}{2} \dots \mid \cdots \mid -1\dots \frac{3}{2}\dots)
$$
and 
$$
w_2:\ \gamma_j \mapsto ( \dots, 0 ; -\frac{1}{2} \dots \mid \cdots \mid 1\dots -\frac{3}{2}\dots).
$$
We treat $w_1$ only. 
Decompose 
$$
w_1=t_{02}t_{\ovl{02}} = t_{12}t_{01} t_{\ovl{01}}t_{12}.
$$
It can be checked that $m(\gamma_j, w_1)=0$, and hence $\ep_j(w_1)=1$. Similarly for $w_2$. We conclude that 
$\ep(w)=1$ for $w$ in the $\bbZ_2\times \bbZ_2$ factor.

In case (e), we choose the parameter $\gamma_j$ to be 
$$
(\dots, 0 ; \frac{1}{2} \dots \mid \underline{1 \ \frac{3}{2}}   \dots \mid \cdots) 
$$
or 
$$
(\cdots\mid \underline{1 \ \frac{3}{2}}  \mid 0, \frac{1}{2} \dots).
$$
We treat the first one only.  The generator of the $\bbZ_2$ factor can be chosen to be 
$$
w: (\dots, 0 ; \frac{1}{2} \dots \mid \underline{1 \ \frac{3}{2}}   \dots \mid \cdots)\mapsto   (\dots, -0 ; -\frac{1}{2} \dots \mid \underline{-1 \ -\frac{3}{2}}   \dots \mid \cdots).
$$
Decompose 
$$
w= t_{02}t_{\ovl{02}}t_{13}t_{\ovl{13}} = t_{12}t_{01}t_{12}t_{23}t_{12}t_{23}  t_{12} t_{\ovl{01}} t_{12} t_{\ovl{10}} t_{10}t_{12} t_{23}t_{12}t_{10}t_{\ovl{10}},
$$
and get $m(\gamma_j, w)=1$. So $\ep_j(w)=-1$ for $w$ in the $\bbZ_2$ factor.

\end{proof}
\end{lemma}

By Lemma \ref{ep-rep},  $\gamma_j$ does not 
satisfy (\ref{cond-star}) if and only if $W^r(\la)$ is nontrivial
or $\bbZ_2^{\bullet}=\bbZ_2$ is contained in $W_{\gamma_j}$. Therefore, 
 we have to rule out 
$\gamma_j$ specified by $\tu{H}^{r^+,r^-,m,s}$ satisfying either of the following.
\begin{itemize}
\item[(1)] $s\ge 3$;
\item[(2)] $m\ge 1$ and $r^+\ge 1$, $r^-\ge 1$;
\item[(3)] $m\ge 1$ and $s\ge 2$. 
\item[(4)] $m\ge 1$ and $s=1$, $r^-=0$, $r^+\ge 1$, and $k=(n-3)/2$.
\end{itemize}

Consequently, we have the following lemma.

\begin{lemma}\label{fix-char-no}
Let $n_{\calO} ^{\chi}$ be the number of irreducible representations of $\tu{G}$ with central character $\chi$ attached to $\calO_{\bbC}$ and $\la$. 
\begin{itemize}
\item[(a)] In Case 1, $n_{\calO}^{\chi}= 4$.
\item[(b)] In Case 2 and 3, $n_{\calO}^{\chi} = 3$.
\item[(c)] In Case 4, $n_{\calO}^{\chi} = 2$.
{\item[(d)]  In Case 5 and 6 with  $\tu{G}=\tu{Spin}(2p+1, 2p-1)$, $n_{\calO}^{\chi} = 2$.
\item[(e)] In Case 5 and 6 with  $\tu{G}=\tu{Spin}(2p+1, 2q+1)$, $q<p-1$ , $n_{\calO}^{\chi} = 1$.}
\item[(f)]  In Case 7 and 8, $n_{\calO}^{\chi} = 2$.
\end{itemize} \end{lemma}

\begin{example}
Let $k=1$ and consider the infinitesimal character $$\la=( n-3, \dots, 1,0; 3/2, 1/2).$$
 $\tu{G}$ admits an admissible representation in the following cases as we fix a genuine central character $\chi$ of $\tu{G}$.
\begin{itemize}
\item[(1)] $k=\frac{n}{2}-1=1, \tu{G}=Spin(4,4), \la = (1,0 ; 3/2, 1/2)$. The counting argument gives the 
parameters from $\tu{H}^{2,2,0,0}$,  $\tu{H}^{1,1,0,2} $,  $\tu{H}^{0,0,2,0} _{I}$, $\tu{H}^{0,0,2,0} _{II}$, so $n_{\calO} ^{\chi}=4$. 
\item[(2)] $\tu{G}=Spin (2n-4, 4)$ with $2n-4>4$. The counting argument gives the 
parameters from $\tu{H}^{n-2,2,0,0}$,  $\tu{H}^{n-3,1,0,2} $,  $\tu{H}^{n-4,0,2,0}$, so $n_{\calO} ^{\chi}=3$.

\item[(3)] $\tu{G}=\tu{Spin}(2n-3,3)$. The counting argument gives the 
parameter from $\tu{H}^{n-2,1,0,1}$, so $n_{\calO} ^{\chi}=1$. 

\item[(4)] $\tu{G}=\tu{Spin}(2n-5,5)$. The counting argument gives the 
parameters from $\tu{H}^{n-3,2,0,1}$,  $\tu{H}^{n-5,0,2,1} $, so $n_{\calO} ^{\chi}=2$. 

\end{itemize}

\end{example}

\begin{lemma} \label{gchar}
Given $\la$ in (\ref{infchar-orbit}), let 
$$\Pi_{g, \la}(Z(\tu{G})) = \{\chi \in \Pi _g (Z( \tu{G} ) )  \mid  \exists \  \pi \in \calU_{\tu{G}} (\calO_{\bbC}, \la) \text{ s.t. } \ \pi | _{Z(\tu{G})} =\chi \}. $$ Then we have 

\begin{itemize}
\item[(a)] in Case 1, $|\Pi_{g, \la}(Z(\tu{G})) |  =4$;
\item[(b)] in Case 2 and 3, $|\Pi_{g, \la}(Z(\tu{G})) |  =2$;
\item[(c)] in Case 4, $|\Pi_{g, \la}(Z(\tu{G})) |  =2$;
{\item[(d)]  In Case 5 and 6 with  $\tu{G}=\tu{Spin}(2p+1, 2p-1)$, $|\Pi_{g, \la}(Z(\tu{G})) |  =2$;
\item[(e)] In Case 5 and 6 with  $\tu{G}=\tu{Spin}(2p+1, 2q+1)$, $q<p-1$, $|\Pi_{g, \la}(Z(\tu{G})) |  =1$;}
\item[(f)] in Case 7 and 8, $|\Pi_{g, \la}(Z(\tu{G})) |  =1$.

\end{itemize}
\end{lemma}

Below is the main theorem of the section and it follows from 
Lemma \ref{fix-char-no} and \ref{gchar} since 
$$| \calU _{\tu{G}} (\calO_c , \la)|  =  \sum \limits_{\chi\in \Pi_{g, \la}(Z(\tu{G})) } n_{\calO} ^{\chi}.$$

\begin{theorem} 
Let $n_{\calO} : = | \calU _{\tu{G}} (\calO_c , \la) |$ be the number of unipotent  representations of $\tu{G}$ attached to $\calO_c$ and $\la$.  Then 
\begin{description}
\item[Case 1]  $n_{\calO}=16$;
\item[Case 2, 3]  $n_{\calO}=6$;
\item[Case 4]  $n_{\calO}=4$;
{\item[Case 5, 6]  When  $\tu{G}=\tu{Spin}(2p+1, 2p-1)$ $n_{\calO}=2$;
\item[Case 5, 6]  When  $\tu{G}=\tu{Spin}(2p+1, 2q+1)$, $q<p-1$, $n_{\calO}=1$;}
\item[Case 7, 8]  $n_{\calO}=1$.
\end{description}
\end{theorem}


\bigskip

\section{A Construction}
\subsection{Littlewood Rule} $F_{\fk h}(\la)$ will denote the finite
dimensional representation with highest weight $\la$ of the  Lie
algebra $\fk h$. 

We will use the following result, a generalization of the Littlewood
rule as in \cite{EW}. Let $(V, \langle\ ,\ \rangle)$ be an orthogonal
space of dimension $m$ with positive definite inner product, and let
$O(m)$ be the corresponding (compact) orthogonal group. A
representation is parametrized by a partition/tableau such that there
are at most $m$ rows, and the sum of the lengths of the first two
columns is $\le m$. An irreducible  representation was parametrized
earlier by its highest weight as $(\tau_1,\dots ,\tau_{[m/2]},\ep)$
with $\ep=\pm 1.$ The corresponding partition is $(\tau_1,\dots
,\tau_{[m/2]},\underbrace{1,\dots ,1}_{m-2(1-\ep)},0,\dots ,0).$ 
Let $Sp(2n,\bb R)$ be the symplectic group of rank $n.$ Fix a Cartan
decomposition $\fk g=\fk k+\fk p=\fk k +\fk p^++\fk p^-,$ in the
standard coordinates. The oscillator correspondence matches $W(\tau)$
with an irreducible highest weight module $\Theta(W)=E_\tau$ as
follows. For a partition $(\la_1,\dots ,\la_m)$, define 
$$
\la^\sharp=(-m/2-\tau_n,\dots ,-m/2-\tau_1),
$$
where the tableau of $\tau$ with at most $m$ parts has possibly been
padded with $0's$ to make $n$ parts. Then $\Theta(W)=E_\tau$ has
highest weight $\la^\sharp.$ 
 The length $\ell(\la)$ of a tableau $(\la_1,\dots ,\la_m)$ is defined
 to be the number of nonzero $\la_i.$   
\begin{prop}\label{p:lwood}
$$
[W_{O(m)}(\tau)\ :\ F(\la)]=[F(\la^\sharp)\ : E_{\tau}] 
$$
Under the assumption $\ell(\la)\le (m+1)/2,$ let $n=m/2$ if $m$ is
even, $n=(m+1)/2$ if $m$ is odd. Then $E_\tau$ is an irreducible
generalized Verma module for $\fk q=\fk k+\fk p^+,$  and the formula
can be written as 
$$
[W(\tau)\ :\ F(\la)]=[F(\la)\ : \ W_{gl(n)}(\tau)\otimes S(\fk p^+)],
$$
where  
$$
S(\fk p^+)=\sum F(2m_1,\dots ,2m_n),\qquad 2m_j\in2\bb Z.
$$
\end{prop}
\begin{proof}
The first formula is standard for the $\Theta-$correspondence; see \cite{EW} for further explanations, references to the original result, and generalizations. When $\ell(\la)\le m/2,$ this is the classical Littlewood rule; we use $n=\ell(\la)\le m/2$. When $\ell(\la)=(m+1)/2,$ necessarily $m$ is odd, and we use $n=(m+1)/2.$ We need to show that the  generalized Verma module for $\fk k+\fk p^+$ with highest weight $\la^\sharp$ is irreducible. The infinitesimal character is 
$$\la^\sharp +\rho=(1/2,-1/2-\tau_1,\dots ,-m/2+1-\tau_{n-1}).
$$
The first term is greater than 0, the rest are negative. The only possible factor would be the highest weight module with weight
$$
(-1/2,-1/2-\tau_1,\dots ,-m/2+1-\tau_{n-1}).
$$
But this weight does not differ from $\la^\sharp$ by an element in the root lattice, so the generalized Verma  module is irreducible.
\end{proof}

 We treat the case $a=2p=2k+2,\ b=2q=2k+2+2r_-$ with
$r_-\ge 0$ {(Cases 1 and 3 in Section \ref{ss:1})} in detail, and note the necessary modifications for
$a=2p+1,\ b=2q-1$ with $r_-\ge0$ {(Cases 4, 6, 8 in Section \ref{ss:1}). }

Let $\fk q=\fk l +\fk u$ be the $\theta$-stable parabolic subalgebra determined by $\xi=(\underbrace{1,\dots ,1}_{p}\bigb 0,\dots ,0)$. 
The Levi component is  $\fk l=\fk{gl}(p)\times
\fk{so}(b).$ The real forms of the factors are $\fk{u}(p)\times \fk{so}(b)$ for
$a=2p,$ and $\fk{u}(p)\times \fk{so}(1,b-1)$ for $a=2p+1$. The nilradical is
$\fk u=\fk u_1+\fk u_2$ with $\fk u_2\subset\fk k$ in all
cases. Furthermore
$$
\fk u_1\cap \fk s=
\begin{cases}
  Span\left\{ X(\ep_i\pm\ep_{p+j})\right\}_{1\le i\le p, 1\le j\le q}&
  \text{ for } a=2p,\\
Span\left\{ X(\ep_i\pm\ep_{p+j}), X(\ep_i)\right\}_{1\le i\le p, 1\le
  j\le q-1} &\text{ for } a=2p+1.
\end{cases}
$$

\medskip
We consider 
\begin{equation}
\label{eq:one}
M(\mu)=U(\fk g)\otimes_{U(\ovl {\fk q} )} [F(\mu_L\bigb \mu_R)],
\end{equation}
the generalized Verma module where 
$F(\mu_L|\mu_R):=F_{gl(p)}(\mu_L)\boxtimes F_{so(b)}(\mu_R)$ is the
module of $gl(p)\times so(b)$ with highest weights $\mu_L$ for $gl(p)$ and
$\mu_R$ for $\fk{so}(b)$.
{We use the standard positive systems, so 
$$
\rho := \rho(\fk{so}(a+b))=
\begin{cases}
(-q,-q-1,\dots ,-p-q+1\bigb q-1,\dots ,0) &\text{ if } a=2p,\\
(-q,  -q-1, \dots ,-p-q+1\bigb q-1,\dots ,1) &\text{ if } a=2p+1.
\end{cases}
$$ 
For $a=2p+1$, this is $\rho$ restricted to the compact part of the
fundamental Cartan subalgebra; we drop a $0$ from the second set of
coordinates.  

\medskip
The infinitesimal character of $M(\mu)$ is as before,
$$
{
\mu+\rho\simeq (p-1/2, p-3/2, \dots ,1/2\bigb q-1,\dots ,1,0)}.
$$ 
The possible $\mu=(\mu_L\bigb\mu_R)$ are given by the equations
$$
\begin{aligned}
&(a_1,\dots ,a_p\bigb b_1,\dots,b_q)+(-q, -q-1, \dots ,-p-q+1\bigb q-1,\dots ,0)\\&=
(\mp 1/2,-3/2,\dots ,-p+1/2\bigb q-1,\dots ,0).
\end{aligned}
$$
When $2k+2<p+q,$ and there are two more for $\wti{Spin}(2p,2p)$:
$$
(a_1,\dots ,a_p\bigb b_1,\dots,b_p)+(-p, -p-1, \dots ,-2p+1\bigb p-1,\dots ,0)=(0,\dots ,-p+1\bigb p-1/2,\dots ,\pm 1/2).  
$$
So there are four $\mu,$ 
\begin{enumerate}
\item[(i)] $F(\mu_L\bigb\mu_R)=F(q-1/2,\dots ,q-1/2\bigb 0,\dots ,0)$.
\item[(ii)] $F(\mu_L\bigb\mu_R)=F(q+1/2,q-1/2,\dots ,q-1/2\bigb 0,\dots ,0).$
\item[(iii)] $F(\mu_L\bigb\mu_R)=F(p,\dots ,p\bigb 1/2,\dots ,1/2)$.
\item[(iv)] $F(\mu_L\bigb\mu_R)=F(p,\dots ,p\bigb 1/2,\dots ,1/2,-1/2)$.
\end{enumerate}
We call them Cases (i)--(iv). Again, Cases (iii) and (iv) only occur for
$k+1=p=q,$ \ie $\wti{Spin}(2p,2p)$. 
For  $\wti{Spin}(2p+1,2q-1),$ only Cases (i) and (ii) occur, and one $0$ is dropped from the coordinates.

\medskip
$M(\mu)$ is reducible, and the
irreducible quotient $L(\mu)$ has associated variety $\calO _{\bbC}=[3\ 2^{2k}\ 1^{2n-4k-3}].$
Its character is 
$$
L(\mu)=\sum_{w\in W(D_p^+)}
(-1)^{\ell(w)}M(w\cdot\mu)\qquad\text{ where }\quad  w\cdot\mu=w(\mu+\rho)-\rho.
$$ 
As a $\fk k$-module,
$\disp{
M(\mu)=U(\fk k)\otimes_{U(\ovl{\fk q}\cap\fk k)} [S(\fk u_1\cap \fk s)\otimes
({ F(\mu_L\bigb \mu _R)  )] }}$.
Embed $\fk{so}(b)\subset \fk{gl}(b)$ in case $b=2q,$ and $\fk{so}(b-1)\subset
\fk{gl}(b-1)$ in case $b=2q-1.$ Write $\fk{gl}(b-\ep)$ with $\ep=0$ if $b=2q$ 
and $\ep=1$ if $b=2q-1.$
\begin{lemma} As an $\fk l\cap\fk k$-module,
$$
S^m(\fk u_1\cap \fk s)=
\sum F_{\fk{gl}(p)}(m_1,\dots ,m_{p})\boxtimes
F_{\fk{gl}({b-\ep})}(m_1,\dots ,m_{p},0,\dots ,0) 
$$
with $\sum m_i=m.$ In all cases $p=k+1\le b-\ep=k+1-\ep + r_-.$
\end{lemma} 

\begin{proof}
$\fk u_1\cap \fk s$ has highest weight $(1,0,\dots ,0\bigb 1,0,\dots ,0).$ The
representation on the second factor is the standard one for $\fk{so}(b-\ep).$
It is the restriction of the standard representation with highest
weight $(1,0,\dots ,0)$ of $\fk{so}(b-\ep)\subset \fk{gl}(b-\ep)$. The claim follows.
\end{proof}

\begin{definition} 
Let $W(\beta)$ and $F(\mu_R)$ be representations of $\fk{so}(b-\ep).$ Then 
$$
\big(W_{\fk{so}(b-\ep)}(\beta)\otimes F_{ \fk{so}(b-\ep)}(\mu_R)\big)_{\fk{gl}(p)}
$$  
means the sum of (with multiplicity) of the composition factors of the
tensor product whose highest weight has at most $p$ nonzero
coordinates replaced  by the irreducible representations of $\fk{gl}(p)$
with the same highest weight.  
\end{definition}
This is related to the Littlewood rule for restriction from 
$\fk{gl}(b-\ep)$ to $\fk{so}(b-\ep).$ 
\begin{prop}
  \label{p:1}
Let $S(\fk p^+)=\sum V(2m_1,\dots ,2m_p)$. Then 
\begin{equation}
  \begin{aligned}
&[V(\delta)\boxtimes W(\beta) :\ S(\fk u_1\cap \fk s)\otimes F(\mu_L\bigb\mu_R)]_{ \fk{gl}(p)\times \fk{so}(b-\ep)}\\
=&[\big(W_{\fk{so}(b-\ep)}(\beta)\otimes F_{\fk{so}(b-\ep)}(\mu_R)^*\big)_{ \fk{gl}(p)}\otimes S( {\fk p^+})\otimes 
F_{\fk{gl}p)}(\mu_L)\ :\ V(\delta)]_{\fk{gl}(p)}.   
  \end{aligned}
\end{equation}
In particular the multiplicity is 0 unless $\beta_{p+1}=\dots =0.$
\end{prop}
\begin{proof} We abbreviate $(a_1,\dots ,a_p,0,\dots ,0)$ as $(a,0).$ 
\begin{equation}\label{eq:mult}
\begin{aligned} 
&{[V(\delta) \boxtimes W(\beta) :\ S(\fk u _1\cap \fk s)\otimes F(\mu _L
  \bigb \mu _R) ]}\\ 
  & = {\sum _a}[V(\delta)\boxtimes W(\beta) :\
  (F_{\fk{gl}(p)}(a)\boxtimes F_{\fk{gl}(b-\ep) }(a,0)
  \bigb_{\fk{so}(b-\ep)} )\otimes F(\mu_L\bigb\mu_R)]\\ 
&= {\sum_a}{ [ V(\delta)\boxtimes W(\beta) :\ (
    F_{\fk{gl}(p)}(a)\otimes F_{\fk{gl}(p)}(\mu _L) )\boxtimes
    (F_{\fk{gl(b-\ep)}}(a, 0) \bigb_{\fk{so}(b-\ep)}\otimes
    F_{\fk{so}(b-\ep)}(\mu_R) )] } \\ 
&= {\sum_a}{  [V(\delta)  :\ F_{\fk{gl}(p)}(a)\otimes
  F_{\fk{gl} _p}(\mu_L)] \cdot [W(\beta) :\  F_{\fk{gl}(b-\ep)}(a,0)
  \bigb_{\fk{so}(b-\ep)} \otimes F_{\fk{so}(b-\ep)}(\mu _R)] } \\ 
&={\sum _a}[V(\delta)\otimes F_{\fk{gl}(p)}(\mu_L)^* :\ F_{\fk{gl}(p)}(a)] \cdot
[W(\beta)\otimes F_{\fk{so}(b-\ep)}(\mu_R)^* : \  F_{\fk{gl}(b-\ep)}(a,0)\bigb_{\fk{so}(b-\ep)}] .
\end{aligned}
\end{equation}
Let 
\begin{equation}\label{eq:mult0.5}
F_{\fk{gl}(p)}(a)\otimes F_{\fk{gl}(p)}(\mu_L)=\sum _{\gamma}
c_{a,\mu_L}^\gamma F_{\fk{gl}(p)}(\gamma).  
\end{equation}
Note that
\begin{equation}
  \label{eq:reciprocal}
c_{a,\mu_L}^{\gamma}=[F_{\fk{gl}(p)}(a)\otimes F_{\fk{gl}(p)}(\mu_L)\ :\
F_{\fk{gl}(p)}(\gamma)]=[F_{\fk{gl}(p)}(a)\ :\ 
F_{\fk{gl}(p)}(\gamma)\otimes F_{\fk{gl}(p)}(\mu_L)^*].
\end{equation}
So
\begin{equation}
  \label{eq:recipr1}
F_{\fk{gl}(p)}(\gamma)\otimes F_{\fk{gl}(p)}(\mu_L)^*=
\sum_{a}c_{a,\mu_L}^{\gamma}F_{\fk{gl}(p)}(a).
\end{equation}
Assume that $\mu_L$ is such that its coordinates are all nonnegative;
this is the case for $w(\mu+\rho)-\rho$ because the first $p$
coordinates of
$w({\mu}+\rho)$ are of the form 
$$
(r_1+1/2,\dots
,r_\ell+1/2,-s_1-1/2,\dots ,-s_{p-\ell}-1/2)
$$ 
with $0\le r_i,s_j\le p,$ and $-\rho=(p+q-1,\dots, q\bigb {-q+1,
  -q+2,\dots,-1, 0)}.$ So the factors in the  tensor product
$F_{\fk{gl}(p)}(a)\otimes F_{\fk{gl}(p)}(\mu)$ have highest weights
$\gamma$ with nonnegative entries as well. Thus  {by
  (\ref{eq:mult}) and  (\ref{eq:mult0.5}), $[V(\delta) \boxtimes
  W(\beta) :\ S(\fk u _1\cap \fk k) \otimes F(\mu _L \bigb \mu _R) ]$ becomes} 
\begin{equation}
  \label{eq:mult1}
  \begin{aligned}
{\sum _a} \sum_{\gamma} c_{a,\mu_L}^{\gamma} [V(\delta)\ :\ F_{\fk{gl}(p)}(\gamma)]\cdot [W(\beta)\otimes
F_{\fk{so}(b-\ep)}(\mu_R)^*\ :\ F_{\fk{gl}(b-\ep)}(a,0)\bigb_{\fk{so}(b-\ep)}].
  \end{aligned}
\end{equation}
Proposition \ref{p:lwood} implies
\begin{equation}
  \label{eq:little}
  [F_{\fk{so}(b-\ep)}(\tau)\ :\ F_{\fk{gl}(b-\ep)}(a,0)\bigb_{\fk{so}(b-\ep)}]= 
[F_{\fk{gl}(p)}(\tau)\otimes S({\fk p}^+)\ :\  F_{\fk{gl}(p)}(a)],
\end{equation}
where the right hand side is a $\fk{gl}(p)$-multiplicity, and {$S({\fk p}^+ )=\sum
V(2m_1,\dots ,{2m_{p}})$ with $m_i\in \bb N$  as a $\fk{gl}(p)$-module.
In particular $\tau$ can have at most $p$ nonzero coordinates, and the right hand side in (\ref{eq:little}) is a $\fk{gl}(p)$-multiplicity. 

The sum is over the $a$ such that $\gamma=\delta$  occurs in $V(a).$ By 
(\ref{eq:recipr1}), 
\begin{equation}
  \label{eq:mult2}
\sum c_{a,\mu_L}^{\delta}F_{\fk{gl}(p)}(a)=V(\delta)\otimes F_{\fk{gl}(p)}(\mu_L)^*.
\end{equation}
So the multiplicity is
\begin{equation}
  \label{eq:mult4}
  \begin{aligned}
&[V(\delta)\boxtimes W(\beta) :\ S(\fk u_1\cap \fk s)\otimes F(\mu_L\bigb\mu_R)]= \\ &={
[ \big(W(\beta)\otimes F_{\fk{so}(b-\ep)}(\mu_R)^*\big)_{gl(p)}\otimes S({\fk p^+} )  : \ \sum_a c_{a, \mu _L, \delta} F_{\fk{gl}(p)}(a)]  }\\
&= [\big(W(\beta)\otimes F_{\fk{so}(b-\ep)}(\mu_R)^*\big)_{gl(p)}\otimes S({\fk p^+} )\ :\ V(\delta)\otimes F_{\fk{gl}(p)}(\mu_L)^*]  \\
&=[\big(W(\beta)\otimes F_{\fk{so}(b-\ep)}(\mu_R)^*\big)_{gl(p)}\otimes S({\fk p^+})\otimes F_{\fk{gl}(p)}(\mu_L)\ :\ V(\delta)].
  \end{aligned}
\end{equation}
}
\end{proof}

\subsubsection*{}
We compute the multiplicity 
\begin{equation}\label{eq:mult3}
[V(\delta) \boxtimes W(\beta): \ S( \fk u _1\cap\fk s) \otimes \sum _{w\in W( D _p ^+)}  \epsilon (w)F (w\cdot \mu )].
\end{equation}
There are four cases. Write $\calX(\mu):=S( \fk u _1\cap\fk s) \otimes \sum
_{w\in W( D _p ^+)}  \epsilon (w)F (w\cdot \mu ).$

\subsubsection*{Cases (i) and (ii)}
 {$\mu_R=Triv$ and 
$$
\mu_L=
\begin{cases}
(q-1/2, q-1/2,\dots, q-1/2) &\text{ in Case (i)},\\
(q+1/2,q-1/2,\dots  ,q-1/2) &\text{ in Case (ii)}.   
\end{cases}
$$

\begin{prop}\label{p:case12} \ 
$$
V(\delta)\boxtimes W(\beta)=(\al_1+q-1/2,\dots ,\al_p+q-1/2\bigb \beta_1,\dots ,\beta_p,0,\dots 0)
$$ 
occurs in $\calX(\mu)$, and with multiplicity 1 if and only if 
{
$$
\begin{cases}
\al_1\ge \beta_1\ge \al_2\ge \dots\ge\al_p\ge |\beta_p| &  \text{ when } a=b=2p,\\
\al_1\ge \beta_1\ge \al_2\ge \dots\ge\al_p\ge \beta_p\ge 0& \text{ otherwise,}
\end{cases}
$$
with $\alpha_i, \beta_j \in \bbZ$.}

{In Case (i), $(\al)-(\beta)$ is in the root lattice, in Case (ii), 
$(\al)-(\beta)-(1,0,\dots ,0)$ is in the root lattice}.


\end{prop}
\begin{proof}
{By Proposition \ref{p:1}, equation
  (\ref{eq:mult3}) becomes 
\begin{equation}  \label{eq:mult12}
  \begin{aligned}
 [V(\delta) \boxtimes W(\beta): \ S( \fk u _1\cap \fk s) \otimes \sum _{w\in W(D
   _p ^+)}  \epsilon (w)F(w\cdot \mu _L\bigb 0 )]\\ = 
[W(\beta)\otimes  S({\fk p^+})\otimes  \sum _{w\in W(D_p ^+)}
\epsilon (w) F(w\cdot \mu_L)\ :\ V(\delta)] 
\end{aligned}
\end{equation}
In all cases $W(\beta)$ is a representation of $\fk{gl}(p)$ where the highest weight has been padded by $0$'s to make a highest weight of $\fk{gl}(p)$.}

\medskip
We can do the
computation in a different Lie algebra, $\fk g'=\fk{sp}(2p).$  
Consider the parabolic subalgebra ${ \fk p' =\fk l ' + \fk u ' }$ 
corresponding to
$\xi'=(1,\dots ,1),$ {where $\fk l ' \cong \fk{gl}(p)$.} 
Let 
$$
M(\la')=U(\fk g')\otimes_{U(\ovl{ \fk p ' })}F(\la')
$$ 
be the generalized Verma module with $\la'$ such that
\begin{equation}\label{eq:la+rhoC}
\la'+\rho(C_p)=
{\begin{cases}
(-1/2, -3/2 ,\dots ,-p+1/2) & \text{in case (i)},\\
(1/2, -3/2 ,\dots, -p+1/2) & \text{in case (ii)}.
\end{cases}}
\end{equation}
Then
$\rho(C_p)=(-1,\dots ,-p)$ and  
\begin{equation}
{\la ' = \begin{cases}
(1/2,\dots, 1/2) & \text{in case (i)},\\
(3/2, 1/2, \dots, 1/2) &\text{in case (ii)}.
\end{cases}}
\end{equation}

\medskip
The quotient $L(\la')$ is one of the metaplectic representations with
{$\fk l' \cap \fk k'$}-structure 
\begin{equation}\label{k-str}
\begin{cases}
\sum _{m\in\bbN}F_{\fk{gl}(p) }(2m+1/2,1/2,\dots ,1/2) & \text{in case (i)}, \\
\sum_{m\in\bbN} F_{\fk{gl}(p)}(2m+3/2,1/2,\dots ,1/2) &\text{ in case (ii)}  
\end{cases}
\end{equation}
and character analogous to $L(\la):$ 
{ 
\begin{eqnarray}
L(\la') &=\sum \limits_{w\in W(D_p^+)} \ep (w) M(w\cdot\la ') =\sum
\limits_{w\in W(D_p ^+)} \ep (w) S(\fk u ')\otimes  F(w\cdot \la '). 
\end{eqnarray}
}

 The infinitesimal
character $\la'+\rho(C_p)=(\mp 1/2,-3/2\dots ,-p+1/2)$ is the same as
the first $p$ coordinates of the infinitesimal character of $M(\la).$
Furthermore, on the first $p$ coordinates,
\begin{equation}\label{eq:diff}
  \begin{aligned}
w\cdot\mu _L =&w(\mu_L +\rho(D_{n}))-\rho(D_{n})\\
=&w(\la'+\rho(C_p))-\rho(C_p)
{+(q-1,\dots ,q-1)} = w\cdot \la ' +(q-1,\dots, q-1).
  \end{aligned}
\end{equation}

{The multiplicity in (\ref{eq:mult12}) is therefore 
\begin{equation}\label{eq:mult4.0}
[W(\beta)\otimes  S( \fk u' )\otimes  \sum _{w\in W(D_p ^+)} \epsilon (w) F(w\cdot \la ')\ :\ V(\delta+ (-q+1,\dots ,-q+1 ))] _{\fk{gl}(p)}.
\end{equation}}

{The multiplicity in (\ref{eq:mult4.0}) becomes 
\begin{equation}  \label{eq:mult4.1}
  \begin{cases}
[W(\beta)\otimes \sum _m F_{\fk{gl}(p)} (2m +1/2,1/2,\dots
,1/2) \ :\ V(\delta +(-q+1)) ] \\ = [W(\beta) \otimes F_{\fk{gl}(p)} (2m,0,\dots,0) :V(\delta+(-q+1/2))] & \text{ in case
  (i)},\\ 
[W(\beta)\otimes \sum _m F_{\fk{gl}(p)}(2m+3/2,1/2,\dots
,1/2)\ :\ V(\delta) ] \\ = [W(\beta) \otimes F_{\fk{gl}(p)} (2m+1,0,\dots,0) :V(\delta+(-q+1/2))]  
& \text{ in case (ii)}.     
  \end{cases}
\end{equation}
}
{By the Littlewood-Richardson rule, 
\begin{equation}\label{LRrule2.0}
W(\beta)\otimes \sum _{k\in \bbN} F_{\fk{gl}(p)}(k, 0, \dots, 0) =\sum V(\beta_1+m_1, \dots, \beta_p +m_p),
\end{equation}
where the sum is taken over the set $\{  m_i\in \bbN \bigb \sum m_i =k, \ m_{i+1} \leq  \beta_i - \beta_{i+1} ,\ 1\leq i\leq p-1\}$.
Therefore, in case (i), }

\begin{equation}
  \label{eq:mult5}
  \begin{aligned}
    &[W(\beta)\otimes \sum _{m\in\bbN} F_{\fk{gl}(p)}(2m,0,\dots,0):V(\delta-q+1/2)]\\
=&\begin{cases}
  1 &\text{ if } \delta_1-q+1/2\ge \beta_1\ge \dots
  \ge\delta_p-q+1/2\ge \beta_p\ge 0,\ p<q\\
  1 &\text{ if } \delta_1-q+1/2\ge \beta_1\ge \dots
  \ge\delta_p-q+1/2\ge |\beta_p|,\ p=q\\
  0 &\text{ otherwise.}
\end{cases}    
  \end{aligned}
\end{equation}
When the multiplicity is 1, there is an additional restriction
that $(\delta)-(q-1/2)-(\beta)$ be in the root lattice.

The case $[W(\beta)\otimes \sum\limits _{m\in \bbN}
F_{\fk{gl}(p)}(2m+1,0,\dots,0):V(\delta-q+1/2)]$ is the same, but 
the multiplicity is 1 only when $(\delta)-(q-1/2)-(\beta)-(1,0,\dots ,0)$
is in the root lattice. {Recalling that the notation is $(\alpha)=(\delta)-(p-1/2)$, the assertions in the Proposition follow. }
\end{proof}
\subsubsection*{Cases (iii) and (iv)}
In these cases $\mu_R=(1/2,\dots ,1/2, {\pm 1/2})=Spin_{\pm},$ and
$\mu_L= (p,\dots ,p).$ 
The multiplicity of $V(\delta)\boxtimes W(\beta)$ {in Proposition (\ref{p:1})} is
$$
\sum \limits_{a} c_{a,\mu_L,\delta} [W(\beta)\otimes
F_{\fk{so}_{2p}}(\mu_R)^*\ :\ F_{\fk{gl}_{2p}}(a,0) \bigb_{\fk{so}
  _{2p}} ]. 
$$
Since the coordinates of the factors of $F_{\fk{gl}_{2p}}(a,0)$ are all
nonnegative integers, the same has to hold for $W(\beta)\otimes
F_{\fk{so}_{2p} }(\mu_R)^*.$ It follows that all the coordinates of $\beta$ are
half-integers, so strictly greater than 0 except for possibly the last
one. 
\begin{prop}\ 
  \label{p:2}
  \begin{itemize}
  \item In Case (iii), 
\begin{equation*}
[V(\delta)\otimes W(\beta) :\calX(\mu)]=
\begin{cases}
  1 &\text{ if } 
\delta_1-p+1/2\ge \beta_1\ge \dots \ge\delta_p-p+1/2\ge |\beta_p|,\\ 
&(\delta)-(p-1/2)-(\beta) \text{ in the root lattice,} \\ 
  0 &\text{ otherwise.}
\end{cases}
\end{equation*}
\item In Case (iv),      
 \begin{equation*}
[V(\delta)\otimes W(\beta) : \calX(\mu)]=
\begin{cases}
  1 &\text{ if } \delta_1-p+1/2\ge \beta_1\ge \dots
  \ge\delta_p-p+1/2\ge |\beta_p|,\\ 
&(\delta)-(p-1/2)-(\beta)-(1,0,\dots ,0) \text{ in the root lattice},\\
  0 &\text{ otherwise.}
\end{cases}
\end{equation*}
  \end{itemize}
In these cases, the coordinates of $\delta$ are in $\bb Z,$ the
coordinates of  $\beta$ in $(\bbZ
+1/2)$.
\end{prop}
\begin{proof}
We will use the following two lemmas and the Littlewood rule.
\begin{lemma} Assume $F_{\fk{so}_{2p}}(\beta)$ is a module for $\fk{so}(2p).$ Then
\label{l:spinso}
  $$
F_{\fk{so}_{2p}}(\beta)\otimes Spin=\sum F_{\fk{so}_{2p}}(\beta_i+\ep_i/2),\qquad \ep_i=\pm 1
$$
and the sum is over the $\ep_i$ such that $(\beta_i+\ep_i/2)$ is a
highest weight.  
\end{lemma}
\begin{proof}
  Omitted.
\end{proof}
\begin{lemma} Assume $F_{\fk{gl}(p) }(\beta)$ is a module for $\fk{gl}(p).$ Then 
  \label{l:spingl}
$$
F_{\fk{gl}(p)}(\beta)\otimes F_{\fk{gl}(p)}(1/2,\dots ,1/2,\underset{k}{\underbrace{-1/2,\dots ,-1/2}})=
\sum F_{\fk{gl}(p)}(\beta_i+\ep_i/2), \qquad \ep_i=\pm 1,
$$
with exactly $k$ $\ep_i$'s equal to $-1,$ and  the sum is over the $\ep_i$ such
that $(\beta_i+\ep_i/2)$ is a highest weight.
\end{lemma}
\begin{proof}
  Omitted.
\end{proof}
Consider
$$
\sum _{a} c_{a,\mu_L,\delta}[W (\beta)\otimes Spin_{\pm}
^*\bigb_{\fk{gl} }\otimes S({\fk p^+ })\
:\ F_{\fk{gl }}(a)].
$$
As before,  this is
$$
[W(\beta) \otimes (Spin _{\pm})^* \otimes S(\fk p^+) \otimes F_{\fk{gl}(p)}(\mu _L) : V(\delta)] _{\fk{gl} _p}.
$$

Replace $F_{\fk{gl}(p)}(\mu_L)$ by   
$\sum \limits_{w\in W(D_p^+ )}\ep(w) F_{\fk{gl}(p)}(w(\mu_L+\rho(D_{2p}))-\rho(D_{2p})):$ 

\begin{eqnarray}\label{multi:case3-4}
[W(\beta)\otimes (Spin _{\pm} )^* \otimes S(\fk p^+)\otimes  \sum _{w\in W(D_p^+)} \ep(w)F_{\fk{gl}(p)}( w\cdot \mu _L)\ :\ V(\delta) ].
\end{eqnarray}

Write $Ch(W)$ for the character of  $W$. 

\begin{lemma}
\begin{eqnarray*}
Ch \left [ (Spin _{\pm} )^* \otimes S(\fk p^+)\otimes  \sum _{w\in W(D_p^+)} \ep(w)F_{\fk{gl}(p)}( w\cdot \mu _L) \right ] =Ch(L(\la '))\otimes e^{(p-1,\dots,p-1)},
\end{eqnarray*}
where $L(\la')$ is the metaplectic representation introduced in the last section.
\end{lemma}
\begin{proof}
\begin{eqnarray}
 Ch(S(\fk p^+)) &= \frac{1}{\prod _{\al \in \triangle  (\fk p^+)} (1-e^{\al}) } = 
(-1)^q \frac{  e^{-\rho (\triangle (\fk p^+))} }{  \prod  \limits_{\al
    \in \triangle (\fk p^+)} (e^{\al /2}-e^{-\al/2}  )  }\notag \\
 & \\
&= (-1)^q \frac{e^{-\rho (\triangle (\fk p^+))}}{\Delta _{\fk p^+}} = (-1)^q \frac{e^{ ( \frac{-p+1}{2}  ,\dots, \frac{-p+1}{2} )}}{\Delta_{\fk p^+}},\notag
\end{eqnarray}
where $q=|\triangle (\fk p^+)|$.
The right hand side can be rewritten as 
\begin{eqnarray*}
RHS&= &e^{(p-1,\dots, p-1)}\cdot Ch[L(\la')]  \\
&=& e^{(p-1,\dots, p-1)}\cdot Ch \left [\sum \limits_{w\in W(D_p^+)}\ep (w) M(w\cdot\la ') \right ]\\ &=& e^{(p-1,\dots, p-1)} \cdot Ch [S(\fk p^+)]\cdot Ch \left [ \sum \limits_{w\in W(D_p ^+)} \ep (w) F_{\fk{gl}(p)}(w\cdot \la ') \right ].
\\  &=& \frac{(-1)^q}{\Delta _{\fk p^+}} e^{(\frac{p-1}{2} ,\dots, \frac{p-1}{2})} \sum _{w\in W(D_p ^+)} \ep (w) \frac{ \sum \limits_{x\in W(A_{p-1})} \ep(x) e^{x [w\cdot \la '+\rho (\fk{gl} _p) ]}    }{ \Delta _{\fk{gl}(p)}}\\
&=&  \frac{(-1)^q}{\Delta _{\fk p^+}}  e^{(\frac{p-1}{2} ,\dots, \frac{p-1}{2})}e^{(\frac{p+1}{2} , \dots, \frac{p+1}{2})} \sum _{w\in W(D_p ^+)} \ep (w) \frac{ \sum \limits_{x\in W(A_{p-1})} \ep(x) e^{x w (  \mp 1/2 , -\frac{3}{2},\dots, -p+1/2 ) }    }{ \Delta _{\fk{gl}(p)}} \\
&=& (-1)^q e^{(p,\dots,p)} \sum \limits _{v\in W(D_p)} \frac{\ep (v) e^{v  (  \mp 1/2 , -\frac{3}{2},\dots, -p+1/2 )} }{\Delta _{D_p} }\\
&=&(-1)^q e^{(p,\dots,p)}Ch[Spin _{\pm} ^*]
\end{eqnarray*}
On the other hand,
\begin{eqnarray*}
LHS&=& Ch[Spin_{\pm}^*] \cdot  (-1)^q \frac{e^{ ( \frac{-p+1}{2}  ,\dots, \frac{-p+1}{2} )}}{\Delta_{\fk p^+}}\cdot
\sum_{w\in W(D_p ^+)} \ep(w) \frac{   \sum \limits_{x\in W(A_{p-1})}  \ep (x) e^{x [w\cdot \mu_L +\rho (\fk{gl} _p)]}}{ \Delta_{\fk{gl} _p} }\\
&=& Ch[Spin_{\pm}^*] \cdot  (-1)^q \frac{e^{ ( \frac{-p+1}{2}  ,\dots, \frac{-p+1}{2} )}}{\Delta_{\fk p^+}}\cdot e^{ (\frac{3p-1}{2} ,\dots, \frac{3p-1}{2}  )} \frac{  \sum_{v\in W(D_p)}\ep (v) e^{v (0,-1,\dots,-p+1)}}{ \Delta_{ \fk{gl} _p}}\\
&=& (-1)^q\cdot Ch[Spin_{\pm}^*] \cdot e^{(p,\dots,p)}\cdot \frac{\Delta_{D_p}}{ \Delta_{\fk p^+} \Delta_{\fk{gl}(p)} } = (-1)^q\cdot Ch[Spin_{\pm}^*] \cdot e^{(p,\dots,p)}.
\end{eqnarray*}
The lemma follows.
\end{proof}
Thus  (\ref{multi:case3-4}) becomes 

\begin{eqnarray}
[W(\beta)\otimes L(\la'): V(\delta +(-p+1))]
\end{eqnarray}

By the same argument as in cases (i) and (ii),  we get

In case (iii),
\begin{equation}
  \label{eq:mult20}
  \begin{aligned}
     [W(\beta)\otimes &\sum _{k\in\bbN} F_{\fk{gl}(p)}(2k,0,\dots,0):V(\delta-p+1/2)]=\\
&=\begin{cases}
  1 &\text{ if } \delta_1-p+1/2\ge \beta_1\ge \dots \ge\delta_p-p+1/2\ge |\beta_p|,\\
  0 &\text{ otherwise.}
\end{cases}   
  \end{aligned}
\end{equation}

When the multiplicity is 1, there is an additional restriction
that $(\delta)-(p-1/2)-(\beta)$ be in the root lattice. 

In case (iv),      

\begin{equation}
  \label{eq:mult21}
  \begin{aligned}
    [W(\beta)\otimes 
\sum _{k\in\bbN} F_{\fk{gl}(p)}(2k +1,0,\dots,0):V(\delta-p+1/2)]\\ =
\begin{cases}
  1 &\text{ if } \delta_1-p+1/2\ge \beta_1\ge \dots \ge\delta_p-p+1/2\ge |\beta_p|,\\
  0 &\text{ otherwise.}
\end{cases}
\end{aligned}
\end{equation}

The multiplicity is 1 only when $(\delta)-(p-1/2)-(\beta)-(1,0,\dots ,0)$
is in the root lattice. {Note that $\delta\in \bbZ ^p$ and $\beta\in (\bbZ +1/2)^p$ in these cases}.

\end{proof}

\subsubsection{} We now apply the cohomological induction functor $\Pi$
from \cite{KV}. 
\begin{theorem}\ 
  \label{t:kstruct1234}
The $\wti{K}$-spectrum of the representations constructed are as follows.

\begin{description}
\item[Case 1] $a=b=2p$
$$
\begin{aligned}
(i) &&(a_p+1/2 , \dots, a_p+1/2\mid b_1,\dots, b_p) &  & \text{ with  }  \sum(a_i+b_j)\in 2\bbZ,\\
 (ii) && (a_p+1/2 , \dots, a_p+1/2\mid b_1,\dots, b_p) &  & \text{ with  }  \sum(a_i+b_j)\in 2\bbZ +1,\\
 (iii) &&(a_p+1, \dots, a_p+1\mid b_1+1/2,\dots, b_p+1/2) &  & \text{ with  }  \sum(a_i+b_j)\in 2\bbZ,\\
(iv) && (a_p+1 , \dots, a_p+1\mid b_1+1/2,\dots, b_p+1/2) &  & \text{ with  }  \sum(a_i+b_j)\in 2\bbZ +1,
\end{aligned}
$$
satisfying $a_1\ge b_1\ge \dots \ge a_p \ge |b_p|$, $a_i, b_j \in \bbZ$.
\medskip
\item[Case 3] $a=2p=2k+2$, $b=2q=2k+2+2r_-$ with $r_->0$
$$
\begin{aligned}
(i) &&(a_p+ r_- + 1/2 , \dots, a_p+r_- +1/2\mid b_1,\dots, b_p, 0,\dots, 0) &  & \text{ with  }  \sum(a_i+b_j)\in 2\bbZ,\\
 (ii) && (a_p+ r_- +1/2 , \dots, a_p+r_- +1/2\mid b_1,\dots, b_p , 0,\dots, 0) &  & \text{ with  }  \sum(a_i+b_j)\in 2\bbZ +1,
\end{aligned}
$$
satisfying $a_1\ge b_1\ge \dots \ge a_p \ge |b_p |$, $a_i, b_j \in \bbZ$.

\medskip
\item[Case 4] $a=2p+1=2q-1=b$
$$
\begin{aligned}
(i) &&(a_p+1/2 , \dots, a_p+1/2\mid b_1,\dots, b_p) &  & \text{ with  }  \sum(a_i+b_j)\in 2\bbZ,\\
 (ii) && (a_p+1/2 , \dots, a_p+1/2\mid b_1,\dots, b_p) &  & \text{ with  }  \sum(a_i+b_j)\in 2\bbZ +1, 
\end{aligned}
$$
satisfying $a_1\ge b_1\ge \dots \ge a_p \ge b_p \ge 0$, $a_i, b_j \in \bbZ$.
\medskip

\item[Case 6] $a=2p+1=2k+1$, $b=2q-1=2k+3+2r_-$
$$
\begin{aligned}
(i) &&(a_p+ r_- + 3/2 , \dots, a_p+r_- +3/2\mid b_1,\dots, b_p, 0,\dots, 0) &  & \text{ with  }  \sum(a_i+b_j)\in 2\bbZ,\\
 (ii) && (a_p+ r_- +3/2 , \dots, a_p+r_- +3/2\mid b_1,\dots, b_p, 0,\dots, 0) &  & \text{ with  }  \sum(a_i+b_j)\in 2\bbZ +1,
\end{aligned}
$$
satisfying $a_1\ge b_1\ge \dots \ge a_p \ge b_p\ge 0$, $a_i, b_j \in \bbZ$.
\medskip
\item[Case 8]  $a=2p+1=2k+3$, $b=2q-1=2k+1+2r_-$
$$
\begin{aligned}
(i) &&(a_p+ r_- -1/2 , \dots, a_p+r_- -1/2\mid b_1,\dots, b_p, 0,\dots, 0) &  & \text{ with  }  \sum(a_i+b_j)\in 2\bbZ,\\
 (ii) && (a_p+ r_- -1/2 , \dots, a_p+r_- -1/2\mid b_1,\dots, b_p, 0,\dots, 0) &  & \text{ with  }  \sum(a_i+b_j)\in 2\bbZ +1,
\end{aligned}
$$
satisfying $a_1\ge b_1\ge \dots \ge a_p \ge b_p \ge 0$, $a_i, b_j \in \bbZ$.

\end{description}

\end{theorem}

\begin{proof}
The main result we use is Corollary 4.160 in \cite{KV}, which computes
$\Pi^j$ of a generalized Verma module. The formulas in the previous
sections write $L(\mu)$ as a sum of generalized Verma modules for $\fk
k.$ 

{We treat Case 3 in detail.}
 The term $W(\beta)$ is already
$\fk{so}(2p)\times \fk{so}(2q)$-finite, only $V(\delta)$ is affected. Then 
\begin{equation}\label{sth}
\delta+\rho(D_p)=(\delta_1, \delta _2 -1, \dots, \delta _p-p+1)
\end{equation}
is formed of all {negative} terms. We need a $w$ that makes it dominant
for our choice of positive system $\fk b\subset\ovl{ \fk q}$;  $w$ is
the long Weyl group element. The highest weight is
$$
(p-1-\delta_p,\dots ,1-\delta _2, -\delta_1)-(0,\dots
,-p+1)=(p-1-\delta_p,\dots ,p-1-\delta_1) 
$$
Going to the more standard positive system for type $D_p$, the
  highest weight is 
\begin{equation}\label{sth1}
(\delta_1 -p+1, \dots, \delta_p -p+1 ).
\end{equation}

{In Cases (i) and (ii),  since $\delta \in (\bbZ +1/2)^p$, $\beta\in \bbZ ^p$,} we get that
the $K$-spectrum of $\Pi^j(L(\la))$ is $\sum (a_1+r_-+1/2,\dots
,a_p+r_-+1/2\bigb b_1,\dots ,b_p, 0,\dots, 0)$ where $a_i,b_j$ are integers and the multiplicity is 1
precisely when
$$
a_1\ge b_1\ge \dots \ge a_p\ge |b_p|.
$$
For case (i), the additional condition is that  $(a\bigb b)$
belongs to the root lattice of $D_{2p}$. For case (ii), the  additional condition is that $(a\bigb b) - (1,0\dots,0\bigb 0,\dots, 0)$ belongs to the root lattice of $D_{2p}.$
{Therefore, the additional parity conditions hold.}

\smallskip
{In the case $a=b=2p$ (\ie Case 1), we get the $\tu{K}$-spectra for case (i) and (ii)  by replacing $r_- =0$.}
In Cases (iii) and (iv),  since $\delta\in \bbZ ^p$, $\beta \in (\bbZ
  +1/2)^p$,  we get that 
the $K$-spectrum of $\Pi^j(L(\la))$ is $\sum (a_1+1,\dots
,a_p+1\bigb b_1+1/2,\dots ,b_p+1/2)$ where $a_i,b_j$ are integers and
the multiplicity is 1 precisely when
$$
a_1\ge b_1\ge \dots \ge a_p\ge |b_p|.
$$
For case (iii), the additional condition is that  $(a\bigb b)$
belongs to the root lattice of $D_{2p}$. For case (iv), the  additional
condition is that $(a\bigb b) - (1,0\dots,0\bigb 0,\dots, 0)$ belongs to
the root lattice of $D_{2p}.$ So we get the additional parity conditions.

{The argument for $a=2p=1,\ b=2q+1$ 
is essentially the same. In equation (\ref{sth}),  $\rho(D_p)$  is replaced by $\rho(B_p)$. 
The highest weight in (\ref{sth1}) becomes 
$$
(\delta_1 -p, \dots, \delta_p -p ).
$$
The $K$-spectrum of $\Pi^j(L(\la))$ is therefore 
$$
\sum (a_1+ q-p -1/2,\dots
,a_p+q-p -1/2\bigb b_1,\dots ,b_p, 0,\dots, 0),
$$
with $a_1\ge b_1\ge\dots\ge a_p\ge b_p\ge 0$ and $a_i, b_j\in \bbZ$.

In Case 4, $q-p=1$,  
in Case 6, $q-p =r_-+2$, and in Case 8, 
$q-p = r_- $. Consequently, the
$\tu{K}$-spectra  are as in statement of the theorem, with the
additional parity conditions due to the same reason as above.  
}
\end{proof}


\end{document}